\documentclass[]{amsart}
\usepackage[utf8]{inputenc}
\usepackage[T1]{fontenc}
\usepackage{amsmath,amssymb, amsthm, graphicx, tikz, mathtools, mathrsfs, comment, setspace, microtype,dsfont,float,hyperref,tikz-cd,enumitem,accents,xspace,bm,subcaption}
\usepackage[justification=centering]{caption}

\usepackage{biblatex}
\newtheorem{theorem}{Theorem}[section]
\newtheorem{corollary}[theorem]{Corollary}
\newtheorem{lemma}[theorem]{Lemma}
\newtheorem{proposition}[theorem]{Proposition}
\newtheorem{definition}[theorem]{Definition}
\newtheorem{heuristic}[theorem]{Heuristic}
\newtheorem{remark}[theorem]{Remark}
\newtheorem{algorithm}[theorem]{Algorithm}

\renewcommand{\ker}{\text{\textnormal{ker}}}
\newcommand{\R}{\mathbb{R}}
\newcommand{\C}{\mathbb{C}}
\newcommand{\Z}{\mathbb{Z}}

\newcommand{\Abs}[1]{\left\lvert #1 \right\rvert}

\newcommand{\sgn}{\text{\textnormal{sgn}}}
\newcommand{\brackets}[1]{\langle #1 \rangle}
\renewcommand{\bar}[1]{\overline{#1}}

\newcommand{\interior}[1]{\accentset{\circ}{#1}}

\renewcommand{\Re}{\text{\textnormal{Re}}}
\renewcommand{\Im}{\text{\textnormal{Im}}}
\newcommand{\norm}[1]{\left\lVert#1\right\rVert}

\newcommand{\tr}{\text{\textnormal{tr}}}

\newcommand{\Res}{\text{\textnormal{Res}}}

\newcommand\restr[2]{{
		\left.\kern-\nulldelimiterspace 
		#1 
		\vphantom{\big|} 
		\right|_{#2} 
}}

\begin{document}
	
\title[]{Quaternionic Green's Function and the Brown Measure of Atomic Operators}
\author{Max Sun Zhou}
\date{}

\begin{abstract}
	We analyze the Brown measure the non-normal operators $X = p + i q$, where $p$ and $q$ are Hermitian, freely independent, and have spectra consisting of finitely many atoms. We use the Quaternionic Green's function, an analogue of the operator-valued $R$-transform in the physics literature, to understand the support and the boundary of the Brown measure of $X$. We present heuristics for the boundary and support of the Brown measure in terms of the Quaternionic Green's function and verify they are true in the cases when the Brown measure of $X$ has been explicitly computed. In the general case, we show that the heuristic implies that the boundary of the Brown measure of $X$ is an algebraic curve, and provide an algorithm producing a polynomial defining this curve.
\end{abstract}	

\maketitle

\section{Introduction}
Let $M$ be a von Neumann algebra with a faithful, normal, tracial state $\tau$. The \textit{Brown measure} of an operator $X \in (M, \tau)$, introduced in \cite{BrownPaper}, is a complex Borel probability measure supported on the spectrum of $X$. It generalizes the spectral measure when $X$ is Hermitian and the empirical spectral distribution when $X$ is a random matrix. 

Recently, there has been much research in computing and analyzing the Brown measure for different families of operators  (see \cite{BelinschiPaper1}, \cite{belinschi2024brown}, \cite{KempDriverHall}, \cite{HoPaper}, \cite{HoZhongPaper}, and \cite{DemniHamdi} for some examples). The techniques used include free subordination, stochastic partial differential equations, and operator-valued free probability. For this paper, we use an analogue of the operator-valued $R$-transform found in the physics literature, the \textit{Quaternionic Green's function} (also referred to as the \textit{Quaternionic $R$-transform}, see \cite{PhysRevE.92.052111}, \cite{jarosz2004novel}, \cite{Jarosz_2006}, \cite{FeinbergZeePaper}, and \cite{JanikPaper}). The Quaternionic Green's function lends itself well to the operators we consider, as the algebraic nature of our operators means that we can perform explicit computations. 

The operators we consider in this paper are of the form $X = p + i q$, where $p$ and $q$ are Hermitian, freely independent, and have spectra consisting of finitely many atoms: 

\begin{definition}
	Let $\bm{X} = p + i q$, where $p, q \in (M, \tau)$ are Hermitian, freely independent, and their spectral measures are atomic, i.e.
	
	\begin{equation}
		\label{eqn:x=p+iq}
		\begin{aligned}
			\mu_p & = a_1 \delta_{\alpha_1} + \cdots + a_k \delta_{\alpha_k} \\
			\mu_q & = b_1 \delta_{\beta_1} + \cdots + b_l \delta_{\beta_l}  \, ,
		\end{aligned}
	\end{equation}
	where $\alpha_i, \beta_j \in \R$, $a_i, b_j \geq 0$, and $a_1 + \cdots + a_k = b_1 + \cdots b_l = 1$.
	
	When the $\alpha_i$ are distinct and $a_i > 0$, we say that $\bm{p}$ \textbf{has} $\bm{k}$ \textbf{atoms}. Similarly, when the $\beta_j$ are distinct and $b_i > 0$, we say that $\bm{q}$ \textbf{has} $\bm{l}$ \textbf{atoms}. 
\end{definition}

The Brown measure of $X$ is hypothesized to be well-approximated by the empirical spectral distributions of $X_n = P_n + i Q_n$ for large $n$, where $P_n, Q_n \in M_n(\C)$ are Hermitian and have the same spectra as $p$ and $q$. We provide some Mathematica visualizations of the empirical spectral distributions for some $X_n$:

\begin{figure}[H]
	\centering
	\begin{subfigure}{0.45 \textwidth}
		\centering
		\includegraphics[width = .9 \textwidth]{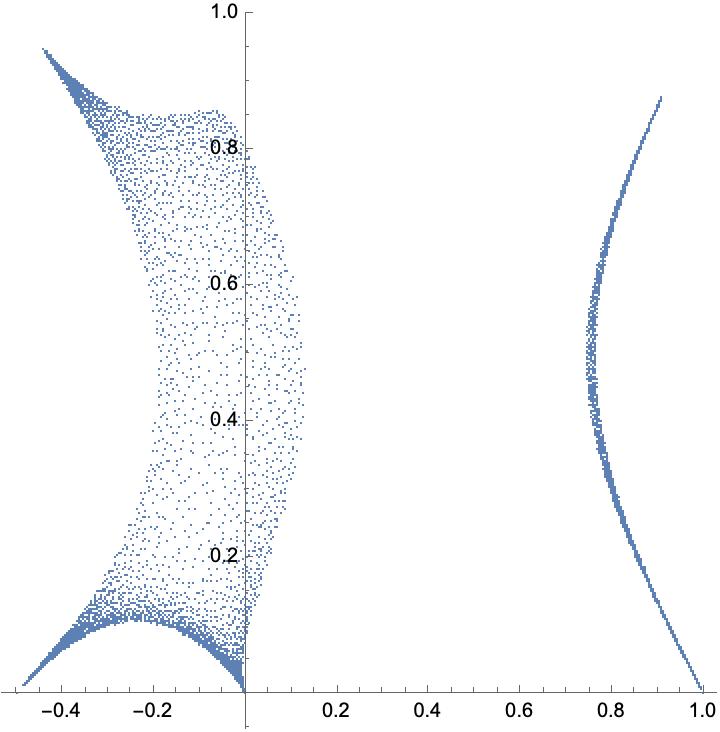}
		\caption{$\mu_{P_n} \approx (5/12) \delta_{-0.5} + (1/3) \delta_{0} + (1/4) \delta_{1}$ \\ $\mu_{Q_n} = (3/4) \delta_0 + (1/4) \delta_1$ \\ $n = 10000$}
	\end{subfigure}
	\hfill
	\begin{subfigure}{0.45 \textwidth}
		\centering
		\includegraphics[width = .9 \textwidth]{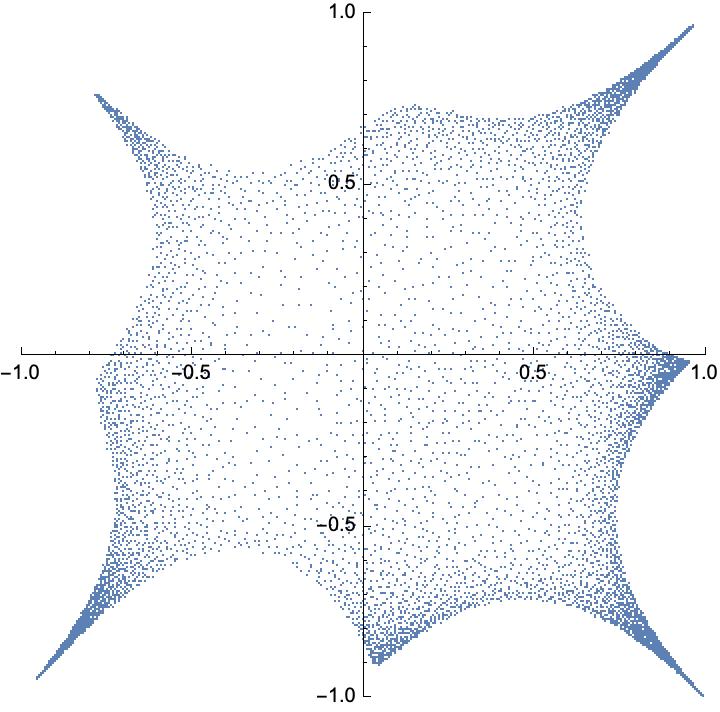}
\caption{$\mu_{P_n} = (1/4) \delta_{-1} + (1/5) \delta_{0} + (11/20) \delta_{1}$ \\ $\mu_{Q_n} = (1/2) \delta_{-1} + (1/4) \delta_{0} + (1/4) \delta_{1}$ \\ $n = 10000$}
	\end{subfigure}
	\caption{ESDs of $X_n = P_n + i Q_n$}
	\label{fig:thm:boundary_curve}
\end{figure}

We have previously explicitly computed the Brown measure in the case when $p$ and $q$ have two atoms in \cite{BrownMeasurePaper1} and observed these measures are supported on hyperbolas. In this case, we have also proven that the empirical spectral distributions of $X_n$ converge almost surely to the Brown measure of $X$ in \cite{ConvergencePaper}.

For this paper, we rigorously adapt the computations from \cite{jarosz2004novel} and \cite{Jarosz_2006} to provide a method to compute the Quaternionic Green's function for operators of the form $X = p + i q$, where $p$ and $q$ are Hermitian and freely independent. Then, we present the heuristics from \cite{jarosz2004novel} and \cite{Jarosz_2006} for the boundary and support of the Brown measure of $X = p + i q$ in terms of the Quaternionic Green's function. 

In the case when $p$ and $q$ have $2$ atoms, we explicitly compute the Quaternionic Green's function and show that the boundary and support heuristics agree with our previous results in \cite{BrownMeasurePaper1}. In the general case when $p$ and $q$ have finitely many atoms, we show that the boundary heuristic implies that the boundary of the Brown measure of $X$ is an algebraic curve and provide an algorithm that produces a polynomial defining this curve.

The rest of the paper is organized as follows. In Section \ref{sec:preliminaries}, we discuss the relationship between the Brown measure and Quaternionic Green's function and summarize our previous results about the Brown measure of $X = p + i q$ when $p$ and $q$ are freely independent . In Section \ref{sec:B_X_outline}, we rigorously adapt the computations from \cite{jarosz2004novel} and \cite{Jarosz_2006} that provide an outline for computing the Quaternionic Green's function of $X = p + i q$ for $p$ and $q$ Hermitian and freely independent. In Section \ref{sec:heuristics}, we provide the boundary and support heuristics that we will use for the rest of the paper. In Section \ref{sec:B_X_compute}, we explicitly compute the Quaternionic Green's function when $p$ and $q$ have $2$ atoms. In Section \ref{sec:boundary}, we analyze the boundary heuristic, verifying it in the case when $p$ and $q$ have $2$ atoms and in the general case deducing that the boundary is an algebraic curve by providing an algorithm that produces a polynomial defining the curve. In Section \ref{sec:support}, we verify the support heuristic when $p$ and $q$ have $2$ atoms with equal weights. 

\section{Preliminaries}
\label{sec:preliminaries}

\subsection{Brown measure}
In this subsection, we recall the definition of the Brown measure of an operator $X \in (M, \tau)$ and the relevant properties. These come from Brown's original paper \cite{BrownPaper} and also from \cite{SpeicherBook}, \cite{TaoBook}, \cite{HaagerupPaper}.

First, we recall the definition of the Fuglede-Kadison determinant, first introduced in \cite{KadisonPaper}: 

\begin{definition}
	Let $x \in (M, \tau)$. Let $\mu_{\Abs{x}}$ be the spectral measure of $\Abs{x} = (x^* x)^{1/2}$. Then, the \textbf{Fuglede-Kadison determinant} of $x$, $\Delta(x)$, is given by: 
	\begin{equation}
		\Delta(x) = \exp \left[ \int_{0}^{\infty} \log t \, d \mu_{\Abs{x}}(t) \right]   \,.
	\end{equation}
\end{definition}

When $x = X_n \in (M_n(\C), \frac{1}{n} \tr)$, then $\Delta(X_n) = \Abs{\det(X_n)}^{1/n}$. Thus, the Fuglede-Kadison determinant is a generalization of the normalized, positive determinant of a complex matrix. 

To construct the Brown measure for $x \in (M, \tau)$, let $f_\epsilon: \C \to \R$, where $\epsilon > 0$, be given by: 

\begin{equation}
	\label{eqn:f_epsilon}
	f_\epsilon(z) = \frac{1}{2} \tau \left[ \log ( (z - x)^*(z - x) + \epsilon  ) \right] =  \frac{1}{2} \int_{0}^{\infty} \log (t^2 + \epsilon) \, d \mu_{ \Abs{z - x} }(t)  \,.
\end{equation}

Then, $f_\epsilon(z)$ decreases $\log \Delta(z - x)$ as $\epsilon \to 0^+$. Computation shows that $f_\epsilon$ is twice-differentiable, with:

\begin{equation}
	\begin{aligned}
		\frac{\partial}{\partial z} f_{\epsilon}(z) & = \frac{1}{2} \tau\left[  (x_z^* x_z + \epsilon)^{-1} x_z^*\right]  \\
		\frac{\partial^2}{\partial \bar{z} \partial z} f_{\epsilon}(z) &= \frac{\epsilon}{2}  \tau \left[ (x_z x_z^* + \epsilon)^{-1}  (x_z^* x_z  + \epsilon)^{-1}\right] .
	\end{aligned}
\end{equation}

Hence, 

\begin{equation}
	\nabla^2  f_{\epsilon}(z)  = 4 \frac{\partial^2}{\partial \bar{z} \partial z} f_{\epsilon}(z) = 2 \epsilon \, \tau \left[ (x_z x_z^* + \epsilon)^{-1}  (x_z^* x_z  + \epsilon)^{-1}\right] \geq 0 \,.
\end{equation}

Thus, $f_\epsilon$ is subharmonic. Since $f_\epsilon(z)$ decreases to $\log \Delta(z - x)$ and $\log \Delta(z - x) > - \infty$ for $\Abs{z} > \norm{x}$, then $\log \Delta(z - x)$ is also subharmonic. Thus, $f_\epsilon(z)$ converges to $\log \Delta(z - x)$ in $L^1_\text{loc}(\C)$ (with respect to the Lebesgue measure), and hence also as distributions. As $\nabla^2 f_\epsilon(z)$ are positive distributions, then so is $\nabla^2 \log \Delta(z - x)$. Hence, this defines a measure, which is the Brown measure:
\begin{definition}
	Let $x \in (M, \tau)$. The \textbf{Brown measure} of $x$ is defined as:
	\begin{equation}
		\mu_x = \frac{1}{2 \pi}\nabla^2 \log \Delta(z - x) \,.
	\end{equation}
\end{definition} 

From the formulas for the derivatives of $f_\epsilon$,  we have the following distributional limits:
\begin{equation}
	\label{eqn:brown_measure_f_epsilon}
	\mu_x = \lim\limits_{\epsilon \to 0^+} \frac{1}{\pi} \frac{\partial}{\partial \bar{z}}  \tau\left[  (x_z^* x_z + \epsilon)^{-1} x_z^*\right]
	= \lim\limits_{\epsilon \to 0^+} \frac{\epsilon}{\pi}  \tau \left[ (x_z x_z^* + \epsilon)^{-1}  (x_z^* x_z  + \epsilon)^{-1}\right] .
\end{equation} 

The Brown measure of $x$ is a probability measure supported on the spectrum of $x$. When $x$ is normal, then the Brown measure is the spectral measure. When $x$ is a random matrix, then the Brown measure is the empirical spectral distribution.

Further, the Riesz representation for the subharmonic function $L(z) = \log \Abs{z - x}$ is given by: 
\begin{equation}
	\log \Delta(z - x) = \int_{\C}^{} \log \Abs{z - w} \, d \mu_x(w) \,,
\end{equation}
and the Brown measure is the unique complex Borel measure to satisfy this equation. 

\subsection{Brown measure of $X = p + i q$}
In this subsection, we recall our prior work from \cite{BrownMeasurePaper1} on the Brown measure of $X = p + i q$, where $p, q \in (M, \tau)$ are Hermitian, freely independent, and have spectra consisting of $2$ atoms. Specifically, let 

\begin{equation}
	\begin{aligned}
		\mu_p & = a \delta_\alpha + (1 - a) \delta_{\alpha'} \\
		\mu_q &= b \delta_\beta + (1 - b) \delta_{\beta'} \,,
	\end{aligned}
\end{equation}
where $a, b \in (0, 1)$, $\alpha, \alpha', \beta, \beta' \in \R$, $\alpha \neq \alpha'$, and $\beta \neq \beta'$.

The exact formulas describing the measure are not necessary, only some key properties which we restate. 

First, we use the definition of the hyperbola and rectangle associated with $X = p + i q$: 
\begin{definition}
	\label{def:hyperbola_rectangle}
	Let $\mathscr{A} = \alpha' - \alpha$ and $\mathscr{B} = \beta' - \beta$. 
	
	The \textbf{hyperbola associated with $X$} is 
	\begin{equation}
		H = \left\lbrace z = x + i y \in \C :	\left( x - \frac{\alpha + \alpha'}{2}  \right)^2 - \left(  y - \frac{\beta + \beta'}{2} \right)^2 = \frac{\mathscr{A}^2 - \mathscr{B}^2}{4}    \right\rbrace \,.
	\end{equation}
	The \textbf{rectangle associated with $X$} is
	\begin{equation}
		R =  \left\lbrace z = x + i y \in \C : x \in  [\alpha \wedge \alpha', \alpha \vee \alpha' ], y \in [\beta \wedge \beta', \beta \vee \beta'] \right\rbrace  \,.
	\end{equation}
\end{definition} 

We summarize the relevant parts of \cite[Theorem 6.1]{BrownMeasurePaper1} and \cite[Corollary 6.5]{BrownMeasurePaper1} describing the Brown measure of $X = p + i q$: 

\begin{theorem}
	\label{thm:brown_measure_p+iq}
	Let $p, q \in (M, \tau)$ be Hermitian, freely independent, and 
	
	\begin{equation}
		\begin{aligned}
			\mu_p & = a \delta_\alpha + (1 - a) \delta_{\alpha'} \\
			\mu_q &= b \delta_\beta + (1 - b) \delta_{\beta'} \,,
		\end{aligned}
	\end{equation}
	where $a, b \in (0, 1)$, $\alpha, \alpha', \beta, \beta' \in \R$, $\alpha \neq \alpha'$, and $\beta \neq \beta'$. 
	
	The Brown measure of $X$ is the convex combination of 4 atoms with another measure $\mu'$: 
	\begin{equation}
		\mu = \epsilon_{0 0} \delta_{ \alpha + i \beta } + \epsilon_{0 1} \delta_{\alpha + i \beta'} + \epsilon_{1 0} \delta_{\alpha'+ i \beta} + \epsilon_{1 1} \delta_{\alpha'+ i \beta'} + \epsilon \mu' \,,
	\end{equation}
	
	where 
	\begin{equation}
		\begin{aligned}
			\epsilon_{0 0} & = \max(0, a  + b - 1) \\
			\epsilon_{0 1} & = \max(0, a + (1 - b) - 1) \\
			\epsilon_{1 0} & = \max(0, (1 - a) + b - 1  )\\
			\epsilon_{1 1} & = \max(0, (1 - a) + (1 - b) - 1) \\
			\epsilon &= 1 - (\epsilon_{0 0} + \epsilon_{0 1} + \epsilon_{1 0} + \epsilon_{1 1}  ) \,.
		\end{aligned}
	\end{equation}	
	The Brown measure $\mu$ is supported on $H \cap R$. Further, $\mu'$ has density extending to all $4$ corners of the intersection of the hyperbola with the boundary of the rectangle if and only if $a = b = 1/2$. Hence, the support of $\mu$ is equal to $H \cap R$ if and only if $a = b = 1/2$.
\end{theorem}

A useful Lemma from \cite[Lemma 4.4]{BrownMeasurePaper1} about the hyperbola and rectangle associated with $X$ is: 

\begin{lemma}
	\label{lem:hyperbola_rectangle}
	Let $\alpha, \alpha', \beta, \beta' \in \R$, where $\alpha \neq \alpha'$ and $\beta \neq \beta'$. Let $\mathscr{A} = \alpha' - \alpha$ and $\mathscr{B} = \beta' - \beta$.
	
	Let 
	\begin{equation}
		\begin{aligned}
			H & = \left\lbrace z = x + i y \in \C :	\left( x - \frac{\alpha + \alpha'}{2}  \right)^2 - \left(  y - \frac{\beta + \beta'}{2} \right)^2 = \frac{\mathscr{A}^2 - \mathscr{B}^2}{4}    \right\rbrace \\
			R & =  \left\lbrace z = x + i y \in \C : x \in  [\alpha \wedge \alpha', \alpha \vee \alpha' ], y \in [\beta \wedge \beta', \beta \vee \beta'] \right\rbrace \,.
		\end{aligned}
	\end{equation}		
	The equation of $H$ is equivalent to: 
	\begin{equation}
		(x - \alpha)(x - \alpha') = (y - \beta)(y - \beta') \,.
	\end{equation}
	The equation of $H$ in coordinates 
	\begin{equation}
		\begin{aligned}
			x' & = x - \frac{\alpha + \alpha'}{2} \\
			y' &= y - \frac{\beta + \beta'}{2}
		\end{aligned}
	\end{equation}
	is 
	\begin{equation}
		\label{eqn:lem:hyperbola_rectangle}
		(x')^2 - \frac{\mathscr{A}^2  }{4} = (y')^2 - \frac{\mathscr{B}^2}{4} .
	\end{equation}
	It follows that for $(x, y) \in H$, 
	\begin{equation}
		(x, y) \in R \iff (\ref{eqn:lem:hyperbola_rectangle}) \leq 0 \iff x \in  [\alpha \wedge \alpha', \alpha \vee \alpha' ] \text{ or } y \in [\beta \wedge \beta', \beta \vee \beta'] \,.
	\end{equation}
	Similarly, 
	\begin{equation}
		(x, y) \in \interior{R} \iff (\ref{eqn:lem:hyperbola_rectangle}) < 0 \iff x \in  (\alpha \wedge \alpha', \alpha \vee \alpha') \text{ or } y \in (\beta \wedge \beta', \beta \vee \beta') \,.
	\end{equation}	
	Alternatively, the equation of the hyperbola is:
	\begin{equation}
		\Re \left( \left( z - \frac{\alpha + \alpha'}{2} - i \frac{\beta + \beta'}{2}  \right)^2 \right) =   \frac{\mathscr{A}^2 - \mathscr{B}^2}{4} \,.
	\end{equation}
	If $z \in H$, then $z \in R$ if and only if
	\begin{equation}
		\Abs{\Im \left( \left( z - \frac{\alpha + \alpha'}{2} - i \frac{\beta + \beta'}{2}  \right)^2\right)}  \leq \frac{\Abs{\mathscr{A} \mathscr{B}}}{2} \,.
	\end{equation}
\end{lemma}
\begin{proof}
	The equivalent equations for the hyperbola are straightforward to check. 
	
	The equivalences for the closed conditions follow from the following equivalences and the equation of the hyperbola in $x', y'$ coordinates
	\begin{equation}
		\begin{aligned}
			(x, y) \in R & \iff (x')^2 - \frac{\mathscr{A}^2}{4} \leq 0 \text{ and } (y')^2 - \frac{\mathscr{B}^2}{4} \leq 0 \\
		\end{aligned}
	\end{equation}
	\begin{equation}
		\begin{aligned}
			(x')^2 - \frac{\mathscr{A}^2  }{4} \leq 0  & \iff \Abs{x'} \leq \frac{\Abs{\mathscr{A}}}{2} & \iff x \in [\alpha \wedge \alpha', \alpha \vee \alpha' ]  \\
			(y')^2 - \frac{\mathscr{B}^2}{4} \leq 0 & \iff  \Abs{y'} \leq \frac{\Abs{\mathscr{B}}}{2} & \iff y \in [\beta \wedge \beta', \beta \vee \beta']  \,.
		\end{aligned}
	\end{equation}
	The equivalences for the open conditions follow from similar equivalences with the closed conditions replaced by open conditions.
	
	The last equation of the hyperbola follows from direct computation. For the inequality of the rectangle, observe that
	\begin{equation}
		\Im \left( \left( z - \frac{\alpha + \alpha'}{2} - i \frac{\beta + \beta'}{2}  \right)^2  \right)  = 2 x' y' \,.
	\end{equation} 
	In light of what was previously shown, 
	\begin{equation}
		x' y' \leq \frac{\Abs{\mathscr{A} \mathscr{B}}}{4} \Longrightarrow x' \leq \frac{\Abs{\mathscr{A}}}{2} \text{ or } y' \leq \frac{\Abs{\mathscr{B}}}{2} \Longrightarrow z \in R \,.
	\end{equation}
	Conversely, 
	\begin{equation}
		z \in R \Longrightarrow x' \leq \frac{\Abs{\mathscr{A}}}{2} \text{ and } y' \leq \frac{\Abs{\mathscr{B}}}{2} \Longrightarrow x' y' \leq \frac{\Abs{\mathscr{A} \mathscr{B}}}{4} \,.
	\end{equation}
\end{proof}

\subsection{Quaternions}
In this subsection, we introduce the notation for quaternions and state some basic properties. 

The \textit{quaternions} form a real 4-dimensional algebra generated by the elements $1$, $\bm{i}$, $\bm{j}$, $\bm{k}$, with the relations
\begin{equation}
	\bm{i}^2 = \bm{j}^2 = \bm{k}^2 = \bm{i} \bm{j} \bm{k} = -1 \,.
\end{equation}
A general quaternion $Q$ can be written as: 
\begin{equation}
	\label{eqn:quaternions_def}
	Q = x_0 + x_1 \bm{i} + x_2 \bm{j} + x_3 \bm{k} , \quad x_0, x_1, x_2, x_3 \in \R \,.
\end{equation} 
From physics, the Pauli matrices in $M_2(\C)$ are defined by: 
\begin{equation}
	\sigma_1 =
	\begin{pmatrix}
		0 & 1 \\
		1 & 0
	\end{pmatrix}, \quad
	\sigma_2 = 
	\begin{pmatrix}
		0 & -i \\
		i & 0
	\end{pmatrix}, \quad
	\sigma_3 = 
	\begin{pmatrix}
		1 & 0 \\
		0 & -1
	\end{pmatrix}.
\end{equation}
It is easy to see that $i \sigma_1$, $i \sigma_2$, and $i \sigma_3$ have the desired relations of the quaternionic generators $\bm{i}$, $\bm{j}$, and $\bm{k}$. The subalgebra of $M_2(\C)$ these elements and the identity matrix generate is $4$-dimensional over $\R$ and hence is a representation of the quaternions. Explicitly, we use the following definition for quaternions for the quaternions:
\begin{definition}
	The \textbf{quaternions} are the real subalgebra of $M_2(\C)$ consisting of the following matrices:
	\begin{equation}
		\label{eqn:quaternions_matrix}
		\mathbb{H} = \left\lbrace 
		Q = 
		\begin{pmatrix}
			A & i \bar{B} \\
			i B & \bar{A}
		\end{pmatrix} : A, B \in \C \right\rbrace .
	\end{equation}
	
	In terms of the coefficients $x_i$ in (\ref{eqn:quaternions_def}), 
	
	\begin{equation}
		\label{eqn:quaternions_coeff}
		A = x_0 + i x_3 \qquad B = x_1 + i x_2 , \quad x_0, x_1, x_2, x_3 \in \R \,.
	\end{equation}
	
	The coefficients $x_i$ are called the \textbf{real coefficients} of $Q$ and $A$, $B$ are the \textbf{complex coefficients} of $Q$.
	
	We will refer to the \textbf{real numbers} in the quaternions as the subset where $x_1 = x_2 = x_3 = 0$ (i.e. $A \in \R$ and $B = 0$). Similarly, we will refer to the \textbf{complex numbers} in the quaternions as the subset where $x_1 = x_2 = 0$ (i.e. $B = 0$).
\end{definition}

We highlight the fact that the quaternions are not a complex algebra, since the center of $\mathbb{H}$ is $\R$ and not $\C$. 

There are some relevant operations on quaternions: 

\begin{itemize}
	\item The \textbf{\textit{inverse}} of a quaternion is just the matrix inverse.
	\item The \textbf{\textit{conjugate}} of a quaternion $Q = x_0 + x_1 i + x_2 j + x_3 k$ is $\bar{Q} = x_0 - x_1 i  - x_2 j - x_3 k$. In the matrix representation of $Q$, this corresponds to the matrix $Q^*$.
	\item The \textbf{\textit{norm}} of a quaternion, $\Abs{Q}$, is defined by: 
	\begin{equation}
		\Abs{Q} = (Q \bar{Q})^{1/2} =  \sqrt{x_0^2 + x_1^2 + x_2^2 + x_3^2}
	\end{equation}
	and is equal to the square root of the determinant of $Q$. From the multiplicative property of the determinant, it easily follows that $\Abs{Q_1 Q_2} = \Abs{Q_1} \Abs{Q_2}$ for quaternions $Q_1$, $Q_2$.
	
	\item The norm on quaternions induces a metric, and when we speak of the \textbf{\textit{convergence}} of a sequence of quaternions $Q_k \to Q$, it is with respect to this metric.
\end{itemize}

Every quaternion $Q$ is diagonalizable with eigenvalues
\begin{equation}
	\label{eqn:quaternion_eigenvalues}
	\begin{aligned}
		g &=  x_0 + i \sqrt{x_1^2 + x_2^2 + x_3^2} \\
		\bar{g} &=  x_0 - i \sqrt{x_1^2 + x_2^2 + x_3^2} \, .
	\end{aligned}
\end{equation}

From the (\ref{eqn:quaternion_eigenvalues}), we observe the following facts: 

\begin{itemize}	
	\item The eigenvalues of $Q$ come in conjugate pairs.
	\item The eigenvalues of $Q$ are distinct if and only if $Q \not \in \R$.
	\item The eigenvalues of $Q$ are real if and only if $Q \in \R$.
	\item $\Abs{g}^2 =  \det Q = \Abs{Q}^2$ (so $\Abs{g} = \Abs{Q}$) and $\det(Q) = \Abs{A}^2 + \Abs{B}^2$.
\end{itemize}

In what follows, it is convenient to define $g^I$ to be the eigenvalue of $Q i$ with non-negative imaginary part. Using the notation in (\ref{eqn:quaternions_matrix}) and (\ref{eqn:quaternions_coeff}),
\begin{equation}
	Q i =
	\begin{pmatrix}
		i A & \bar{B} \\
		- B & - i \bar{A} 
	\end{pmatrix}
	=
	\begin{pmatrix}
		i A  & i \bar{(i B)} \\
		i (i B) & \bar{i A}
	\end{pmatrix} .
\end{equation}
Hence,
\begin{equation}
	\label{eqn:gI}
	g^I = - x_3 + i \sqrt{x_2^2 + x_1^2 + x_0^2} \,.
\end{equation}
\begin{remark}
	Given a quaternion $Q$, our convention (unless stated otherwise) will be to use $g$ (resp $g^I$) to denote the eigenvalue of $Q$ (resp. $Q i$) with non-negative imaginary part, i.e. as in (\ref{eqn:quaternion_eigenvalues}) and (\ref{eqn:gI}).
\end{remark}

It is easy to see directly from formulas that $\Abs{g} = \Abs{g^I}$. This also follows from the fact that $g^I$ is an eigenvalue of $Q i$ and $\Abs{i} = 1$.

We also record the following observations about when $g, g^I \in \R$ implies that $Q \in \R$ or $Q \in i \R$:

\begin{proposition}
	For $Q \in \mathbb{H}$, $g \in \R$ implies that $Q \in \R$ and $g^I \in \R$ implies that $Q \in i \R$. But, $g^I \in  i \R$ does not imply that $Q \in \R$. But, $Q \in \C$ and $g^I \in i \R$ does imply that $Q \in \R$.
\end{proposition}
\begin{proof}
	The first statements about $g \in \R$ and $g^I \in \R$ follow from the fact that $Q$ has real eigenvalues if and only if $Q \in \R$ (and similarly for $Q i$). If $Q = \bm{i}$, then $g^I \in i \R$ but $Q \not \in \R$. If $Q \in \C$, then $g^I$ is either $i g$ or $\bar{i g}$ depending on the sign of $\Im(i g)$, and so $g^I \in i \R$ implies $g \in \R$, and hence $Q \in \R$.
\end{proof}

The conventions $\Im(g) \geq 0$ and $\Im(g^I) \geq 0$ are useful for some general computations, but it has a drawback: if $Q \in \C$, then the eigenvalue $g$ may not be equal to $Q$: $g = Q$ when $\Im(Q) \geq 0$ and $g = \bar{Q}$ otherwise. Similarly, $Q i \in \C$, but we can only conclude that either $g^I = Q i$ or $g^I = \bar{Q i}$.

But, many of the complex-valued functions of $g$ we will consider arise from context as functions of a quaternion $Q$. This means that these functions only depend on the \textit{set} of eigenvalues of $Q$, i.e. these functions are invariant under changing $g$ and $\bar{g}$. We record the following Lemma for these situations:

\begin{lemma}
	\label{lem:conjugation}
	Let $Q \in \C$, $g$ (resp. $\bar{g}$) as in (\ref{eqn:quaternion_eigenvalues}) (resp. (\ref{eqn:gI}), and $f_1, f_2: \C \to \C$ where
	\begin{equation}
		f_1(g) = f_1(\bar{g}) \qquad f_2(g) = f_2(\bar{g}) \,.
	\end{equation}
	Then,
	\begin{equation}
		f_1(g) = f_1(Q) \qquad f_2(g^I) = f_2(Q i) \,.
	\end{equation}
	In particular, when we evaluate $f_1$ and $f_2$ at $g$ and $g^I$ respectively, we may assume that $g = Q$ and $g^I = i g$.
\end{lemma}
\begin{proof}
	Recall that $g = Q$ when $\Im(Q) \geq 0$ and $g = \bar{Q}$ otherwise. Since $f_1(g) = f_1(\bar{g})$, in either case $f_1(g) = f_1(Q)$. Similar logic shows that $f_2(g^I) = f_2(Q i)$. Hence, for purposes of evaluating $f_1(g)$ and $f_2(g^I)$, we may assume that $g = Q$ and $g^I = Q i = i g$.
\end{proof}

We also record the following fact about the convergence of quaternions to a real number: 

\begin{lemma}
	\label{lem:Q_real_convergence}
	Let $Q_k$ be a sequence of quaternions where the eigenvalues of $Q_k$ converge to some real number $g \in \C$. Then, $Q_k$ converges to $g \in \C$.
\end{lemma}
\begin{proof}
	Let $g_k, \bar{g_k}$ be the eigenvalues of $Q_k$, where
	\begin{equation}
		g_k = (x_0)_k + i \sqrt{(x_1)_k^2 + (x_2)_k^2 + (x_3)_k^2} \,.
	\end{equation}
	$g_k$ converges to a real $t \in \R$ if and only if $(x_1)_k \to 0$, $(x_2)_k \to 0$, $(x_3)_k \to 0$, and $(x_0)_k \to t$. This implies that $Q_k$ converges to $t$. 
\end{proof}

\subsection{Quaternionic Green's function}
In this subsection, we will define the Quaternionic Green's function of an arbitrary $X \in (M, \tau)$ in terms of the operator-valued Cauchy transform. It generalizes the Stieltjes transform by considering a non-Hermitian operator and being a function defined on quaternions. We will briefly define the relevant concepts in operator-valued free probability and apply them in our situation (summarized from (\cite{SpeicherBook}, Chapters 9, 10)  and (\cite{II_1book}, Chapter 9)). Then, we will show how to pass from the conventions in the mathematics literature to the conventions in the physics literature in our setup. 

First, we review the definitions of the conditional expectation and operator-valued probability spaces:

\begin{definition}
	Let $M$ be a von Neumann algebra and $B$ a von Neumann subalgebra. A \textbf{conditional expectation} from $E$ to $B$ is a linear map $E: M \to B$ that satisfies the following properties: 
	\begin{enumerate}
		\item If $x \geq 0$, then $E(x) \geq 0$. 
		\item $E(b) = b$ for $b \in B$.
		\item $E(b_1 x b_2) = b_1 E(x) b_2$ for $b_1, b_2 \in B$, $x \in M$.
	\end{enumerate}
	We will refer to $(M, E, B)$ as an \textbf{operator-valued probability space}.
\end{definition}

When $(M, \tau)$ is a tracial von Neumann algebra and $B = \C$, we can take $E = \tau$. As in the case when $B = \C$, there is a notion of freeness with respect to $E$: 

\begin{definition}
	Let $(M, E, B)$ be an operator-valued probability space. Let $\{A_i\}_{i \in I}$ be a family of von Neumann subalgebras of $M$ that each contain $B$. Then, $\{A_i\}_{i \in I}$ are \textbf{freely independent with amalgamation over }$\bm{B}$ if for any $a_j \in A_{k(j)}$ with $k(j) \neq k(j + 1)$, $j = 1, \ldots, n = 1$ and $E(a_i) = 0$, then 
	\begin{equation}
		E(a_1 \ldots, a_n) = 0 \,.
	\end{equation}
	Let $r, (m_k)_{1 \leq k \leq r}$ be positive integers. The sets $\{X_{1, p}, \ldots, X_{m_p, p}  \}_{1 \leq p \leq r}$ of non-commutative random variables are \textbf{free with amalgamation over }$\bm{B}$ if the algebras they generate with $B$ are free.
\end{definition}

Now, we define the operator-valued Cauchy transform and consider its domain of definition. 

Recall that the usual complex-valued Cauchy transform for a Hermitian $x \in (M, \tau)$ is a complex analytic function $G_x: \C \setminus \R \to \C \setminus \R$ given by: 
\begin{equation}
	G_x(z) = \tau \left[ (z - x)^{-1} \right] .
\end{equation}
Letting $\mathbb{H}^+(\C)$ be the upper half-plane and $\mathbb{H}^-(\C)$ be the lower half-plane, then in particular, $G_x: \mathbb{H}^+(\C) \to \mathbb{H}^-(\C)$.

We can generalize the complex-valued Cauchy transform by considering a Cauchy transform that takes values in a subalgebra $B \subset M$, where $\tau$ is replaced by the conditional expectation $E$. We state the definition of the operator-valued Cauchy transform and some facts about it: 

\begin{definition}
	Let $(M, E, B)$ be an operator-valued probability space. For $x \in (M, \tau)$, let $\Im(x) = \frac{x + x^*}{2 i}$. Let the \textbf{operator upper/lower half-planes} be:
	\begin{equation}
		\begin{aligned}
			\mathbb{H}^+(B) & = \{ x \in B : \Im(x) > 0\} \\
			\mathbb{H}^-(B) & = \{ x \in B : \Im(x) < 0 \} \,.
		\end{aligned}
	\end{equation}
	Let $x \in (M, \tau)$ where $x$ is Hermitian. Then, the \textbf{operator-valued Cauchy transform} $\mathcal{G}_{\mathbf{X}}: \mathbb{H}^+(B) \to \mathbb{H}^- (B)$ is given by: 
	\begin{equation}
		\mathcal{G}_{\mathbf{X}}(b) = E \left[ (b - X)^{-1}  \right]  .
	\end{equation}
	Further, the function $\mathcal{G}_{\mathbf{X}}$ is Fr\'echet differentiable on $\mathbb{H}^+(B)$.
\end{definition}

Analogous to the complex-valued case, there is also an operator-valued $R$-transform: 

\begin{definition}
	Let $(M, E, B)$ be an operator-valued probability space. Let $x \in (M, \tau)$ where $x$ is Hermitian. 
	
	Then for $b \in \mathbb{H}^+(B)$ in a neighborhood of infinity, $\mathcal{G}_x(b)$ is invertible. Thus, for $c \in \mathbb{H}^-(B)$ in a neighborhood of $0$, we may define the \textbf{operator-valued }$\bm{R}$\textbf{-transform}
	\begin{equation}
		\mathcal{R}_x(c) = \mathcal{G}_x^{\brackets{-1}}(c) - c^{-1} \,.
	\end{equation}
\end{definition}

Finally, the connection between the freeness and the $R$-transform in the operator-valued case is the operator-valued addition law: 

\begin{theorem}
	Let $(M, E, B)$ be an operator-valued probability space, and let $x, y \in (M, E, B)$ be Hermitian and freely independent with amalgamation over $B$. Where the functions are defined,
	\begin{equation}
		\mathcal{R}_{x + y}(c) = \mathcal{R}_{x}(c) + \mathcal{R}_{y}(c) \,.
	\end{equation}
\end{theorem}

Now, we introduce the setup in the mathematical literature to our problem of determining the Brown measure of $X = p + i q$: Let $(M, \tau)$ be a tracial von-Neumann algebra. Then,  $\left( M_2(M),   \tau \left[ \frac{1}{2} \tr \right]  \right)   $ is another tracial von-Neumann algebra. The unique trace-preserving conditional expectation from $M_2(M)$ onto $M_2(\C)$ is the \textbf{\textit{block trace}}: $\text{bTr}: M_2(M) \to M_2(\C)$, given by: 
\begin{equation}
	\text{bTr} \left[ \begin{pmatrix}
		m_{11} & m_{1 2} \\
		m_{2 1} & m_{2 2}
	\end{pmatrix}  \right]
	=  \begin{pmatrix}
		\tau(m_{11}) & \tau(m_{1 2}) \\
		\tau(m_{2 1}) & \tau(m_{2 2})
	\end{pmatrix}  \,.
\end{equation}   
Thus, we will work in the operator-valued probability space $(M_2(M), \text{bTr}, M_2(\C))$. We will use $\sim$ over the notation coming from the mathematics literature, then drop it for the corresponding objects from the physics literature.

Let $X \in M$. Then, consider 
\begin{equation}
	\mathbf{\tilde{X}} = \begin{pmatrix}
		0 & X \\
		X^* & 0 
	\end{pmatrix} \in M_2(M) \,.
\end{equation}
Note that $\mathbf{\tilde{X}}$ is Hermitian, so we may consider the operator-valued Cauchy transform $\mathcal{G}_\mathbf{\tilde{X}}: \mathbb{H}^+(M_2(\C)) \to \mathbb{H}^-(M_2(\C))$ given by: 
\begin{equation}
	\mathcal{G}_\mathbf{\tilde{X}} (\tilde{Q}) = \text{bTr} [(\tilde{Q} - \mathbf{\tilde{X}})^{-1}]  \,,
\end{equation}
where $\tilde{Q} \in \mathbb{H}^+(M_2(\C))$.

We consider $\tilde{Q} \in M_2(\C)$ of the form: 
\begin{equation}
	\tilde{Q}
	= 
	\begin{pmatrix}
		i \bar{B} & A \\
		\bar{A} & i B
	\end{pmatrix}, \qquad A, B \in \C. 
\end{equation}
Then, 
\begin{equation}
	\Im(\tilde{Q}) = 
	\begin{pmatrix}
		\Re(B) & 0 \\
		0 & \Re(B)
	\end{pmatrix}
\end{equation}
and hence $\tilde{Q} \in \mathbb{H}^+(M_2(\C))$ when $\Re(B) > 0$.

To pass from the mathematics notation to the physics notation, let $J \in M_2(\C)$ be the following matrix:
\begin{equation}
	J = 
	\begin{pmatrix}
		0 & 1 \\
		1 & 0
	\end{pmatrix} .
\end{equation} 
Note that $J^2 = I$, so that $J = J^{-1}$. Then, let the corresponding physics quantities be: 
\begin{equation}
	\label{eqn:math_phys_1}
	\begin{aligned}
		\mathbf{X} & = \mathbf{\tilde{X}} J = \begin{pmatrix}
			X & 0 \\
			0 & X^*
		\end{pmatrix} \in M_2(M) \\
		Q &= \tilde{Q} J = \begin{pmatrix}
			A & i \bar{B} \\
			i B & \bar{A}
		\end{pmatrix} \in M_2(\C) \,.
	\end{aligned}
\end{equation}
Note that $Q$ is exactly the general form of a quaternion from (\ref{eqn:quaternions_matrix}). Then, we can define the \textbf{\textit{Quaternionic Green's function}} by: 
\begin{equation}
	\mathcal{G}_\mathbf{X}(Q) = \text{bTr}[(Q - \mathbf{X})^{-1}] \,.
\end{equation}
The Quaternionic Green's function $\mathcal{G}_\mathbf{X}(Q)$ and operator-valued Cauchy transform $\mathcal{G}_\mathbf{\tilde{X}} (\tilde{Q})$ are related by the following formula: 
\begin{equation}
	\label{eqn:math_phys_2}
	\mathcal{G}_{\mathbf{X}}(Q) =  J \, \mathcal{G}_\mathbf{\tilde{X}} (\tilde{Q}) \,.
\end{equation}
%

From general facts about the operator-valued Cauchy transform, we know that $\mathcal{G}_{\mathbf{X}}(Q)$ is well-defined for $B > 0$ and maps into $J \, \mathbb{H}^-(M_2(\C))$. But, we can say more: $\mathcal{G}_{\mathbf{X}}(Q)$ is defined for quaternions $Q$ where $B \neq 0$, and in this case $\mathcal{G}_{\mathbf{X}}(Q)$ is a quaternion. For details of the straightforward computation, see \cite{PhysRevE.92.052111}. Thus, we have the following definition of the Quaternionic Green's function: 

\begin{definition}
	\label{def:quaternionic_green}
	Let $X \in (M, \tau)$. For $A \in \C$, let $X_A = A - X$. Then, the \textbf{Quaternionic Green's function} $\mathcal{G}_\mathbf{X}: \mathbb{H} \setminus \sigma(X) \to \mathbb{H}$ is given by:
	\begin{equation}
		\begin{aligned}
			\mathcal{G}_\mathbf{X}(Q) 
			& = \mathrm{bTr}[(Q - \mathbf{X})^{-1}] \\
			& = \mathrm{bTr} \left[ \left(\begin{pmatrix}
				A & i \bar{B} \\
				i B & \bar{A}
			\end{pmatrix} - 
			\begin{pmatrix}
				X & 0 \\
				0 & X^* 
			\end{pmatrix} \right)^{-1}  \right] \\
			&= 	\begin{pmatrix}
				\tau [((X_A)^*X_A + \Abs{B}^2)^{-1} (X_A)^* ]  & - i \bar{B} \, \tau[((X_A)^*X_A + \Abs{B}^2 )^{-1}] \\
				- i B \, \tau[((X_A)^*X_A + \Abs{B}^2 )^{-1}] & \tau[X_A((X_A)^* X_A + \Abs{B}^2)^{-1}]  
			\end{pmatrix} .
		\end{aligned}
	\end{equation} 
\end{definition}

When $B = 0$ in the quaternion $Q$, then $Q = A \in \C$ and 
\begin{equation}
	\mathcal{G}_\mathbf{X}(Q) = \mathcal{G}_\mathbf{X}(A) = \tau[(A - X)^{-1}] \,,
\end{equation}
and this formula is well-defined even for non-Hermitian $X$ when $A \not \in \sigma(X)$. We will discuss the restriction of $\mathcal{G}_{\mathbf{X}}$ to $\C$ in a later section. 

Let $\mathbb{H}^{+}$ be the \textbf{\textit{upper Quaternionic half-plane}} (and $\mathbb{H}^-$ be the \textbf{\textit{lower Quaternionic half-plane}}): 
\begin{equation}
	\begin{aligned}
		\mathbb{H}^{+} & = \{ Q \in \mathbb{H} : B > 0\} \\
		\mathbb{H}^{-} & = \{ Q \in \mathbb{H} : B < 0\}  \,.
	\end{aligned}
\end{equation}
Note that the Quaternionic upper/lower half-planes are not the operator upper/lower half-planes of $M_2(\C)$ restricted to $\mathbb{H}$.

As a generalization of the fact that the complex Cauchy transform swaps the upper/lower half-planes, it is easy to see from the formula for $\mathcal{G}_{\mathbf{X}}$ that:  
\begin{equation}
	\mathcal{G}_\mathbf{X}: \mathbb{H}^{\pm} \to \mathbb{H}^{\mp} \,.
\end{equation} 
Similarly, as a consequence of the formula for $\mathcal{G}_{\mathbf{X}}$, there is a generalization of the conjugate symmetry of the complex Cauchy transform: 
\begin{equation}
	\mathcal{G}_{\mathbf{X^*}}(Q^*) = \mathcal{G}_\mathbf{X}(Q)^* \,.
\end{equation}
Now, we provide the relevant formulas relating the Inverse Quaternionic Green's function with the inverse operator-valued Cauchy transform and the Quaternionic $R$-transform with the operator-valued $R$-transform. These formulas follow directly from (\ref{eqn:math_phys_1}) and (\ref{eqn:math_phys_2}):

\begin{definition}
	Let $X \in (M, \tau)$. The \textbf{Inverse Quaternionic Green's function}, $\mathcal{B}_{\mathbf{X}} = \mathcal{G}_{\mathbf{X}}^{\brackets{-1}}$, is defined for quaternions $W \in J \, \mathbb{H}^{-}(M_2(\C))$ in a neighborhood of $0$. Its relationship with the inverse operator-valued Cauchy transform $\mathcal{G}_{\mathbf{\tilde{X}}}^{\brackets{-1}}$ is: 
	\begin{equation}
		\mathcal{B}_{\mathbf{X}}(W) = \mathcal{G}_{\mathbf{X}}^{\brackets{-1}}(W) = \mathcal{G}_{\mathbf{\tilde{X}}}^{\brackets{-1}}(J \, W) \, J .
	\end{equation}
	The \textbf{Inverse Quaternionic }$\bm{R}$\textbf{-transform}, $\mathcal{R}_{\mathbf{X}}$, is defined for quaternions $W \in J \, \mathbb{H}^{-}(M_2(\C))$ in a neighborhood of $0$. Its relationship with the inverse operator-valued Cauchy transform $\mathcal{R}_{\mathbf{\tilde{X}}}$ is: 
	\begin{equation}
		\label{eqn:quaternionic_operator_R}
		\mathcal{R}_{\mathbf{X}}(W) = \mathcal{G}_{\mathbf{X}}^{\brackets{-1}}(W) - W^{-1} = \mathcal{R}_{\mathbf{\tilde{X}}}( J \, W) \, W .
	\end{equation}
\end{definition}

We conclude that there is an addition law for the Quaternionic $R$-transform: 

\begin{proposition}
	Let $x, y \in (M, \tau)$ be freely independent. For quaternions $W \in J \, \mathbb{H}^{-}(M_2(\C))$ in a neighborhood of $0$, 
	\begin{equation}
		\mathcal{R}_\mathbf{x + y}(W) = \mathcal{R}_{\mathbf{x}}(W) + \mathcal{R}_{\mathbf{y}}(W) \,.
	\end{equation}
\end{proposition}
\begin{proof}
	Consider the operator-valued probability space $(M_2(M), \text{bTr}, M_2(\C))$. If $x$ and $y$ are free in $(M, \tau)$, then it is easy to see that $\mathbf{\tilde{x}}$ and $\mathbf{\tilde{y}}$ are free with amalgamation over $M_2(\C)$. Hence, from the operator-valued addition law, 
	\begin{equation}
		\mathcal{R}_{\mathbf{\tilde{x} + \tilde{y}}}(c) = \mathcal{R}_{\mathbf{\tilde{x}}}(c) + \mathcal{R}_{\mathbf{\tilde{y}}}(c) \,,
	\end{equation}
	and this identity holds for $c \in \mathbb{H}^-(M_2(\C))$ in a neighborhood of $0$. Then, for quaternions $W \in J \, \mathbb{H}^{-}(M_2(\C))$ in a neighborhood of $0$, 
	\begin{equation}
		\begin{aligned}
			\mathcal{R}_{\mathbf{x} + \mathbf{y}}(W) & = \mathcal{R}_{\mathbf{\tilde{x} + \tilde{y}}}( J\, W) \, W \\
			&= \mathcal{R}_{\mathbf{\tilde{x}}}(J \, W) \, W + \mathcal{R}_{\mathbf{\tilde{y}}}(J \, W) \, W \\
			&= \mathcal{R}_{\mathbf{x}}(W) + \mathcal{R}_{\mathbf{y}}(W) \,.
		\end{aligned}
	\end{equation}
\end{proof}

\subsection{Quaternionic Green's function and Brown measure}
In this subsection, we discuss the relationship between the Quaternionic Green's function and the Brown measure. There is a strong analogy between the Quaternionic Green's function for an arbitrary $X \in (M, \tau)$ and the usual complex Cauchy transform for a Hermitian $x \in (M, \tau)$. Recall that the key utility of $G_x$ is that the spectral measure of $x$ is a distributional limit of the Cauchy transform approaching the real axis: 

\begin{equation}
	\mu_x = \lim\limits_{b \to 0^+ }- \frac{1}{\pi} \Im \, G_\mu(\cdot + i b) \, .
\end{equation}

To complete the analogy, we describe how to recover the Brown measure of $X$ as a limit of the Quaternionic Green's function approaching the complex plane. 

First, let $F: \mathbb{H} \to \mathbb{H}$ capture the first part of the quaternion, i.e. for a quaternion $Q$ as in (\ref{eqn:quaternions_matrix}), $F( Q) = A$.

For $z \in \C$ and $\epsilon > 0$, consider the quaternion
\begin{equation}
	z_\epsilon = 
	\begin{pmatrix}
		z & i \epsilon \\
		i \epsilon & \bar{z}
	\end{pmatrix} \in \mathbb{H} .
\end{equation}
Then, 
\begin{equation}
	F(\mathcal{G}_{\mathbf{X}}(z_\epsilon)) = \tau [((X_z)^*X_z + \epsilon^2)^{-1} (X_z)^* ] = 2 \, \frac{\partial}{\partial z} f_{\epsilon^2}(z) \, ,
\end{equation}
where $f_{\epsilon^2}: \C \to \R$ is from (\ref{eqn:f_epsilon}), given by: 
Then, from (\ref{eqn:brown_measure_f_epsilon}), the Brown measure of $X$, $\mu_X$, is given by the following distributional limit: 
\begin{equation}
	\mu_X = \lim\limits_{\epsilon \to 0^+} \frac{1}{\pi }  \frac{\partial}{\partial \bar{z}} F(\mathcal{G}_\mathbf{X}(z_\epsilon)) \,.
\end{equation}
In the mathematics notation, $\tilde{z_\epsilon} = z_\epsilon \, J \in \mathbb{H}^+(M_2(\C))$ for $\epsilon > 0$ and we can recover the Brown measure from a similar formula in terms of the operator-valued Cauchy transform as $\epsilon \to 0^+$.

\section{Outline for computing Inverse Quaternionic Green's function for $X = p + i q$}
\label{sec:B_X_outline}
In this section, we consider $X = p + i q$, where $p, q \in (M, \tau)$ are Hermitian and freely independent. Consider the operator-valued probability space $(M_2(M), \text{bTr}, M_2(\C))$ and the following elements of $M_2(M)$: 
\begin{equation}
	\mathbf{p} = 
	\begin{pmatrix}
		p & 0 \\
		0 & p
	\end{pmatrix}
	\qquad
	i \mathbf{q} = 
	\begin{pmatrix}
		i q & 0 \\
		0 & - i q
	\end{pmatrix} .
\end{equation}
From the addition law for the Quaternionic $R$-transform,
\begin{equation}
	\mathcal{R}_{\mathbf{X}}(Q) = \mathcal{R}_{\mathbf{p}}(Q) + \mathcal{R}_{i \mathbf{q}}(Q) \,.
\end{equation}  
In terms of the Inverse Quaternionic Green's function, 
\begin{equation}
	\label{eqn:bX_free}
	\mathcal{B}_{\mathbf{X}}(Q) = \mathcal{B}_{\mathbf{p}}(Q) + \mathcal{B}_{i \mathbf{q}}(Q) - Q^{-1} \,.
\end{equation}
In this section, we will outline how to obtain expressions for $\mathcal{B}_{\mathbf{p}}(Q)$ and $\mathcal{B}_{i \mathbf{q}}(Q)$, and hence how to obtain an expression for $\mathcal{B}_{\mathbf{X}}(Q)$.

The computations in this section are taken from (\cite{jarosz2004novel}, Section 4), so we will justify the domains where the computations make sense and state their results. 

\subsection{(Inverse) Quaternionic Green's function at a complex number}
In this subsection, we make some observations about the restriction of the (Inverse) Quaternionic Green's function of an arbitrary $X \in (M, \tau)$ to the complex numbers. 

Recall that the Cauchy transform of a Hermitian $x \in (M, \tau)$ has $G_x(z) \in \R$ if and only if $z \in \R \setminus \sigma(x)$. We generalize this fact for $\mathcal{G}_{\mathbf{X}}$:

Let $X \in (M, \tau)$ and consider a complex number quaternion, $z \in \C \setminus \sigma(X)$. From the definition of the Quaternionic Green's function, 
\begin{equation}
	\mathcal{G}_{\mathbf{X}}(z) = \tau[(z - X)^{-1}] \,.
\end{equation}
Note that this is the same formula for the Cauchy transform $G_X(z)$, except $X$ is potentially non-Hermitian. Even though $X$ is not necessarily Hermitian, this function is defined and analytic for $z \not \in \sigma(X)$. In this context, we will refer to this function as the \textbf{\textit{complex Green's function}} (or \textbf{\textit{Holomorphic Green's function}}, as used in \cite{jarosz2004novel}) of $X$ and still use the notation $G_X$. The significance of the term ``Holomorphic Green's function'' will be discussed in a later section.

Conversely, consider $\mathcal{G}_{\mathbf{X}}(Q)$ for $Q \in \mathbb{H} \setminus \C$, i.e. $B \neq 0$. The off-diagonal term of $\mathcal{G}_{\mathbf{X}}(Q)$ is: 
\begin{equation}
	- i \bar{B} \, \tau[((X_A)^*X_A + \Abs{B}^2 )^{-1}] \,,
\end{equation}
where $X_A = A - X$. If $B \neq 0$, then this term is non-zero. Hence, 
\begin{equation}
	\mathcal{G}_{\mathbf{X}}(Q) \in \C \iff Q \in \C \,,
\end{equation}
and in this case $\mathcal{G}_{\mathbf{X}}(z) = G_X(z)$.

When $X$ is non-Hermitian, $G_X(z)$ still satisfies $z G_X(z) \to 1$ as $\Abs{z} \to \infty$. In particular, $G_X$ is still invertible in a neighborhood of infinity, and its inverse is defined in a neighborhood of $0$. Let $B_X = G_X^{\brackets{-1}}$. Let $\mathcal{B}_\mathbf{X}: V \to \mathbb{H}$. For $Q \in V$, 
\begin{equation}
	\mathcal{B}_{\mathbf{X}}(Q) \in \C \iff Q = \mathcal{G}_{\mathbf{X}}( \mathcal{B}_{\mathbf{X}}(Q) ) \in \C \,.
\end{equation}
Let $U, V \subset \mathbb{H}$ be domains where $\mathcal{G}_{\mathbf{X}}: U \to V$,  $\mathcal{B}_{\mathbf{X}}: V \to U$, $\mathcal{B}_\mathbf{X} \circ \mathcal{G}_\mathbf{X} = \text{id}_U$, and $\mathcal{G}_\mathbf{X} \circ \mathcal{B}_\mathbf{X} = \text{id}_V$. Then, for $z \in U \cap \C$ and $w \in V \cap \C$, 
\begin{equation}
	\begin{aligned}
		G_X( \mathcal{B}_\mathbf{X}(w) ) & = \mathcal{G}_{\mathbf{X}}( \mathcal{B}_\mathbf{X}(w) ) = w \\
		z & = \mathcal{B}_{\mathbf{X}} ( \mathcal{G}_{\mathbf{X}}(z) ) = \mathcal{B}_{\mathbf{X}}(G_X(z)) \,.
	\end{aligned}
\end{equation}
Hence, $\mathcal{B}_\mathbf{X}$ is an inverse for $G_X$ on $U \cap \C$, so $\mathcal{B}_{\mathbf{X}} = B_X$ on $U \cap \C$. 

We summarize all of this in the following Proposition: 

\begin{proposition}
	\label{prop:B_X_complex}
	Let $X \in (M, \tau)$. Then, $\mathcal{G}_{\mathbf{X}}(Q) \in \C$ if and only if $Q \in \C \setminus \sigma(X)$. For $z \in \C \setminus \sigma(X)$, 
	\begin{equation}
		\mathcal{G}_\mathbf{X}(z) = G_X(z) \,.
	\end{equation}
	Let $B_X = G_X^{\brackets{-1}}$ and let $\mathcal{B}_{\mathbf{X}} = \mathcal{G}_{\mathbf{X}}^{\brackets{-1}}$. Let $\mathcal{B}_\mathbf{X}: V \to \mathbb{H}$. Then, $\mathcal{B}_{\mathbf{X}}(Q) \in \C$ if and only if $Q \in V \cap \C$. For $w \in V \cap \C$, 
	\begin{equation}
		\mathcal{B}_\mathbf{X}(w) = B_X(w) \,.
	\end{equation}
\end{proposition}

\subsection{(Inverse) Quaternionic Green's function for a Hermitian operator}
In this subsection, we describe how to compute the Quaternionic Green's Function for a Hermitian $H \in (M, \tau)$, $\mathcal{G}_{\mathbf{H}}$, in terms of the usual Cauchy transform $G_H$. Using analogous computations, we will determine where the Inverse Green's Function for $H$, $\mathcal{B}_{\mathbf{H}}$, is well-defined and compute $\mathcal{B}_{\mathbf{H}}$ in terms of the inverse of the usual Cauchy transform, $B_H = G_H^{\brackets{-1}}$.

To compute $\mathcal{G}_{\mathbf{H}}(Q)$, recall when $Q = g \in \C \setminus \sigma(H)$, 
\begin{equation}
	\mathcal{G}_\mathbf{H}(g) = G_H(g) \,.
\end{equation}
For an arbitrary quaternion $Q \not \in \R$, we may diagonalize $Q$ by: 
\begin{equation}
	Q = S^{-1} g S \,,
\end{equation}
where 
\begin{equation}
	S = 
	\begin{pmatrix}
		i B & g - A \\
		\bar{g} - \bar{A} & i \bar{B}
	\end{pmatrix} .
\end{equation}
Since $H$ is Hermitian, then $\mathbf{H}$  commutes $M_2(\C)$. This combined with the bimodularity of \text{bTr} shows that
\begin{equation}
	\mathcal{G}_\mathbf{H}(Q) = \mathcal{G}_\mathbf{H}( S^{-1} g S) = S^{-1} \mathcal{G}_\mathbf{H}(  g )S \,.
\end{equation}
Note that since $Q \not \in \R$, then $g \not \in \R$, so $\mathcal{G}_\mathbf{H}(  g )$ is well-defined. 

Thus, we can express $\mathcal{G}_\mathbf{H}(Q)$ in terms of $Q$ and $G_H(g)$. The result of the computation is: 
\begin{equation}
	\mathcal{G}_\mathbf{H}(Q) = 
	\begin{dcases}
		\gamma_H(g) 1 - \gamma_H'(g) Q^* & Q \not \in \R \\
		G_H(g) & Q = g \in \C \setminus \sigma(H) \,,
	\end{dcases}
\end{equation}
where 
\begin{equation}
	\begin{aligned}
		\gamma_H(g) & = \frac{g G_H(g) - \bar{g} G_H(\bar{g})  }{g - \bar{g}} \\
		\gamma_H'(g) & = \frac{G_H(g) - G_H(\bar{g})}{g - \bar{g}} \,.
	\end{aligned}
\end{equation}
It is straightforward to check that both formulas for $\mathcal{G}_\mathbf{H}(Q)$ agree for $Q \in \C \setminus \R$.

While the computation is carried out with the convention that $\Im(g) \geq 0$, $\gamma_H$ and $\gamma_H'$ only depends on the \textit{set} of eigenvalues of $Q$ since $\gamma_H(g) = \gamma_H(\bar{g})$ and $\gamma_H'(g) = \gamma_H'(\bar{g})$. Finally, note that $\gamma_H$, $\gamma_H'$ are real-valued. 

For what follows, we define the following notation for the eigenspaces of a quaternion $Q$: 

\begin{definition}
	Let $Q \in \mathbb{H}$ and $\lambda \in \C$. Then, define  $\bm{E_\lambda(Q)}$ to be the $\mathbf{\lambda}$\textbf{-eigenspace} for $Q$, i.e. $E_\lambda(Q) = \ker(Q - \lambda)$.
\end{definition}

We observe the properties $\mathcal{G}_{\mathbf{H}}(g) = G_H(g)$ and $\mathcal{G}_{\mathbf{H}}(Q) = S^{-1} \mathcal{G}_{\mathbf{H}}(g) S$ lead to the following coordinate-free characterization of $\mathcal{G}_\mathbf{H}$ that will be useful in defining $\mathcal{B}_{\mathbf{H}}$.

\begin{proposition}
	Let $H \in (M, \tau)$ be Hermitian. Let $Q \in \mathbb{H}$ be a quaternion such that $Q \not \in \sigma(X)$. Then, $\mathcal{G}_{\mathbf{H}}: \mathbb{H} \setminus \sigma(X) \to \mathbb{H}$ is defined by the property that
	\begin{equation}
		E_{G_H(g)}(\mathcal{G}_{\mathbf{H}}(Q)) = E_{g}(Q) \,.
	\end{equation}
\end{proposition}
\begin{proof}
	First, note that the definition is well-defined, i.e. that $G_H$ is 1-1 on the set of eigenvalues of $Q$. When $Q \in \R$, then there is only one eigenvalue. When $Q \not \in \R$, then $Q$ has two eigenvalues $g, \bar{g} \in \C \setminus \R$. Then, since $G_H(g) \in \R$ if and only if $g \in \R \setminus \sigma(H)$, then $G_H(\bar{g}) = \bar{G_H(g)} \neq G_H(g)$. As $Q$ is a quaternion, it is diagonalizable, and hence this property defines $\mathcal{G}_{\mathbf{H}}(Q)$.
\end{proof}

This leads to a coordinate-free definition of $\mathcal{B}_{\mathbf{H}} = \mathcal{G}_{\mathbf{H}}^{\brackets{-1}}$ and a domain where it is defined. Then, the analogous computations as for $\mathcal{G}_{\mathbf{H}}(Q)$ gives an explicit formula for $\mathcal{B}_{\mathbf{H}}$ in terms of $Q$ and $B_H(g)$. The results of this are summarized in the following Proposition: 

\begin{proposition}
	\label{prop:B_H}
	Let $H \in (M, \tau)$ be Hermitian. Let $B_H: U \to \C$ be the inverse of $G_H$, where $U$ is an open neighborhood of $0$ that is fixed under complex conjugation. Then, $\mathcal{B}_{\mathbf{H}} = \mathcal{G}_{\mathbf{H}}^{\brackets{-1}}$ is defined on $\mathbb{H}_U = \{ Q \in \mathbb{H} : g \in U  \}$ and $\mathcal{B}_{\mathbf{H}}$ is defined by the property that: 
	\begin{equation}
		E_{B_H(g)}(\mathcal{B}_{\mathbf{H}}(Q)) = E_g(Q) \,.
	\end{equation}
	In particular, $\mathcal{B}_{\mathbf{H}}(g) = B_H(g)$ for $g \in U$ and $\mathcal{B}_{\mathbf{H}}(Q) = S^{-1} \mathcal{B}_{\mathbf{H}}(Q) S$ for $Q \in \mathbb{H}_U$, $S \in GL_2(\C)$.  
	
	$\mathcal{B}_\mathbf{H}$ is continuous on $\mathbb{H}_U$ and an explicit formula for $\mathcal{B}_{\mathbf{H}}$ is given by: 
	\begin{equation}
		\label{eqn:blue_hermitian}
		\mathcal{B}_{\mathbf{H}}(Q)
		=
		\begin{dcases}
			\beta_H(g) 1 - \beta_H'(g) Q^* & Q \not \in \R \\
			B_H(x_0) 1 & Q = x_0 1 \in \R \,,
		\end{dcases}
	\end{equation}
	where 
	\begin{equation}
		\label{eqn:blue_beta}
		\begin{aligned}
			\beta_H(g) & = \frac{g B_H(g) - \bar{g} B_H(\bar{g})  }{g - \bar{g}} \\
			\beta_H'(g) & = \frac{B_H(g) - B_H(\bar{g})}{g - \bar{g}} \,.
		\end{aligned}
	\end{equation}
	Finally, $\beta_H(g) = \beta_H(\bar{g})$, $\beta_H'(g) = \beta_H'(\bar{g})$, and $\beta_H$, $\beta_H'$ are real-valued.
\end{proposition}
\begin{proof}
	First, note that we may extend the domain of $B_H$ to one that is fixed by complex conjugation, because of the conjugate symmetry of $G_H$. 
	
	The definition is well-defined because $B_H$ is  1-1 on the set of eigenvalues of $Q$. When $Q \in \R$, there is only one eigenvalue. When $Q \not \in \R$, then $Q$ has two eigenvalues $g, \bar{g} \in \C \setminus \R$. Then, since $B_H(g) \in \R$ implies that $g \in \R$. As $\mathcal{B}_{\mathbf{H}}(Q)$ is diagonalizable, then this defines $\mathcal{B}_{\mathbf{H}}(Q)$.
	
	This definition gives an inverse for $\mathcal{G}_{\mathbf{H}}$, since
	\begin{equation}
		\begin{aligned}
			E_g(\mathcal{G}_{\mathbf{H}} ( \mathcal{B}_{\mathbf{H}}(Q)  )  )  &= E_{G_H(B_H(g))}(\mathcal{G}_{\mathbf{H}} ( \mathcal{B}_{\mathbf{H}}(Q)  )  ) \\
			&= E_{B_H(g)}(\mathcal{B}_{\mathbf{H}}(Q)  ) \\
			&= E_{g}(Q) \,. 
		\end{aligned}
	\end{equation}
	implies that $\mathcal{G}_{\mathbf{H}} ( \mathcal{B}_{\mathbf{H}}(Q)  ) = Q$ and similar equalities show that $\mathcal{B}_{\mathbf{H}} ( \mathcal{G}_{\mathbf{H}}(Q)  ) = Q$.
	
	For $g \in \mathbb{H}^+(\C) \cap U$,
	\begin{equation}
		\begin{aligned}
			E_{B_H(g)}(\mathcal{B}_{\mathbf{H}}(g)) & = E_g(g) = e_1 \\
			E_{B_H(\bar{g})}(\mathcal{B}_{\mathbf{H}}(g)) & = E_{\bar{g}}(g) = e_2 \,.
		\end{aligned}
	\end{equation}
	so then $\mathcal{B}_{\mathbf{H}}(g) = B_H(g)$. When $g \in \mathbb{H}^-(\C) \cap U$, $e_1$ and $e_2$ are swapped in the above equalities. For $S \in GL_2(\C)$ and $Q \in \mathbb{H}_U$,
	\begin{equation}
		\begin{aligned}
			E_{B_H(g)}( \mathcal{B}_{\mathbf{H}}( S^{-1} Q S)  )  & = E_g (S^{-1} Q S) \\
			& = S^{-1} E_{g}(Q)  \\
			& = S^{-1} E_{B_H(g)}(\mathcal{B}_{\mathbf{H}}(Q)) \\
			&= E_{B_H(g)}(S^{-1} \mathcal{B}_{\mathbf{H}}(Q) S)  \,.
		\end{aligned}
	\end{equation}
	As $\mathcal{B}_{\mathbf{H}}(Q)$ is diagonalizable, then $\mathcal{B}_{\mathbf{H}}(S^{-1} Q S) = S^{-1} \mathcal{B}_{\mathbf{H}}(Q) S$. 
	
	For continuity of $\mathcal{B}_{\mathbf{H}}$, let $Q_n \to Q$, where $Q_n, Q \in \mathbb{H}_U$. 
	
	If $Q \in \R$, then $g_n, \bar{g_n} \to g = Q \in \R$ and $B_H(g_n), B_H(\bar{g_n}) \to B_H(g) = B_H(Q) \in \R$. Hence, the eigenvalues of $\mathcal{B}_{\mathbf{H}}(Q_n)$ converge to $B_H(Q)$, and from Lemma \ref{lem:Q_real_convergence}, this implies that $\mathcal{B}_{\mathbf{H}}(Q_n)$ converges to $\mathcal{B}_{\mathbf{H}}(Q) = B_H(Q)$.
	
	If $Q \not \in \R$, then $g_n \to g$, $\bar{g_n} \to \bar{g}$, and the diagonalizing transforms for $Q_n$, 
	\begin{equation}
		S_n = 
		\begin{pmatrix}
			i B_n & g_n - A_n \\
			\bar{g_n} - \bar{A_n} & i \bar{B_n}
		\end{pmatrix} 
	\end{equation}
	converge to $S$, the diagonalizing transform for $Q$. Then, using the conjugation property of $\mathcal{B}_{\mathbf{H}}$ shows that $\mathcal{B}_{\mathbf{H}}(Q_n)$ converges to $\mathcal{B}_{\mathbf{H}}(Q)$.
	
	The explicit formula for $\mathcal{B}_{\mathbf{H}}$ follows from the analogous computation as for $\mathcal{G}_{\mathbf{H}}$. The properties for $\beta_H, \beta'_H$ are self-evident. 
\end{proof}

\subsection{(Inverse) Quaternionic Green's function and multiplication by a complex number}
In this subsection, we consider an arbitrary $X \in (M, \tau)$. For $c \in \C$, we compare $\mathcal{G}_{c \mathbf{X}}$ with $\mathcal{G}_\mathbf{X}$, and similarly $\mathcal{B}_{c \mathbf{X}}$ with $\mathcal{B}_{\mathbf{X}}$.

For $c = 0$ and $Q \in \mathbb{H} \setminus \{0\}$,
\begin{equation}
	\mathcal{G}_{c \mathbf{X}}(Q) = \mathcal{G}_{0}(Q) = Q^{-1} \,.
\end{equation}
For $c \neq 0$ and $Q \in \mathbb{H} \setminus \sigma(c X) = \mathbb{H} \setminus (c \sigma(X))$,
\begin{equation}
	\label{eqn:G_cX}
	\mathcal{G}_{c \mathbf{X}}(Q) 
	= \mathcal{G}_{\mathbf{X}} \left(  
	\frac{1}{c} Q \right)  
	\frac{1}{c} \,.
\end{equation}
The order of multiplication is important, as the quaternion $Q$ may not commute with $c \in \C$.

This generalizes the formula for the complex Green's function, 
\begin{equation}
	G_{c X}(z) = \frac{1}{c}G_X\left( \frac{1}{c} z \right) .
\end{equation}
The formula for the inverse Green's function $\mathcal{B}_{cX}$ can be obtained by inverting this formula. This is summarized in the following Proposition: 

\begin{proposition}
	\label{prop:B_cX}
	Let $X \in (M, \tau)$ and let $\mathcal{B}_{\mathbf{X}}: V \to \mathbb{H}$, where $V \subset \mathbb{H}$. 
	
	For $c = 0$, $\mathcal{B}_{c \mathbf{X}} : \mathbb{H} \setminus \{0\} \to \mathbb{H}$ and 
	\begin{equation}
		\mathcal{B}_{c \mathbf{X}}(Q) = Q^{-1} \,.
	\end{equation}
	For $c \neq 0$, $\mathcal{B}_{c \mathbf{X}}: V \left( \frac{1}{c} \right)  \to \mathbb{H}$, and 
	\begin{equation}
		\label{eqn:blue_iq}
		\mathcal{B}_{c \mathbf{X}} (Q) = c \mathcal{B}_{\mathbf{X}}(Q c) \,.
	\end{equation}
\end{proposition}
\begin{proof}
	The result for $c = 0$ follows from $\mathcal{G}_{0 \mathbf{X}} = Q^{-1}$, so we consider $c \neq 0$. 
	
	Let $U, V \subset \mathbb{H}$ be domains where $\mathcal{G}_{\mathbf{X}}: U \to V$,  $\mathcal{B}_{\mathbf{X}}: V \to U$, $\mathcal{B}_\mathbf{X} \circ \mathcal{G}_\mathbf{X} = \text{id}_U$, and $\mathcal{G}_\mathbf{X} \circ \mathcal{B}_\mathbf{X} = \text{id}_V$. From (\ref{eqn:G_cX}), $\mathcal{G}_{c \mathbf{X}}: c \,  U \to V \left( \frac{1}{c} \right)$. Similarly, defining $\mathcal{B}_{c \mathbf{X}} (Q) = c \mathcal{B}_{\mathbf{X}}(Q c)$ shows that $\mathcal{B}_{c \mathbf{X}}: V \left( \frac{1}{c} \right) \to c \,  U$ and it is straightforward to check that $\mathcal{B}_{c\mathbf{X}} \circ \mathcal{G}_{c \mathbf{X}} = \text{id}_{c \, U}$, and $\mathcal{G}_{c\mathbf{X}} \circ \mathcal{B}_{c\mathbf{X}} = \text{id}_{V \left( \frac{1}{c} \right)}$.
\end{proof}

For the complex $B_{c X}$, there is a similar formula:
\begin{equation}
	\label{eqn:blue_complex_complex}
	B_{c X}(z) = c B_X(c z) \,.
\end{equation}
This can be proved analogously from the formula for the complex Green's function. 

\subsection{Inverse Quaternionic Green's function for $X = p + i q$}
In this subsection, we summarize how to compute $\mathcal{B}_{\mathbf{X}}$ when $X = p + i q$, where $p, q \in (M, \tau)$ are Hermitian and freely independent. Recall the definition of $g^I$ from (\ref{eqn:gI}).

From the addition law for the Quaternionic $R$-transform, and Proposition \ref{prop:B_cX},
\begin{equation}
	\label{eqn:BX_eqn}
	\begin{aligned}
		\mathcal{B}_{\mathbf{X}}(Q) & = \mathcal{B}_{\mathbf{p}}(Q) + \mathcal{B}_{i \mathbf{q}}(Q) - Q^{-1}  \\
		&= \mathcal{B}_{\mathbf{p}}(Q) + i \mathcal{B}_{\mathbf{q}}(Q i) - Q^{-1}  \,.
	\end{aligned}
\end{equation}
Let $B_p: U \to \C$ and $B_q: V \to \C$. Then, the right-hand side of (\ref{eqn:BX_eqn}) is well-defined for $g \in U$, $g^I \in V$.

Using (\ref{eqn:blue_hermitian}) to rewrite $\mathcal{B}_{\mathbf{p}}$ and $\mathcal{B}_{\mathbf{q}}$ when $g, g^I \not \in \R$, 
\begin{equation}
	\label{eqn:blue_formula}
	\mathcal{B}_{\mathbf{X}}(Q) = \beta_p(g) + \beta_{q}(g^I) i
	- \left(  \beta_{p}'(g) + \beta_{q}'(g^I) + \frac{1}{\det Q} \right) Q^*  \,.
\end{equation}
Expanding this completely,
\begin{equation}
	\label{eqn:blue_formula_expanded}
	\mathcal{B}_{\mathbf{X}}(Q) =
	\mathcal{B}_{\mathbf{X}}\left( 
	\begin{pmatrix}
		A & i \bar{B} \\
		i B & \bar{A}
	\end{pmatrix}
	\right)  
	= 
	\begin{pmatrix}
		k + i k' - l \bar{A} & l i \bar{B} \\
		l i B & k - i k' - l A
	\end{pmatrix}
	\, ,
\end{equation}
where 
\begin{equation}
	\label{eqn:k_k'_l}
	\begin{aligned}
		k & = \beta_p(g) \\
		k' &= \beta_q(g^I) \\
		l &= \beta_p'(g) + \beta_q'(g^I) + \frac{1}{\det Q} \, .
	\end{aligned}
\end{equation}
We conclude by making some observations:

\begin{itemize}
	\item $k$, $k'$, and $l$ are real-valued.
	
	\item The addition law (\ref{eqn:bX_free}) defines $\mathcal{B}_{\mathbf{X}}$ and works for all $Q$ in the domain of $\mathcal{B}_{\mathbf{X}}$, but the expanded formula (\ref{eqn:blue_formula_expanded}) only works when $g, g^I \not \in \R$. 
	
	\item We can apply Lemma \ref{lem:conjugation} to $\beta_p, \beta_q, \beta_p', \beta_q'$, which are all real-valued and respect conjugation. Then, when evaluating these functions at some $Q \in \C$, we can assume that $g = Q $ and $g^I = i g$ in (\ref{eqn:blue_formula}) and (\ref{eqn:blue_formula_expanded}).
	
	\item Due to the presence of $\mathcal{B}_\mathbf{q}(Q i)$ in (\ref{eqn:blue_iq}), it does not follow that $\mathcal{B}_\mathbf{X}(S^{-1} Q S) = S^{-1} \mathcal{B}_\mathbf{X}(Q) S$ as in the case when $X$ is Hermitian. In particular, $\mathcal{B}_\mathbf{X}$ is not determined by how it acts on diagonal (i.e. complex) quaternions. Further, it is not necessarily true like in the Hermitian case that the eigenvalues of $\mathcal{B}_\mathbf{X}$ are just $B_X(g)$, $B_X(\bar{g})$.

\end{itemize}

In summary, we will complete the following steps to compute $\mathcal{B}_{\mathbf{X}}$:

\begin{enumerate}
	\item Compute $B_p$ (resp. $B_q$).
	\item Use (\ref{eqn:blue_beta}) to compute $\beta_p, \beta_p'$ (resp. $\beta_q, \beta_q'$).
	\item Use (\ref{eqn:blue_hermitian}) to compute $\mathcal{B}_\mathbf{p}$ (resp. $\mathcal{B}_\mathbf{q}$).
	\item Use (\ref{eqn:blue_iq}) to compute $\mathcal{B}_{i\mathbf{q}}$.
	\item Use (\ref{eqn:blue_formula_expanded}) to compute $\mathcal{B}_\mathbf{X}$.
\end{enumerate}

\section{Heuristics}
\label{sec:heuristics}
In this section, we provide the heuristics for the boundary and support of the Brown measure of $X = p + i q$ when $p, q \in (M, \tau)$ are Hermitian and freely independent. 

Recall that the Brown measure of $X \in (M, \tau)$ can be defined by: 
\begin{equation}
	\label{eqn:brown_measure_limit}
	\mu_X = \lim\limits_{\epsilon \to 0^+} \frac{1}{\pi }  \frac{\partial}{\partial \bar{z}} F(\mathcal{G}_\mathbf{X}(z_\epsilon)) \,.
\end{equation}
Let $X = p + i q$ where $p, q$ are Hermitian and freely independent. Consider $z \in \C$ such that $\mathcal{G}_{\mathbf{X}}(z_\epsilon)$ is in the domain of $\mathcal{B}_{\mathbf{X}}$ for all sufficiently small $\epsilon > 0$.

Let 
\begin{equation}
	\mathcal{G}_\mathbf{X} (z_\epsilon)
	= 	\begin{pmatrix}
		A_\epsilon & i \bar{B_\epsilon} \\
		i B_\epsilon & \bar{A_\epsilon}
	\end{pmatrix}	\,,
\end{equation}
where $A_\epsilon$ and $B_\epsilon$ are implicitly understood to depend on $z$. We wish to analyze $\mathcal{G}_\mathbf{X} (z_\epsilon)$ as $\epsilon \to 0^+$.

From the free independence of $p$ and $q$ and the Quaternionic addition law, it is more natural to consider $\mathcal{B}_\mathbf{X}$. If $\mathcal{G}_{\mathbf{X}}(z_\epsilon)$ is in the domain of $\mathcal{B}_\mathbf{X}$, then 
\begin{equation}
	\label{eqn:setup_BX}
	\mathcal{B}_\mathbf{X}
	\left( 
	Q_\epsilon	\right) 
	= 
	z_\epsilon
\end{equation}
has the unique solution of 
\begin{equation}
	Q_\epsilon =
	\mathcal{G}_\mathbf{X} (z_\epsilon) \,.
\end{equation}
Since we can explicitly compute $\mathcal{B}_\mathbf{X}$, then we can use (\ref{eqn:setup_BX}) to understand $\mathcal{G}_\mathbf{X} (z_\epsilon)$. Note that we can always use (\ref{eqn:blue_formula_expanded}) to expand $\mathcal{B}_{\mathbf{X}}(Q_\epsilon)$: if $Q_\epsilon$ has $g_\epsilon \in \R$ or $g_\epsilon^I \in \R$, then $Q_\epsilon \in \C$ and from Proposition \ref{prop:B_X_complex}, $\mathcal{B}_{\mathbf{X}}(Q_\epsilon) \in \C$, a contradiction to $\epsilon > 0$. 

By expanding the left-hand side of (\ref{eqn:setup_BX}) with (\ref{eqn:blue_formula_expanded}), the equation for the off-diagonal terms is:
\begin{equation}
	\label{eqn:lib}
	l(Q_\epsilon) i B_\epsilon = i \epsilon \,.
\end{equation}   
As $\epsilon \to 0^+$, then either $l_\epsilon = l(Q_\epsilon)$ is small or $B_\epsilon$ is small. If we consider a sequence of $\epsilon_k$ to $0$, we may extract a subsequence where either $l_{\epsilon_k} = l(Q_{\epsilon_k})$ converges to $0$ or $B_{\epsilon_k}$ converges to $0$. 

To understand the situation heuristically, assume that
\begin{equation}
	\lim\limits_{\epsilon \to 0^+} Q_\epsilon = Q = 
	\begin{pmatrix}
		A & i \bar{B} \\
		i B & \bar{A}
	\end{pmatrix} \,.
\end{equation}
Then, either 
\begin{equation}
	\lim\limits_{\epsilon \to 0^+} B_\epsilon = B = 0 \qquad \text{ or } \qquad \lim\limits_{\epsilon \to 0^+} l_\epsilon = 0 \,.
\end{equation}
Assume it is possible to interchange the limit and derivative in (\ref{eqn:brown_measure_limit}). Then, 
\begin{equation}
	\mu_X = \frac{1}{\pi} \frac{\partial}{\partial \bar{z}} F(Q)\,.
\end{equation}
If $\mathcal{B}_{\mathbf{X}}$ is defined and continuous at $Q$, then 
\begin{equation}
	z = \lim\limits_{\epsilon \to 0^+} z_\epsilon = \lim\limits_{\epsilon \to 0^+} \mathcal{B}_\mathbf{X}(Q_\epsilon) = \mathcal{B}_{\mathbf{X}}(Q) \,.
\end{equation}
Thus, 
\begin{equation}
	\mathcal{G}_{\mathbf{X}}(z) =  Q \,, 
\end{equation}
so that $B = 0$. As $G_X$ is holomorphic at $z$, then
\begin{equation}
	\frac{\partial}{\partial \bar{z}} F(Q) = \frac{\partial}{\partial \bar{z}} G_X(z) = 0 \,.
\end{equation}
Our first heuristic is that whenever $B = 0$ (i.e. not just in the continuous case), the Brown measure has zero density at $z$. Then, the Brown measure is supported on the set where $B \neq 0$. We also must consider the case where $Q_\epsilon$ has no limit as $\epsilon \to 0^+$. We state the heuristic formally:

\begin{heuristic}
	\label{heur:support}
	Let $p, q \in (M, \tau)$ where $p$ and $q$ are Hermitian and freely independent. Then, the support of the Brown measure of $X = p + i q$ is the closure of the set of $z \in \C$ such that: 
	\begin{equation}
		\lim\limits_{\epsilon \to 0^+} \mathcal{G}_{\mathbf{X}}(z_\epsilon) = 
		\begin{pmatrix}
			A & i \bar{B} \\
			i B & \bar{A}
		\end{pmatrix} 
	\end{equation}
	for some $B \neq 0$ or where the limit does not exist. 
\end{heuristic}

We can verify this heuristic when $p$ and $q$ have $2$ atoms that have equal weights (Theorem \ref{thm:B_not_0}).

For the second heuristic, first note that:
\begin{equation}
	\lim\limits_{\epsilon \to 0^+} B_\epsilon = B \neq 0 \Longrightarrow \lim\limits_{\epsilon \to 0^+} l_\epsilon = 0 \,.
\end{equation}
Hence, given the first heuristic, the support of the Brown measure should also be contained in the closure of the set of $z \in \C$ where $\lim_{\epsilon \to 0^+} l_\epsilon = 0$. But, it could be possible that this set has non-empty intersection with the set of $z \in \C$ where $B = 0$, where the measure has zero density. In particular, when $\mathcal{B}_\mathbf{X}$ is continuous at $Q$, then the intersection of these sets corresponds to a second order zero in the off-diagonal terms of (\ref{eqn:blue_formula_expanded}). Our second heuristic is that the intersection of these two sets is the boundary of the Brown measure: 

\begin{heuristic}
	\label{heur:boundary}
	Let $p, q \in (M, \tau)$ where $p$ and $q$ are Hermitian and freely independent. Then, the boundary of the Brown measure of $X = p + i q$ is the closure of the intersection of the set of $z \in \C$ such that: 
	\begin{equation}
		\lim\limits_{\epsilon \to 0^+} \mathcal{G}_{\mathbf{X}}(z_\epsilon) = Q \in \C  
	\end{equation}
	with the set of $z \in \C$ such that $\lim_{\epsilon \to 0^+} l_\epsilon = 0$.
\end{heuristic}

In (\cite{jarosz2004novel}, Section 6), the authors verify the boundary heuristic for some random matrix models. With some restrictions on $z$, we will see that the intersection of the sets corresponds to the solutions of a system of equations. When $p$ and $q$ have $2$ atoms, Proposition \ref{prop:l_0_B_0_compute} states that its closure is the boundary (i.e. support) of $\mu'$.

When $p$ and $q$ have an arbitrary number of atoms, Mathematica simulations suggest that the solution set to the system of equations contains the boundary of the Brown measure. When $p$ and $q$ have generic positions of atoms and weights, this solution set is an algebraic curve. We will provide an algorithm to produce a non-zero polynomial whose zero set contains this solution set. This is the content of Theorem \ref{thm:boundary_curve}. 

We can now explain the terminology of the ``Holomorphic Green's function'' for $G_X(z)$. Consider $z \in \C$ where the following limit exists: 
\begin{equation}
	\lim\limits_{\epsilon \to 0^+} \mathcal{G}_{\mathbf{X}}(z_\epsilon) =  	
	\begin{pmatrix}
		A(z) & i \bar{B(z)} \\
		i B(z) & \bar{A(z)}
	\end{pmatrix} \,.
\end{equation}
When $B(z) = 0$ and $\mathcal{B}_\mathbf{X}$ is defined and continuous at the limit point, recall that $A(z) = G_X(z)$. Hence, the \textbf{\textit{Holomorphic Green's function}} is the limit of the Quaternionic Green's function towards the complex plane where the resulting complex function is holomorphic. 

When $B(z) \neq 0$, then from Heuristic \ref{heur:support}, the Brown measure of $X$ is not zero in a neighborhood of $z$, so then $A(z)$ is not holomorphic in a neighborhood of $z$. This limit is the \textbf{\textit{Non-Holomorphic Green's function}}.

\begin{remark}
	We will use the following conventions throughout Sections \ref{sec:boundary} and \ref{sec:support}:
	\begin{itemize}
		\item Quantities that depend on a continuous limit as $\epsilon \to 0^+$ and are in the context of (\ref{eqn:setup_BX}) have a subscript $\epsilon$.
		\item Quantities that depend on a sequence $\epsilon_k \to 0^+$  and are in the context of (\ref{eqn:setup_BX}) have a subscript $\epsilon_k$.
		\item Quantities that are just general sequences with no implicit context have a subscript $k$.
		\item Unless specified otherwise, $p$, $q$, and $X$ are implied to mean the specific case of $p, q$ Hermitian and freely independent with 2 atoms, and $X = p + i q$.
		\item General sequences of quaternions $Q_k$ are implied to be in the domain of $\mathcal{B}_{\mathbf{X}}$.
	\end{itemize}
\end{remark}

\section{Computing $\mathcal{B}_\mathbf{X}$ for $X = p + i q$}  
\label{sec:B_X_compute}
In this section, we follow the outline for computing $\mathcal{B}_\mathbf{X}$ when $p$ and $q$ have 2 atoms.

Recall that we may assume that:
\begin{equation}
	\label{eqn:measures_defn}
	\begin{aligned}
		\mu_p & = a \delta_\alpha + (1 - a) \delta_{\alpha'} \\
		\mu_q &= b \delta_\beta + (1 - b) \delta_{\beta'} \, ,
	\end{aligned} 
\end{equation}
for $a, b \in (0, 1)$, $\alpha_n, \alpha_n', \beta_n, \beta_n' \in \R$, $\alpha \neq \alpha'$, and $\beta \neq \beta'$.

\subsection{Notation and conventions}
In this subsection, let us highlight some notation and conventions we will use for the rest of the paper.

The notation $\sqrt{z}$ will always denote the principal square root, defined on $\C \setminus (- \infty, 0)$ and taking $\sqrt{1} = + 1$.

We will also use the following definitions: 

\begin{definition}
	\label{def:d_p_d_q}
	Let $D_p, D_q$ be the following polynomials: 
	
	\begin{equation}
		\label{eqn:dp_defn}
		\begin{aligned}
			D_p(w) & = ((\alpha' - \alpha)w + (1 - 2a) )^2 + 4 a (1 - a) \\
			D_q(w) &= ((\beta' - \beta)w + (1 - 2b) )^2 + 4 b(1 - b) \,.
		\end{aligned}
	\end{equation}
\end{definition}

\begin{definition}
	\label{def:I_q_I_q}
	Let $I_p, I_q$ be the following subsets of $\C$: 
	\begin{equation}
		\label{eqn:Ip_defn}
		\begin{aligned}
			I_p 
			&= \left\lbrace  - \frac{1 - 2a}{\alpha' - \alpha}  + i y: \Abs{y} > \frac{2 \sqrt{a(1 - a)} }{\Abs{\alpha' -  \alpha}}   \right\rbrace \\
			I_q
			&= \left\lbrace  - \frac{1 - 2b}{\beta' - \beta}  + i y: \Abs{y} > \frac{2 \sqrt{b(1 - b)} }{\Abs{\beta' -  \beta}}   \right\rbrace  .
		\end{aligned}
	\end{equation} 
\end{definition}

\subsection{Auxiliary functions}
In this subsection, we will prove some lemmas about some related functions: 

\begin{lemma}
	\label{lem:cty1}
	For $w \in \C \setminus I_p$, let 
	\begin{equation}
		f(w) = \sqrt{ D_p(w) } \, .
	\end{equation}
	Then, $f$ is continuous on $\C \setminus I_p$ and analytic on $\C \setminus \bar{I_p}$.
	
	Let $\sgn(x) = x / \Abs{x}$ for $x \in \R \setminus \{0\}$.
	
	Let $w_0 \in I_p$.  
	
	If $w$ approaches $w_0$ from the right, then  
	\begin{equation}
		\lim_{w \to w_0^+} f(w) =  \sgn(\Im(w_0))  i \sqrt{ - D_p(w_0)  } \,.
	\end{equation}
	If $w$ approaches $w_0$ from the left, then 
	\begin{equation}
		\lim_{w \to w_0^-} f(w) = -  \sgn(\Im(w_0))  i \sqrt{ - D_p(w_0)  } \,.
	\end{equation}
\end{lemma}
\begin{proof}
	Note that $f$ is a composition of the principal square root with $D_p(w)$. $D_p(w)$ is well-defined and continuous everywhere, the principal square root is well-defined and continuous except at $\C \setminus (- \infty, 0)$ and analytic on $\C \setminus (- \infty, 0]$. The limit to some $w \in (- \infty, 0)$ from above is $+ \sqrt{- w} i$ and the limit from below is $- \sqrt{- w} i$.
	
	Hence, $f$ is defined and continuous except when: 
	\begin{equation}
		\begin{aligned}
			&  D_p(w) = ((\alpha' - \alpha) w + (1 - 2a))^2 + 4a(1 - a)  \in (- \infty, 0) \\
			& \iff  ((\alpha' - \alpha) w + (1 - 2a))^2   \in (- \infty, - 4 a (1 - a)) \\
			& \iff  (\alpha' - \alpha) w + (1 - 2a)  \in \left\lbrace i y: \Abs{y} > 2 \sqrt{a(1 - a)}  \right\rbrace  \\
			& \iff  (\alpha' - \alpha) w  \in \left\lbrace - (1 - 2a) + i y: \Abs{y} > 2 \sqrt{a(1 - a)}  \right\rbrace \\
			& \iff w \in \left\lbrace  - \frac{1 - 2a}{\alpha' - \alpha}  + i y: \Abs{y} > \frac{2 \sqrt{a(1 - a)} }{\Abs{\alpha' - \alpha}}   \right\rbrace \\
			& \iff w \in I_p \,.
		\end{aligned}
	\end{equation}
	For analyticity, 
	\begin{equation}
		D_p^{-1}(\{0\}) = \left\lbrace  - \frac{1 - 2a}{\alpha' - \alpha}  \pm  i \frac{2 \sqrt{a(1 - a)} }{\Abs{\alpha' -  \alpha}}   \right\rbrace \, ,
	\end{equation}
	so we just need to remove those two points from the domain to make $f$ analytic.
	
	For the last points on limits approaching $w_0$, it follows from casework considering cases of $\sgn(\alpha' - \alpha)$, $\sgn(\Im(w_0))$, and $w$ approaching $w_0$ from the left or right. Let us just verify the case when $\sgn(\alpha' - \alpha) = \sgn(\Im(w_0)) = + 1$ and $w$ approaches $w_0$ from the right. 
	
	Then, $(\alpha' - \alpha)w + (1 - 2a)$ approaches $(\alpha' - \alpha) w_0 + (1 - 2a) \in i \R$ from the right. Since $\Im((\alpha' - \alpha) w_0 + (1 - 2a)) > 0$, then $D_p(w)$ approaches $D_p(w_0) \in (- \infty, 0)$ from above. Finally, taking square roots, $f(w)$ approaches $i \sqrt{ - D_p(w_0)  }$ (from the right).
\end{proof}

\begin{lemma}
	\label{lem:cty2}
	For $w \in \C \setminus (I_p \cup \R)$, let 
	\begin{equation}
		f(w) = \frac{\sqrt{D_p(w)} - \sqrt{D_p(\bar{w})} }{w - \bar{w}} = \frac{\Im \left( \sqrt{D_p(w)}\right) }{\Im(w)} \, .
	\end{equation}
	The formula for $f$ is continuous on $\C \setminus (I_p \cup \R)$ and can be extended continuously to $\R$ by: 
	\begin{equation}
		f(t) = \frac{ (\alpha' - \alpha )((\alpha' - \alpha)t + (1 - 2a))   }{\sqrt{D_p(t)}}  \,.
	\end{equation}
	Let $w_0 \in I_p$. 
	
	If $w$ approaches $w_0$ from the right, then 
	\begin{equation}
		\lim\limits_{w \to w_0^+} f(w) = \frac{\sqrt{- D_p(w_0)}}{\Abs{\Im(w_0)}} \,.
	\end{equation}
	If $w$ approaches $w_0$ from the left, then
	\begin{equation}
		\lim\limits_{w \to w_0^-} f(w) = - \frac{\sqrt{- D_p(w_0)}}{\Abs{\Im(w_0)}} \,.
	\end{equation}
\end{lemma}
\begin{proof}
	From Lemma \ref{lem:cty1} and the fact that $I_p$ is fixed by complex conjugation, $\sqrt{D_p(w)}$, $\sqrt{D_p(\bar{w})}$ are defined and continuous except on $I_p$. Since the denominator vanishes on $\R$, the expression for $f$ is valid and continuous on $\C \setminus (I_p \cup \R)$.
	
	Consider $w \in \C \setminus (I_p \cup \R)$ approaching some $t \in \R$.
	
	As $D_p$ is a polynomial with real coefficients, from the property of the principal square root,
	\begin{equation}
		\sqrt{D_p(\bar{w})} = 		\sqrt{\bar{D_p(w)}} = 		\bar{\sqrt{D_p(w)}} \,.
	\end{equation}
	Hence,
	\begin{equation}
		\label{eqn:lem:cty2:1}
		f(w) = \frac{\Im \left( \sqrt{D_p(w)}\right) }{\Im(w)} \,.
	\end{equation}
	By taking real parts of $\left(  \sqrt{D_p(w)} \right) ^2 = D_p(w)$,
	\begin{equation}
		\begin{aligned}
			& 2 \, \Re\left( \sqrt{D_p(w)}\right)  \Im \left( \sqrt{D_p(w)} \right)  \\
			& = \Im(D_p(w)) \\
			&= \Im(((\alpha' - \alpha)w + (1 - 2a) )^2 + 4 a (1 - a) ) \\
			&= \Im(((\alpha' - \alpha)w + (1 - 2a) )^2) \\
			&= 2 \, \Re ((\alpha' - \alpha)w + (1 - 2a)) \, \Im ((\alpha' - \alpha)w + (1 - 2a)) \\
			&= 2 ( (\alpha' - \alpha) \Re (w) + (1 - 2a) ) ( \alpha' - \alpha ) \Im (w) \,.
		\end{aligned}
	\end{equation}
	Dividing both sides by $2 \, \Im(w)$,
	\begin{equation}
		\Re\left( \sqrt{D_p(w)}\right)  \frac{\Im \left( \sqrt{D_p(w)} \right)}{\Im (w)} = ( (\alpha' - \alpha) \Re (w) + (1 - 2a) ) ( \alpha' - \alpha ) \,.
	\end{equation}
	As $w \to t \in \R$, $D_p(w)$ converges to $D_p(t) \geq 4a(1 - a) > 0$ since $a \in (0, 1)$. Thus, it is possible to divide both sides by $\Re\left( \sqrt{D_p(w)}\right)$ for $w$ sufficiently close to $t$. Then,
	\begin{equation}
		\label{eqn:lem:cty2:3}
		f(w) = \frac{( \alpha' - \alpha )( (\alpha' - \alpha) \Re (w) + (1 - 2a) )}{\Re\left( \sqrt{D_p(w)}\right)} \,.
	\end{equation}
	Finally, taking $w \to t \in \R$, 
	\begin{equation}
		\lim\limits_{w \to t} f(w) =  \frac{( \alpha' - \alpha )( (\alpha' - \alpha) t + (1 - 2a) ) }{\sqrt{D_p(t)}} \,.
	\end{equation}
	This final expression is continuous on $\R$, so $f$ may be extended continuously to $\R$ by this expression. 
	
	For the final limits towards $I_p$, using (\ref{eqn:lem:cty2:1}) and Lemma \ref{lem:cty1}, 
	\begin{equation}
		\begin{aligned}
			\lim\limits_{w \to w_0^\pm} f(w) 
			& = \lim\limits_{w \to w_0^\pm} \frac{\Im \left( \sqrt{D_p(w)}\right) }{\Im(w)} \\
			& = \pm \frac{ \sgn(\Im(w_0))  \sqrt{ - D_p(w_0)  }}{\Im (w_0)}  \\
			& = \pm \frac{\sqrt{- D_p(w_0)}}{\Abs{\Im (w_0)}} \,.
		\end{aligned}
	\end{equation}
\end{proof}

\begin{lemma}
	\label{lem:cty3}
	For $w \in \C \setminus (I_p \cup \R)$, let
	\begin{equation}
		f(w) = \frac{\bar{w} \sqrt{D_p(w)} - w \sqrt{D_p(\bar{w})} }{w - \bar{w}}  = \frac{\Im (\bar{w} \sqrt{D_p(w)}  )}{\Im(w)}\, ,
	\end{equation}
	The formula for $f$ is continuous on $\C \setminus (I_p \cup \R)$ and can be extended continuously to $\R$ by: 
	\begin{equation}
		f(t) = \frac{- (1 - 2a)(\alpha' - \alpha) t - 1}{\sqrt{D_p(t)}} \,.
	\end{equation}
	Let $w_0 \in I_p$. 
	
	If $w$ approaches $w_0$ from the right, then 
	\begin{equation}
		\lim\limits_{w \to w_0^+} f(w) = - \frac{(1 - 2a) \sqrt{- D_p(w_0)}}{(\alpha' - \alpha) \Abs{ \Im (w_0)}} \,.
	\end{equation}
	If $w$ approaches $w_0$ from the left, then
	\begin{equation}
		\lim\limits_{w \to w_0^-} f(w) =   \frac{(1 - 2a) \sqrt{- D_p(w_0)  }}{(\alpha' - \alpha)  \Abs{\Im (w_0)}} \,.
	\end{equation}
	Hence, $f$ extends continuously to $I_p$ if and only if $a = 1/2$, in which case letting $f \equiv 0$ on $I_p$ makes $f$ continuous on $\C$.
\end{lemma}
\begin{proof}
	From Lemma \ref{lem:cty1} and the fact that $I_p$ is fixed by complex conjugation, $\sqrt{D_p(w)}$, $\sqrt{D_p(\bar{w})}$ are defined and continuous except on $I_p$. Since the denominator vanishes on $\R$, the expression for $f$ is valid and continuous on $\C \setminus (I_p \cup \R)$.
	
	Consider $w \in \C \setminus (I_p \cup \R)$ approaching some $t \in \R$.
	
	As $D_p$ is a polynomial with real coefficients, from the property of the principal square root, 
	\begin{equation}
		\sqrt{D_p(\bar{w})} = 		\sqrt{\bar{D_p(w)}} = 		\bar{\sqrt{D_p(w)}} \,.
	\end{equation}
	Hence,
	\begin{equation}
		f(w) = \frac{\Im (\bar{w} \sqrt{D_p(w)}  )}{\Im(w)} \,.
	\end{equation}
	Expanding the numerator, 
	\begin{equation}
		\begin{aligned}
			\Im \left( \bar{w} \sqrt{D_p(w)} \right) &= \Im(\bar{w}) \Re \left( \sqrt{D_p(w)} \right)  + \Re(\bar{w}) \Im\left( \sqrt{D_p(w)}\right) 
			\\ &= - \Im(w) \Re \left( \sqrt{D_p(w)} \right)  + \Re(w) \Im\left( \sqrt{D_p(w)}\right)  \,.
		\end{aligned}
	\end{equation}
	Dividing both sides by $\Im(w)$,
	\begin{equation}
		\label{eqn:lem:cty3:1}
		f(w) = - \Re \left( \sqrt{D_p(w)} \right) + \frac{\Re(w) \Im\left( \sqrt{D_p(w)}\right) }{\Im (w)} \,.
	\end{equation}
	The right-hand side contains the function in Lemma \ref{lem:cty2}. As this function has a continuous extension to $\R$, taking the limit as $w \to t \in \R$ and simplifying yields: 
	\begin{equation}
		\begin{aligned}
			\lim\limits_{w \to t} f(w) & = - \sqrt{D_p(t)} + \frac{t ( \alpha' - \alpha )( (\alpha' - \alpha) t + (1 - 2a) ) }{\sqrt{D_p(t)}} \\
			&= \frac{- (1 - 2a)(\alpha' - \alpha) t - 1}{\sqrt{D_p(t)}} \,.
		\end{aligned}
	\end{equation}
	This expression is continuous on $\R$, so $f$ may be extended continuously to $\R$ by this expression.
	
	For the limits as $w$ approaches $w_0 \in I_p$, note that
	\begin{equation}
		\lim\limits_{w \to w_0} \Re \left( \sqrt{D_p(w)} \right) = 0  \qquad \lim\limits_{w \to w_0} \Re(w) = - \frac{1 - 2a}{\alpha' - \alpha} \,.
	\end{equation}
	Applying Lemma \ref{lem:cty2} to (\ref{eqn:lem:cty3:1}) produces:
	\begin{equation}
		\begin{aligned}
			\lim\limits_{w \to w_0^{\pm}} f(w) & = \lim\limits_{w \to w_0^{\pm}} - \Re \left( \sqrt{D_p(w)} \right)  + \lim\limits_{w \to w_0^{\pm}} \Re(w)  \lim\limits_{w \to w_0^{\pm}} \frac{ \Im\left( \sqrt{D_p(w)}\right) }{\Im (w)} \\
			&= 0 - \frac{1 - 2a}{\alpha' - \alpha}   \lim\limits_{w \to w_0^{\pm}} \frac{ \Im\left( \sqrt{D_p(w)}\right) }{\Im (w)}
			\\ &= \mp \frac{(1 - 2a) \sqrt{- D_p(w_0)}}{(\alpha' - \alpha) \Abs{\Im(w_0)}} \,.
		\end{aligned}
	\end{equation}
	The left/right limits are negatives of each other, so they are equal if and only if both are zero. As $D_p(w_0) < 0$ for $w_0 \in I_p$, the limits are zero if and only if $a = 1/2$. In this situation, defining $f \equiv 0$ on $I_p$ makes $f$ continuous on all of $\C$.
\end{proof}

\subsection{Computing $\mathcal{B}_{\mathbf{X}}$}
In this subsection, we follow the steps described in Section \ref{sec:B_X_outline} to compute $\mathcal{B}_{\mathbf{X}}$. Since the computations for  $\{ B_p, \beta_p, \beta_p',  \mathcal{B}_{\mathbf{p}}\}$ and   $\{ B_q, \beta_q, \beta_q',  \mathcal{B}_{\mathbf{q}}\}$ are analogous, for these functions, we will just prove and state the results for $p$.

The first step is to compute $B_p$: 

\begin{proposition}
	\label{prop:complex_b_formula}
	For $w \in \C \setminus (I_p \cup \{0\})$, define
	\begin{equation}
		B_p(w) = \frac{\alpha + \alpha'}{2} + \frac{1 + \sqrt{D_p(w)}}{2w} \,.
	\end{equation}
	For $z$ in a neighborhood of infinity, $B_p(w) = G_p^{\brackets{-1}}(w) = z$. Further, $B_p$ is continuous on $\C \setminus (I_p \cup \{0\})$ and analytic on $\C \setminus (\bar{I_p} \cup \{0\})$. 
	
	Let $w_0 \in I_p$. 
	
	If $w$ approaches $w_0$ from the right, then 
	\begin{equation}
		\lim_{w \to w_0^+} B_p(w) =  \frac{\alpha + \alpha'}{2} + \frac{1 + \sgn( \Im(w_0))  i \sqrt{ - D_p(w_0)  }}{2 w_0} \,.
	\end{equation}
	If $w$ approaches $w_0$ from the left, then 
	\begin{equation}
		\lim_{w \to w_0^-} B_p(w) =  \frac{\alpha + \alpha'}{2} + \frac{1 - \sgn( \Im(w_0))  i \sqrt{ - D_p(w_0)  }}{2 w_0} \,.
	\end{equation}
	Additionally,
	\begin{equation}
		\lim\limits_{w \to 0} \Abs{B_p(w) } = \infty \,.
	\end{equation}
\end{proposition}
\begin{proof}
	As 
	\begin{equation}
		\mu_p = a \delta_\alpha + (1 - a) \delta_{\alpha'} \, ,
	\end{equation}
	then the complex Green's function for $p$ is given by: 
	\begin{equation}
		\label{eqn:complex_green}
		w = G_p(z) = \tau [(z - p)^{-1}] = \frac{a}{z - \alpha} + \frac{1 - a}{z - \alpha'} \,.
	\end{equation}
	Recall that in general, $G_p(z)$ is invertible in a neighborhood of $\infty$. To invert $G_p(z)$, note that (\ref{eqn:complex_green}) holds if and only if
	\begin{equation}
		\label{eqn:prop:complex_b_formula1}
		w (z - \alpha)(z - \alpha') = a(z - \alpha') + (1 - a)(z - \alpha) \,.
	\end{equation}
	This is true for $z \neq \alpha, \alpha'$ and the polynomial equation is not satisfied at $z = \alpha, \alpha'$ at any $w$. Hence, the solutions to the polynomial equation are solutions to $w = G_p(z)$ for all $z \in \C$.
	
	The polynomial equation can be rewritten:
	\begin{equation}
		w z^2 - ( w(\alpha + \alpha') + 1 )z + (w \alpha \alpha' + \alpha \alpha' + (1 - a) \alpha) = 0 \,.
	\end{equation}
	Fixing a $w$ and solving for $z$, then from the quadratic formula and simplifying,
	\begin{equation}
		\label{eqn:bp_inverse_quad_formula}
		\begin{aligned}
			z & = \frac{\alpha + \alpha'}{2} + \frac{1 \pm \sqrt{ ((\alpha - \alpha')w + (1 - 2a) )^2 + 4 a(1 - a)    }}{2w}  \\
			& =  \frac{\alpha + \alpha'}{2} + \frac{1 \pm \sqrt{ D_p(w)   }}{2w}\, .
		\end{aligned}
	\end{equation}
	To determine the sign of the square root of the inverse defined for $z$ in a neighborhood of infinity, note that for $\lim_{\Abs{z} \to \infty }G_p(z) = 0$. When the sign is $-$, the quotient involving the square root is $1/2$ times the derivative of $\sqrt{D_p(w)}$ at $w = 0$, which is finite as $w \to 0$. Hence, the inverse of $G_p(z)$ for $z$ in a neighborhood of $\infty$ must take the $+$ sign. Hence, 
	\begin{equation}
		B_p(w) = \frac{\alpha + \alpha'}{2} + \frac{1 + \sqrt{D_p(w)}}{2w} 
	\end{equation}
	and
	\begin{equation}
		\lim\limits_{w \to 0} \Abs{B_p(w) } = \infty \,.
	\end{equation}
	From Lemma \ref{lem:cty1}, the formula in (\ref{eqn:bp_inverse_quad_formula}) is defined and continuous on $\C \setminus (I_p \cup \{0\})$ and analytic on $\C \setminus (\bar{I_p} \cup \{0\})$. As this set is connected and the formula forms an inverse of $G_p(z)$ for $z$ in a neighborhood of infinity, then this formula is an inverse for $G_p(z)$ on its domain. Hence,
	\begin{equation}
		B_p(w) = \frac{\alpha + \alpha'}{2} + \frac{1 + \sqrt{D_p(w)}}{2w}  \,.
	\end{equation}
	From Lemma \ref{lem:cty1}, $B_p$ and the fact that $B_p$ is an inverse for $G_p$ in a neighborhood of $w = 0$, then $B_p$ is defined and continuous on $\C \setminus (I_p \cup \{0\})$ and analytic on $\C \setminus (\bar{I_p} \cup \{0\})$. The left/right limits of $B_p$ can also be computed using Lemma \ref{lem:cty1}.	
\end{proof}

Next, we compute $\beta_p$, $\beta_p'$ and determined the domains where they are well-defined and continuous. The results for $\beta_q$, $\beta_q'$ are completely analogous.

\begin{proposition}
	\label{prop:beta_formula}
	For $w \in \C \setminus (I_p \cup  \R)$, 
	\begin{equation}
		\beta_p(g) = \frac{\alpha + \alpha'}{2} + \frac{1}{2} \frac{\sqrt{D_p(g)} - \sqrt{D_p(\bar{g})} }{g - \bar{g}} \,.
	\end{equation}
	$\beta_p$ can be extended to be continuously to $\R$ by: 
	\begin{equation}
		\beta_p(t) = \frac{\alpha + \alpha'}{2} + \frac{1}{2} \frac{ (\alpha' - \alpha )((\alpha' - \alpha)t + (1 - 2a))   }{\sqrt{D_p(t)}} \,.
	\end{equation}
	Let $g_0 \in I_p$. 
	
	If $g$ approaches $g_0$ from the right, then 
	\begin{equation}
		\lim\limits_{g \to g_0^+} \beta_p(g) =  \frac{\alpha + \alpha'}{2} +  \frac{\sqrt{- D_p(g_0)}}{2 \Abs{ \Im(g_0) }} \,.
	\end{equation}
	If $g$ approaches $g_0$ from the left, then
	\begin{equation}
		\lim\limits_{g \to g_0^-} \beta_p(g) = \frac{\alpha + \alpha'}{2} -    \frac{\sqrt{- D_p(g_0)}}{2 \Abs{\Im (g_0)}} \,.
	\end{equation}
\end{proposition}
\begin{proof}
	Computation using (\ref{eqn:blue_beta}) and Proposition (\ref{prop:complex_b_formula}) yields
	\begin{equation}
		\beta_p(g) = \frac{\alpha + \alpha'}{2} + \frac{1}{2} \frac{\sqrt{D_p(g)} - \sqrt{D_p(\bar{g})} }{g - \bar{g}} \,.
	\end{equation}
	Everything else follows from the fact that 
	\begin{equation}
		\beta_p(g) = \frac{\alpha + \alpha'}{2} + \frac{1}{2} f(g) \, ,
	\end{equation}
	where $f$ is from Lemma \ref{lem:cty2}
\end{proof}

\begin{proposition}
	\label{prop:beta'_formula}
	For $w \in \C \setminus (I_p \cup \R)$, 
	\begin{equation}
		\beta_p'(g) = \frac{1}{2 \Abs{g}^2} \left( - 1 + \frac{\bar{g} \sqrt{D_p(g)} - g \sqrt{D_p(\bar{g})} }{g - \bar{g}} \right) .
	\end{equation} 
	$\beta_p'$ can be extended continuously to $\R \setminus \{0\}$ by: 
	\begin{equation}
		\beta_p'(t) = \frac{1}{2 t^2} \left( - 1 + \frac{- (1 - 2a)(\alpha' - \alpha) t - 1}{\sqrt{D_p(t)}} \right) .
	\end{equation} 	
	At $g = 0$, 
	\begin{equation}
		\lim\limits_{g \to 0} \beta_p'(g) = - \infty \,.
	\end{equation}
	Let $g_0 \in I_p$. 
	
	If $g$ approaches $g_0$ from the right, then 
	\begin{equation}
		\lim\limits_{g \to g_0^+} \beta_p'(g) =  \frac{1}{2 \Abs{g_0}^2} \left( -1 - \frac{(1 - 2a) \sqrt{- D_p(g_0)}}{(\alpha' - \alpha) \Abs{\Im(g_0)}}\right)   .
	\end{equation}
	If $g$ approaches $g_0$ from the left, then
	\begin{equation}
		\lim\limits_{g \to g_0^-} \beta_p'(g) = \frac{1}{2 \Abs{g_0}^2} \left(-1 + \frac{(1 - 2a) \sqrt{- D_p(g_0)  }}{(\alpha' - \alpha)  \Abs{\Im(g_0)}} \right) .
	\end{equation}
	$\beta_p'$ can be extended continuously to $I_p$ if and only if $a = 1/2$, in which case on $I_p$, 
	\begin{equation}
		\beta_p'(t) = - \frac{1}{2 \Abs{g}^2} \,.
	\end{equation}
\end{proposition}
\begin{proof}
	Computation using (\ref{eqn:blue_beta}) and Proposition (\ref{prop:complex_b_formula}) yields
	\begin{equation}
		\beta_p'(g) = \frac{1}{2 \Abs{g}^2} \left( - 1 + \frac{\bar{g} \sqrt{D_p(g)} - g \sqrt{D_p(\bar{g})} }{g - \bar{g}} \right) . 
	\end{equation}
	Most of the claims for the domains where $\beta'$ is well-defined and continuous follows from the fact that 
	\begin{equation}
		\beta_p'(g) = \frac{1}{2 \Abs{g}^2} \left( - 1 + f(g) \right) \, ,
	\end{equation}
	where $f$ is from Lemma \ref{lem:cty3}.
	
	The only extra complication is $g = 0$. For this, we use Lemma \ref{lem:cty3} to see that the limit of the term of $\beta_p'(g)$ inside the parentheses is:
	\begin{equation}
		\begin{aligned}
			\lim\limits_{g \to 0} - 1 + \frac{\bar{g} \sqrt{D_p(g)} - g \sqrt{D_p(\bar{g})} }{g - \bar{g}}
			& = -1 + \restr{\frac{- (1 - 2a)(\alpha' - \alpha) t - 1}{\sqrt{D_p(t)}}}{t = 0} \\
			& = -1 + \frac{-1}{\sqrt{D_p(0)}} \,.
		\end{aligned}
	\end{equation}
	This final quantity is negative, so $\beta_p'(g) \to - \infty$ as $g \to 0$. Hence, $g = 0$ is not in the domains of definition and continuity for $\beta_p$. 
	
	The left and right limits towards some $g_0 \in I_p$ follow directly from Lemma \ref{lem:cty3}.
\end{proof}

Next, we compute $\mathcal{B}_{\mathbf{p}}$ and $\mathcal{B}_{i \mathbf{q}}$ and determine the domains where they are well-defined and continuous. In particular, these domains agree with the ones from Propositions \ref{prop:B_H} and \ref{prop:B_cX}.

When we refer to $\mathcal{B}_{\mathbf{p}}$ (resp $\mathcal{B}_{i \mathbf{q}}$) being continuous for $g \in S$ for some $S \subset \C$, it meant to be understood as $\mathcal{B}_{\mathbf{p}}$ (resp $\mathcal{B}_{i \mathbf{q}}$) being continuous for quaternions $Q$ where its eigenvalue $g \in S$. The well-definedness and continuity of $\mathcal{B}_{\mathbf{p}}(Q)$ depends mainly on the well-defined and continuous of $\beta_p, \beta_p'$ at $g$. This is summarized in the following Proposition:

\begin{proposition}
	\label{prop:bp}
	For $Q \in \mathbb{H}$ such that $g \in \C \setminus (I_p \cup \{0\})$,  
	\begin{equation}
		\mathcal{B}_{\mathbf{p}}(Q) = 
		\begin{cases}
			B_p(g_0)  & Q = g_0 \in \C \\
			\beta_p(g)  - \beta_p'(g) Q^* & Q \not \in \R \,,
		\end{cases} 
	\end{equation}
	where 
	\begin{equation}
		\begin{aligned}
			B_p(g_0) &= \frac{\alpha + \alpha'}{2} + \frac{1 + \sqrt{D_p(g_0)}}{2g_0}  \\
			\beta_p(g) &= 
			\begin{dcases}
				\frac{\alpha + \alpha'}{2} + \frac{1}{2} \frac{\sqrt{D_p(g)} - \sqrt{D_p(\bar{g})} }{g - \bar{g}} & g \in \C \setminus (I_p \cup \R) \\ 
				\frac{\alpha + \alpha'}{2} + \frac{1}{2} \frac{ (\alpha' - \alpha )((\alpha' - \alpha)t + (1 - 2a))   }{\sqrt{D_p(t)}}  & g = t \in \R
			\end{dcases} \\
			\beta_p'(g) &= 
			\begin{dcases}
				\frac{1}{2 \Abs{g}^2} \left( - 1 + \frac{\bar{g} \sqrt{D_p(g)} - g \sqrt{D_p(\bar{g})} }{g - \bar{g}} \right)  & g \in \C \setminus (I_p \cup \R) \\
				\frac{1}{2 t^2} \left( - 1 + \frac{- (1 - 2a)(\alpha' - \alpha) t - 1}{\sqrt{D_p(t)}} \right) & g = t \in \R \setminus \{0\} 	\,.			
			\end{dcases}
		\end{aligned}
	\end{equation} 
	Then, $\mathcal{B}_{\mathbf{p}}$ is continuous on this subset of $\mathbb{H}$ and in no larger domain, and
	\begin{equation}
		\lim\limits_{Q \to 0} \Abs{\mathcal{B}_{\mathbf{p}}(Q)} = \infty \,.
	\end{equation}
\end{proposition}
\begin{proof}
	The formulas follow from (\ref{eqn:blue_hermitian}), (\ref{eqn:blue_beta}) and Propositions \ref{prop:beta_formula} and \ref{prop:beta'_formula}.
	
	For the continuity points of $\mathcal{B}_{\mathbf{p}}$, it follows from the formulas for $\beta_p$ and $\beta_p'$ that $\mathcal{B}_{\mathbf{p}}$ is well-defined and continuous for $g \in \C \setminus (I_p \cup \R)$. From Propositions \ref{prop:beta_formula} and \ref{prop:beta'_formula}, $\mathcal{B}_{\mathbf{p}}$ does have a limit as $g$ approaches a non-zero real number. But, $\mathcal{B}_{\mathbf{p}}$ is already defined on $\R \setminus \{0\}$. It is a straightforward computation to verify that: 
	\begin{equation}
		\lim\limits_{Q \to g_0 \in \R \setminus \{0\}  } \mathcal{B}_{\mathbf{p}}(Q) = \lim\limits_{g \to g_0} \beta_p(g) - \beta_p'(g) Q^* = B_p(g_0) \,,
	\end{equation} 
	so that $\mathcal{B}_{\mathbf{p}}$ is continuous on $\C \setminus (I_p \cup \{0\})$. 
	
	To analyze $\mathcal{B}_{\mathbf{p}}(Q)$ as $Q \to 0$,  recall from Proposition \ref{prop:B_H} that the eigenvalues of $\mathcal{B}_{\mathbf{p}}(Q)$ are $B_p(g), B_p(\bar{g})$. As $Q \to 0$, $g \to 0$. From Proposition \ref{prop:complex_b_formula}, as $g \to 0$, $\Abs{B_p(g)} \to \infty$. Hence, the eigenvalues of $\mathcal{B}_{\mathbf{p}}(Q)$ diverge to $\infty$. Hence, as $Q \to 0$, $\Abs{\mathcal{B}_{\mathbf{p}}(Q)} \to \infty$.
	
	To analyze the limit as $Q$ approaches $Q_0$ where $g$ tends to $ g_0 \in I_p$, it suffices to consider the difference between the left and right limits as $g$ approaches $g_0 \in I_p$ for $\mathcal{B}_{\mathbf{p}}$:
	\begin{equation}
		\lim\limits_{g \to g_0^+}  \mathcal{B}_{\mathbf{p}}(Q) - \lim\limits_{g \to g_0^-} \mathcal{B}_{\mathbf{p}}(Q) \,.
	\end{equation}
	From Propositions \ref{prop:B_H}, \ref{prop:beta_formula}, and \ref{prop:beta'_formula}, the diagonal term of this difference is: 
	\begin{equation}
		\label{eqn:prop:bp}
		\begin{aligned}
			\frac{\sqrt{- D_p(g_0)}}{\Abs{\Im(g_0)}} + \frac{(1 - 2a) \sqrt{- D_p(g_0)  }}{(\alpha' - \alpha)  \Abs{g_0}^2 \Abs{\Im(g_0)}} \bar{A}
			& = \frac{\sqrt{- D_p(g_0)}}{\Abs{\Im(g_0)}} \left(1 + \frac{1 - 2a}{\alpha' - \alpha} \frac{A}{\Abs{g_0}^2}  \right) \\
			&= \frac{\sqrt{- D_p(g_0)}}{\Abs{\Im(g_0)}} \left(1 - \frac{\Re(g_0)}{\Abs{g_0}} \frac{A}{\Abs{g_0}}  \right) \,.
		\end{aligned}
	\end{equation}
	The factor in front of the last term is non-zero, so it suffices to analyze when the term inside the parentheses is zero. It is clear that $\Re(g_0) \leq \Abs{g_0}$, with equality occurring only when $g_0$ is real. Recall that $g_0 = \sqrt{\Abs{A}^2 + \Abs{B}^2}$, so $\Abs{A} \leq \Abs{g_0}$. Hence, the term inside the parentheses can be zero only when $g_0$ is real. As $I_p \cap \R = \emptyset$, this quantity is non-zero, so $\mathcal{B}_{\mathbf{p}}$ is discontinuous at $I_p$. 	
\end{proof}

The results for $\mathcal{B}_{i \mathbf{q}}$ follow immediately from Propositions \ref{prop:B_cX} and \ref{prop:bp}:

\begin{proposition}
	\label{prop:biq}
	For $Q \in \mathbb{H}$ such that $g^I \in \C \setminus (I_p \cup \{0\})$,  
	\begin{equation}
		\mathcal{B}_{i \mathbf{q}}(Q) = 
		\begin{cases}
			i B_q(i g_0) & Q i = i g_0 \in \C \\
			i \beta_q(g^I) - \beta_q'(g^I) Q^* & Q i \not \in \R \,,
		\end{cases}
	\end{equation}
	where 
	\begin{equation}
		\begin{aligned}
			B_q(i g_0) &= \frac{\beta + \beta'}{2} + \frac{1 + \sqrt{D_q(i g_0)}}{2 ig_0} \,, \\
			\beta_q(g^I) &= 
			\begin{dcases}
				\frac{\beta + \beta'}{2} + \frac{1}{2} \frac{\sqrt{D_q(g^I)} - \sqrt{D_q(\bar{g^I})} }{g^I - \bar{g^I}} & g^I \in \C \setminus (I_p \cup \R) \\ 
				\frac{\beta + \beta'}{2} + \frac{1}{2} \frac{ (\beta' - \beta )((\beta' - \beta)t + (1 - 2b))   }{\sqrt{D_q(t)}}  & g^I = t \in \R
			\end{dcases} \\
			\beta_q'(g^I) &= 
			\begin{dcases}
				\frac{1}{2 \Abs{g^I}^2} \left( - 1 + \frac{\bar{g^I} \sqrt{D_q(g^I)} - g^I \sqrt{D_q(\bar{g^I})} }{g^I - \bar{g^I}} \right)  & g^I \in \C \setminus (I_q \cup \R) \\
				\frac{1}{2 t^2} \left( - 1 + \frac{- (1 - 2b)(\beta' - \beta) t - 1}{\sqrt{D_q(t)}} \right) & g^I = t \in \R \setminus \{0\} \,. 				
			\end{dcases}
		\end{aligned}
	\end{equation} 
	Then, $\mathcal{B}_{i \mathbf{q}}$ is continuous on this subset of $\mathbb{H}$ and in no larger domain, and 
	\begin{equation}
		\lim\limits_{Q \to 0} \Abs{\mathcal{B}_{i \mathbf{q}}(Q)} = \infty \,.
	\end{equation}
\end{proposition}

All that is left is to determine a formula for $\mathcal{B}_{\mathbf{X}}$ and determine the domains where it is defined and continuous. 

First, let us state the formula and domains of well-definedness and continuity of the complex Green's function $B_X$: 

\begin{proposition}
	\label{prop:complex_B_X}
	For $w \in \C$ such that $w \not \in I_p$ and $i w \not \in I_q$ and $w \neq 0$,
	\begin{equation}
		B_X(w) = \frac{\alpha + \alpha'}{2} + i \frac{\beta + \beta'}{2} + \frac{\sqrt{D_p(w)}  + \sqrt{D_q(i w)} }{2 w}
	\end{equation}
	is analytic. Further, $B_X$ cannot be continuously extended to any larger domain.
	Additionally,
	\begin{equation}
		\lim\limits_{w \to 0} w B_X(w) = 1\,.
	\end{equation}
\end{proposition}
\begin{proof}
	From the addition law for the Quaternionic $R$-transform, 
	\begin{equation}
		\mathcal{B}_{\mathbf{X}}(Q) = \mathcal{B}_{\mathbf{p}}(Q) + \mathcal{B}_{i \mathbf{q}}(Q) - Q^{-1} \,.
	\end{equation}	
	Using Proposition \ref{prop:B_X_complex} to restrict this formula $w \in \C$ (and taking note of the domains of $\mathcal{B}_{\mathbf{p}}$ and $\mathcal{B}_{i \mathbf{q}}$ from Propositions \ref{prop:bp} and \ref{prop:biq}), then 
	\begin{equation}
		B_X(w) = B_{p}(w) + i B_{q}(i w) - w^{-1} \,,
	\end{equation}
	Then, the formula follows from Proposition \ref{prop:complex_b_formula}.
	
	For continuity and analyticity, we need to check what happens if $w \in I_p$ and $i w \in I_q$, in case the discontinuities of $\sqrt{D_p(w)}$ and $\sqrt{D_q(i w )}$ cancel out. 
	
	We may choose different sequences approaching $w$ such that all four combinations of limits from applying Lemma \ref{lem:cty1} to $p$ and $q$ are possible: 
	\begin{equation}
		\frac{\alpha + \alpha'}{2} + i \frac{\beta + \beta'}{2} \pm i \frac{\sqrt{ - D_p(w)  }}{2 w} \pm i   \frac{ \sqrt{ - D_q( i w)  }}{2w} \,.
	\end{equation}
	Setting pairs of these expressions equal to each other, then for $B_X$ to be continuous at $w$, 
	\begin{equation}
		- D_p(w) = - D_q(i w) = 0 \,.
	\end{equation}
	But this is impossible for $w \in I_p$ and $i w \in I_q$. 
	
	Hence, when $w \in I_p$ and $i w \in I_q$, $B_X$ is discontinuous at $w$.
	
	The limit follows from the fact that $B_X$ is defined in a punctured neighborhood centered at $0$ and	
	\begin{equation}
		\begin{aligned}
			\lim\limits_{\Abs{z} \to \infty} G_X(z) & = 0 \\
			\lim\limits_{\Abs{z} \to \infty} z G_X(z) & = 1 \,.
		\end{aligned}
	\end{equation}
\end{proof}

The addition law (\ref{eqn:bX_free}) and Propositions \ref{prop:bp} and \ref{prop:biq} imply that  $\mathcal{B}_{\mathbf{X}}$ is defined and continuous where  $g \not \in I_p \cup \{0\}$ and $g^I \not \in I_q \cup \{0\}$. In order to show that this domain is the maximal domain of definition for $\mathcal{B}_{\mathbf{X}}$, we need to analyze the following limits more carefully: 

\begin{enumerate}
	\item $g \to 0$ (and hence $g^I \to 0$).
	\item $g  \to g_0 \in I_p$ and $g^I \to g_0^I \in I_q$ (in case the discontinuities of $\mathcal{B}_{\mathbf{p}}$ and $\mathcal{B}_{i \mathbf{q}}$ to cancel each other out).
\end{enumerate}
We will use (\ref{eqn:blue_formula_expanded}) when $g^I, g \not \in \R$. In particular, we need to first analyze $l = l(g, g^I)$:

\begin{proposition}
	\label{prop:l}
	Let $S_p, S_q$ be the following subsets of $\C$:
	\begin{equation}
		\begin{aligned}
			S_p & = 
			\begin{dcases}
				I_p \cup \{0\} & a \neq \frac{1}{2}\\
				\{0\} & a = \frac{1}{2}
			\end{dcases}		\\
			S_q & =
			\begin{dcases}
				I_q \cup  \{0\}  & b \neq \frac{1}{2} \\
				\{0\} & b = \frac{1}{2} \,.
			\end{dcases}
		\end{aligned}
	\end{equation}
	For $Q \in \mathbb{H}$ such that $g \in \C \setminus (S_p \cup \R)$ or $g^I \in \C \setminus (S_q \cup \R)$,
	\begin{equation}
		\label{eqn:l}
		l(Q) = l (g, g^I) = \frac{1}{2 \Abs{g}^2} \left( \frac{\bar{g} \sqrt{D_p(g)} - g \sqrt{D_p(\bar{g})} }{g - \bar{g}} + \frac{\bar{g^I} \sqrt{D_q(g^I)} - g^I \sqrt{D_q(\bar{g^I})} }{g^I - \bar{g^I}} \right) \,.
	\end{equation}
	Then, $l$ can be extended to be continuous at $Q \in \mathbb{H}$ such that $g \in \C \setminus S_p$ or $g^I \in \C \setminus S_q$ and in no larger domain and
	\begin{equation}
		\lim\limits_{g, g^I \to 0}  \Abs{g}^2  l (g, g^I) = - 2 \,.
	\end{equation}
\end{proposition}
\begin{proof}
	For $g \in \C \setminus (I_p \cup \R)$ and $g^I \in \C \setminus (I_p \cup \R)$, using the definition of $l(Q)$ and Proposition \ref{prop:beta'_formula} yields (\ref{eqn:l}).
	
	From Lemma \ref{lem:cty3}, the formula for $l$ can be continuously extended to at least when $g \in \C \setminus S_p$ and $g^I \in \C \setminus S_q$. If exactly one of $g \in S_p$ or $g^I \in S_q$ then the limit also will not exist, as $l$ will be a sum of two terms, one where the limit exists and one where the limit doesn't.
	
	Thus, we just need to consider what happens when:
	
	\begin{enumerate}
		\item $g \to 0$ (and hence $g^I \to 0$).
		\item $g \to g_0 \in S_p \setminus \{0\}$ and $g^I \to g_0^I \in S_q  \setminus \{0\}$. Note that this only makes sense when $a \neq 1/2$ and $b \neq 1/2$, in which case we may assume $g_0 \in I_p$ and $g_0^I \in I_q$.	
	\end{enumerate}
	
	For the first case when $g \to 0$ (and $g^I \to 0$), using Lemma \ref{lem:cty3}, the term inside the parentheses in (\ref{eqn:l}) tends towards
	\begin{equation}
		- \frac{1}{\sqrt{D_p(0)}} - \frac{1}{\sqrt{D_q(0)}} = -1 - 1 = -2 \,.
	\end{equation}
	Hence, as $g \to 0$, $\Abs{g}^2 l(g, g^I) \to -2 $, so $l(g, g^I) \to - \infty$.
	
	For the second case when $g_0 \in I_p$ and $g_0^I \in I_q$,  recall that 
	\begin{equation}
		\begin{aligned}
			g_0 & = x_0 + i \sqrt{x_1^2 + x_2^2 + x_3^2} \\
			g_0^I & = - x_3 + i \sqrt{x_2^2 + x_1^2 + x_0^2}	\, ,
		\end{aligned}
	\end{equation}
	By choosing if $x_0$ is approached from left or right and if $x_3$ is approached from left or right, all four combinations of limits from applying Lemma \ref{lem:cty3} to $p$ and $q$ are possible: 
	\begin{equation}
		\pm \frac{(1 - 2a) \sqrt{- D_p(g)}}{(\alpha' - \alpha) \Abs{\Im(g)}} \pm  \frac{(1 - 2b) \sqrt{- D_q(g^I)}}{(\beta' - \beta) \Abs{\Im(g^I)}} \,.
	\end{equation}
	By setting pairs of these expressions equal to each other, then for the limit to exist, 
	\begin{equation}
		\frac{(1 - 2a) \sqrt{- D_p(g)}}{(\alpha' - \alpha) \Abs{\Im(g)}} = \frac{(1 - 2b) \sqrt{- D_q(g^I)}}{(\beta' - \beta) \Abs{\Im(g^I)}} = 0 \,.
	\end{equation}
	However, this can only happen if $a = b = 1/2$, a contradiction. 
\end{proof}

We conclude this section by proving the result for $\mathcal{B}_{\mathbf{X}}(Q)$: 

\begin{theorem}
	\label{thm:BX_cty}
	For $Q \in \mathbb{H}$ such that $g \not \in I_p \cup \{0\} $ and $g^I \not \in I_q \cup \{0\}$, 
	\begin{equation}
		\mathcal{B}_{\mathbf{X}}(Q) =
		\mathcal{B}_{\mathbf{X}}\left( 
		\begin{pmatrix}
			A & i \bar{B} \\
			i B & \bar{A}
		\end{pmatrix}
		\right)  
		= 
		\begin{pmatrix}
			k + i k' - l \bar{A} & l i \bar{B} \\
			l i B & k - i k' - l A
		\end{pmatrix}
		\, ,
	\end{equation}
	where
	\begin{equation}
		\begin{aligned}
			k & = \beta_p(g) \\
			k' &= \beta_q(g^I) \\
			l &= \beta_p'(g) + \beta_q'(g^I) + \frac{1}{\det Q} \, .
		\end{aligned}
	\end{equation}
	and $\beta_p(g), \beta_q(g)$ are defined in Proposition \ref{prop:bp}, and $l$ is defined in Proposition \ref{prop:l}.
	
	$\mathcal{B}_{\mathbf{X}}$ is continuous on this subset of $\mathbb{H}$ and in no larger domain, and 
	\begin{equation}
		\lim\limits_{Q \to 0} \Abs{\mathcal{B}_{\mathbf{X}}(Q)} = \infty \,.
	\end{equation}
\end{theorem}

\begin{proof}
	Recall that all that remains is to examine the limit of $\mathcal{B}_{\mathbf{X}}(Q)$ when $g, g^I$ approaches the following limits:
	
	\begin{enumerate}
		\item $g \to 0$ (and hence $g^I \to 0$).
		\item $g  \to g_0 \in I_p$ and $g^I \to g_0^I \in I_q$ (in case the discontinuities of $\mathcal{B}_{\mathbf{p}}$ and $\mathcal{B}_{i \mathbf{q}}$ to cancel each other out).
	\end{enumerate}
	
	For the first limit, when $g \in \R$ or $g^I \in \R$, then either $Q \in \R$ or $Q \in i \R$, and we can apply Proposition \ref{prop:complex_B_X} to see that $\Abs{\mathcal{B}_{\mathbf{X}}(Q)} \to \infty$. When $g, g^I \not \in \R$, then we can use (\ref{eqn:blue_formula_expanded}). From Proposition \ref{prop:beta_formula}, $k$ and $k'$ have limits as $g, g^I \to 0$. Hence, it suffices to consider the limit: 
	\begin{equation}
		\lim\limits_{Q \to 0} 
		\begin{pmatrix}
			l \bar{A} & -l i \bar{B} \\
			-l i B & l A
		\end{pmatrix}
		= \lim\limits_{Q \to 0} l \begin{pmatrix}
			\bar{A} & - i \bar{B} \\
			- i B & A
		\end{pmatrix}
		= \lim\limits_{Q \to 0}  l Q^* .
	\end{equation}
	Taking quaternionic norms and using Proposition \ref{prop:l},
	\begin{equation}
		\lim\limits_{Q \to 0} \Abs{l Q^* }
		= \lim\limits_{Q \to 0} \Abs{l \Abs{Q}^2   \frac{Q^*}{\Abs{Q}^2} }
		= \lim\limits_{Q \to 0} \Abs{l \Abs{Q}^2}  \Abs{ \frac{Q^*}{\Abs{Q}^2} }
		= \lim\limits_{Q \to 0}  \frac{2}{\Abs{Q}}
		= \infty \,.
	\end{equation}
	Hence, $\lim\limits_{Q \to 0} \Abs{\mathcal{B}_{\mathbf{X}}(Q)} = \infty$.
	
	Finally, consider if $Q$ approaches $Q_0$ where $g_0 \in I_p$ and $g_0^I \in I_q$. 
	
	From the addition law and Proposition \ref{prop:B_cX},
	\begin{equation}
		\mathcal{B}_{\mathbf{X}}(Q) = \mathcal{B}_{\mathbf{p}}(Q) + i \mathcal{B}_{ \mathbf{q}}(Q i) - Q^{-1} \,.
	\end{equation}
	Recall that: 
	\begin{equation}
		\begin{aligned}
			g_0 & = x_0 + i \sqrt{x_1^2 + x_2^2 + x_3^2} \\
			g_0^I & = - x_3 + i \sqrt{x_2^2 + x_1^2 + x_0^2}	\, ,
		\end{aligned}
	\end{equation}
	By choosing if $x_0$ is approached from left or right and if $x_3$ is approached from left or right, all four combinations of left/right limits are possible for choices as $Q$ approaches $Q_0$:
	\begin{equation}
		\lim\limits_{Q \to Q_0} \mathcal{B}_{\mathbf{X}}(Q) =  \lim\limits_{g \to g_0^{\pm}}\mathcal{B}_{\mathbf{p}}(Q) + i \lim\limits_{g^I \to g_0^{I \pm}} \mathcal{B}_{\mathbf{q}}(Q i) - Q_0^{-1} .
	\end{equation}
	By setting pairs of these expressions equal to each other, then for the limit to exist: 
	\begin{equation}
		\begin{aligned}
			\lim\limits_{g \to g_0^+}  \mathcal{B}_{\mathbf{p}}(Q) - \lim\limits_{g \to g_0^-} \mathcal{B}_{\mathbf{p}}(Q) &= 0 \\
			\lim\limits_{g^I \to g_0^{I +}}  \mathcal{B}_{\mathbf{q}}(Q i) - \lim\limits_{g \to g_0^{I -}} \mathcal{B}_{\mathbf{q}}(Q i) &= 0 \,.
		\end{aligned}
	\end{equation}
	From (\ref{eqn:prop:bp}) in the proof of Proposition \ref{prop:bp}, this cannot happen when $g_0 \in I_p$ or $g_0^I \in I_q$. Hence, $\mathcal{B}_{\mathbf{X}}$ is discontinuous for $g \in I_p \cup \{0\}$ or $g^I \in I_q \cup \{0\}$.
\end{proof}

\section{Boundary of the Brown measure}
\label{sec:boundary}
In this section, we consider Heuristic \ref{heur:boundary} about the boundary of the Brown measure of $X = p + i q$ where $p, q \in (M, \tau)$ are Hermitian and freely independent. 

First, we restrict to only considering certain $z$ that satisfy some continuity conditions. Then, we will verify that the heuristic is almost true in the case when $p$ and $q$ have 2 atoms: the closure of the intersection of the sets in the heuristic is the boundary (i.e. support) of $\mu'$. 

The main result of this section is that for $X = p + i q$, where $p, q$ have arbitrarily many atoms, the intersection of the two sets in the heuristic is an algebraic curve. Mathematica computations of the empirical spectral distributions of $X_n$ suggest that this algebraic curve contains the boundary of the Brown measure of $X$.

\subsection{General $p$ and $q$}
In this subsection, we first consider a general $X = p + i q$, where $p, q$ are Hermitian and freely independent. With some restrictions on $z$, we first present a computation that characterizes the $z$ where both $l_\epsilon \to 0$ and $B_\epsilon \to 0$:

We only consider $z \in \C$ where: 

\begin{enumerate}
	\item $\lim\limits_{\epsilon \to 0} \mathcal{G}_{\mathbf{X}}(z_\epsilon) = \lim\limits_{\epsilon \to 0} Q_\epsilon = Q$ for some $Q \in \mathbb{H}$.
	\item $\mathcal{B}_{\mathbf{X}}$ and all related functions are defined and continuous in a neighborhood of $Q$.
	\item $g, g^I \not \in \R$ (i.e. $Q \not \in \R \cup i \R$).
\end{enumerate}

The first condition is the continuity assumption made in stating the heuristic. The second condition is needed to use $\mathcal{B}_{\mathbf{X}}$. The third condition is relevant in order to use (\ref{eqn:blue_formula_expanded}). In general, since we are taking closures of the intersection of $\{ z : l_\epsilon \to 0 \} $ and $\{ z : B_\epsilon \to B = 0 \}$, the third condition should not be significant in recovering almost all of the boundary of the Brown measure. In the case when $p$ and $q$ have two atoms, this condition causes our computation to not recover the atoms of the measure. 

The first and second conditions allow us to pass to the limit to see that:

\begin{equation}
	\mathcal{B}_{\mathbf{X}}(Q) = \lim\limits_{\epsilon \to 0^+} \mathcal{B}_{\mathbf{X}}(Q_\epsilon)  = \lim\limits_{\epsilon \to 0^+} z_\epsilon = z \,.
\end{equation}

By passing to the limit, the condition $l_\epsilon \to 0$ implies that $l(Q) = 0$.

The conditions $B_\epsilon \to 0$ and $g, g^I \not \in \R$ imply that $Q \in \C \setminus (\R \cup i \R)$.

First, we assume $Q \in \mathbb{H}$ where $l(Q) = 0$ and $\mathcal{B}_{\mathbf{X}}(Q) = z$. Let $z = x + i y$. The third assumption allows us to apply (\ref{eqn:blue_formula_expanded}) to $\mathcal{B}_{\mathbf{X}}(Q) = z$, which yields:
\begin{equation}
	\begin{aligned}
		x & = k = \beta_p(g) = \frac{g B_{p}(g)  - \bar{g} B_p(\bar{g}) }{g - \bar{g}}\\
		y &= k' = \beta_q(g^I) = \frac{g^I B_q(g^I)  - \bar{g^I} B_{q} (\bar{g^I}) }{g - \bar{g^I}} \,.
	\end{aligned}
\end{equation}
Let $m$ describe the intercept of the (complex) line passing through $(g,  g B_{p}(g))$ and $(\bar{g}, \bar{g}B_p(\bar{g}))$, i.e. 
\begin{equation}
	\begin{aligned}
		g B_p(g) & = x g + m \\
		\bar{g} B_p(\bar{g}) &= x \bar{g} + m \,.
	\end{aligned}
\end{equation}
Conjugating the first equation and comparing with the second equation shows that $m \in \R$. Noting that $B_p(\bar{g}) = \bar{B_p(g)}$ and $g \neq 0$, then: 
\begin{equation}
	\label{eqn:bp_g}
	x = \beta_p(g) \iff B_p(g) = x + \frac{m}{g} \text{ for some } m \in \R \,.
\end{equation}
Similarly, let $m'$ describe the intercept of the line passing through $(g^I, g B_q(g^I))$ and $(\bar{g^I}, \bar{g^I} B_q (\bar{g^I}))$:
\begin{equation}
	\begin{aligned}
		g^I B_q(g^I) & = y g^I + m' \\
		\bar{g^I} B_q(\bar{g^I}) &= y \bar{g^I} + m' \,.
	\end{aligned}
\end{equation}
Similar arguments show:
\begin{equation}
	\label{eqn:bq_gI}
	y = \beta_q(g^I) \iff B_q(g^I) = y + \frac{m'}{g^I} \text{ for some } m' \in \R  \,.
\end{equation}
Referring to our original definitions of $\beta_p$ and $l$ in (\ref{eqn:blue_beta}) and (\ref{eqn:k_k'_l}), 
\begin{equation}
	l = \frac{B_p(g) - B_p(\bar{g})}{g - \bar{g}} + \frac{B_q(g^I) - B_q(\bar{g^I})}{g - \bar{g^I}} + \frac{1}{\Abs{g}^2} \,.
\end{equation}
Inserting the two equations (\ref{eqn:bp_g}) and (\ref{eqn:bq_gI}) into this expression and simplifying shows that
\begin{equation}
	l = 0 \iff m + m' = 1 \,.
\end{equation} 
Recall that we assumed initially that $g, g^I \not \in \R$, i.e. $Q \not \in \R \cup i \R$. Hence, 
\begin{equation}
	\label{eqn:l=0_equiv}
	\begin{aligned}
		\mathcal{B}_{\mathbf{X}}(Q)  & = z \text{ for } Q \in \mathbb{H} \setminus (\R \cup i \R)  \text{ where } l(Q) = 0 \\
		& \iff \\
		B_p(g) & = x + \frac{m}{g} \\
		B_q(g^I) &= y + \frac{1 - m}{g^I} \text{ for some } m \in \R \,.
	\end{aligned}
\end{equation}
Heuristic \ref{heur:support} suggests that the closure of the set of $z$ that satisfy this system of equations should contain the support of the Brown measure of $X$.

If we impose the additional condition that $B = 0$, then $Q \in \C$. Applying Lemma \ref{lem:conjugation} to $B_p(g) - (x + m / g)$ and $ B_q(g^I) - (y + (1 - m) / g^I)$, we may assume that $Q = g$ and $g^I = i g$. Then, 
\begin{equation}
	\label{eqn:l=B=0_equiv}
	\begin{aligned}
		\mathcal{B}_{\mathbf{X}}(g) & = z \text{ for } g \in \C \setminus (\R \cup i \R)  \text{ where } l(g) = 0 \\
		& \iff \\
		B_p(g) & = x + \frac{m}{g} \\
		B_q(i g) &= y + \frac{1 - m}{i g} \text{ for some } m \in \R \,.
	\end{aligned}
\end{equation}
Thus, Heuristic \ref{heur:boundary} with the continuity assumptions suggests that the closure of the set of $z$ that satisfy this system of equations should be the boundary of the Brown measure of $X$.

The support of a general measure on $\C$ is $2$-dimensional and the boundary of the support of a general measure is $1$-dimensional. This agrees with the dimensions of generic solutions to systems of equations with the same number of equations and variables as in (\ref{eqn:l=0_equiv}) and (\ref{eqn:l=B=0_equiv}):

Applying $G_p$ to both sides in (\ref{eqn:l=0_equiv}), and using $\Abs{g^I}^2 = \Abs{g}^2$, the following system of equations contains the solutions to (\ref{eqn:l=0_equiv}):
\begin{equation}
	\begin{aligned}
		g & = G_p \left( x + \frac{m}{g} \right)  \\
		g^I &= G_q \left( y + \frac{1 - m}{g^I} \right) \\
		\Abs{g}^2 &= \Abs{g^I}^2 .
	\end{aligned}
\end{equation}
This is a system of $5$ real equations (taking the real/imaginary parts of the first two equations) with $7$ real variables ($g, g^I \in \C$, $x, y, m \in \R$). Thus, we expect in general the solution set to be $2$-dimensional over $\R$, like the support of a generic measure on $\C$. 

Adding the condition $B = 0$, then the system of equations (\ref{eqn:l=B=0_equiv}) is equivalent to: 
\begin{equation}
	\begin{aligned}
		g & = G_p \left( x + \frac{m}{g} \right)  \\
		g^I &= G_q \left( y + \frac{1 - m}{g^I} \right) \\
		g^I & = i g \,.
	\end{aligned} \,.
\end{equation}
This system of equations has one more real equation than the previous system, with $6$ real equations with $7$ real variables. So, the solution set is a subset of the previous solution set and we expect it to be $1$-dimensional over $\R$, like the boundary of the support of a generic measure on $\C$. 

\subsection{When $p$ and $q$ have 2 atoms}
In this subsection, we apply the computation of the previous subsection to the case when $p$ and $q$ have 2 atoms. Our results are that the $z$ that solve the system of equations in (\ref{eqn:l=0_equiv}) is a set that is contained in the intersection of the hyperbola and open rectangle associated with $X = p + i q$ (from Definition \ref{def:hyperbola_rectangle}). It is unclear whether this set is actually the support of the Brown measure or not, we will discuss the difficulties in determining this. But, adding the extra condition $B = 0$ to produce the system of equations in (\ref{eqn:l=B=0_equiv}) does recover the support of $\mu'$. We will also discuss the atoms of the Brown measure of $X$. 

First, we consider the $z$ that solve the system of equations in (\ref{eqn:l=0_equiv}):

\begin{proposition}
	\label{prop:l_0_compute}
	Let $X = p + i q$, where $p, q \in (M, \tau)$ are Hermitian, freely independent, and have 2 atoms, i.e.
	\begin{equation}
		\begin{aligned}
			\mu_p & = a \delta_\alpha + (1 - a) \delta_{\alpha'} \\
			\mu_q &= b \delta_\beta + (1 - b) \delta_{\beta'} \, ,
		\end{aligned} 
	\end{equation}
	where $a, b \in (0, 1)$, $\alpha_n, \alpha_n', \beta_n, \beta_n' \in \R$, $\alpha \neq \alpha'$, and $\beta \neq \beta'$.
	
	The set of $z \in \C$ such that
	\begin{equation}
		\mathcal{B}_{\mathbf{X}}(Q) = z
	\end{equation}
	for some $Q \in \mathbb{H} \setminus (\R \cup i \R)$ where $l(Q)  = 0$ is contained in the intersection of the hyperbola and open rectangle associated with $X$.
\end{proposition}
\begin{proof}
	Applying $G_p$ to both sides in (\ref{eqn:l=0_equiv}) produces the equivalent system of equations: 
	\begin{equation}
		\label{eqn:G_p_G_q}
		\begin{aligned}
			g & = G_p \left( x + \frac{m}{g} \right)  \\
			g^I &= G_q \left( y + \frac{1 - m}{g^I} \right) \,.
		\end{aligned}
	\end{equation}
	From simplifying these rational expressions and using the formulas for $G_p$ and $G_q$ (see (\ref{eqn:complex_green})), these equations are equivalent to: 
	\begin{equation}
		\label{eqn:2_quadratics}
		\begin{aligned}
			(x - \alpha) (x - \alpha') g^2 +  [(x - \alpha)(m - (1 - a)) + (x - \alpha')(m - a)] g + m^2 - m & = 0  \\
			(y - \beta) (y - \beta') (g^I)^2 -  [(y - \beta)(m - b) + (y - \beta')(m - (1 - b))  ] g^I + m^2 - m & = 0 \,.
		\end{aligned}
	\end{equation}
	We proceed to show that these two equations are quadratic equations with real coefficients and a non-zero constant term: Note that $Q \not \in (\R \cup i \R)$ if and only if $g, g^I \not \in \R$. Then, $g$ and $g^I$ and their conjugates also satisfy their respective equations, so the polynomials have two distinct roots. The only other possibility besides the polynomials being quadratic is that they are the zero polynomial. Without loss of generality, consider the first polynomial. The degree 2 term being zero implies that $x = \alpha$ or $x = \alpha'$. The constant term being zero implies that $m = 0$ or $m = 1$. But then, this implies that the linear term is non-zero, a contradiction.
	
	Then, from $\Abs{g}^2 = \Abs{g^I}^2$, 
	\begin{equation}
		\frac{m^2 - m}{(x - \alpha)(x - \alpha')} = \frac{m^2 - m}{(y - \beta)(y - \beta')} \,.
	\end{equation}
	As $g \neq 0$, then $m^2 - m \neq 0$ (i.e. $m \neq 0, 1$), so:
	\begin{equation}
		(x - \alpha)(x - \alpha') = (y - \beta)(y - \beta') \,.
	\end{equation}
	From Lemma \ref{lem:hyperbola_rectangle}, this is equivalent to the equation of the hyperbola. 
	
	Finally, we will verify that $z = x + iy$ lies in the open rectangle associated with $X$. The first expression in (\ref{eqn:2_quadratics}) viewed as a polynomial in $g$ has two roots at $g \neq \bar{g}$. Hence, the discriminant is negative:
	\begin{equation}
		\begin{aligned}
			0 &> ((x - \alpha)(m - (1 - a)) + (x - \alpha')(m - a) )^2 - 4 (x - \alpha) (x - \alpha') (m^2 - m) \\
			&= (\alpha' - \alpha)^2 m^2 + 2(\alpha' - \alpha)[ 2(a - 1/2) x + ( \alpha(1 - a) - \alpha' \alpha ) ] m + \\
			& \qquad + ( x - (\alpha (1 - a) + \alpha' a) )^2 \,.
		\end{aligned}
	\end{equation}
	Viewed as a polynomial in $m$ over $\R$, this final expression attains a negative value at some real $m$. Hence, there must be two distinct real roots, so the discriminant is positive:
	\begin{equation}
		0 < - 16 (1 - a)a (x - \alpha)(x - \alpha') (\alpha' - \alpha)^2 \,.
	\end{equation}
	As $a \in (0, 1)$, then
	\begin{equation}
		(x - \alpha)(x - \alpha') < 0 \, ,
	\end{equation}
	i.e. $x \in ( \alpha \wedge \alpha', \alpha \vee \alpha'   )$. From Lemma \ref{lem:hyperbola_rectangle}, $z = x + i y$ lies on the open rectangle associated with $X$. Hence, we conclude that when $z = \mathcal{B}_{\mathbf{X}}(Q)$ for a quaternion $Q \not \in \R \cup i \R$ such that $l(Q) = 0$, $z$ lies on the intersection of the hyperbola and open rectangle. 
\end{proof}

From Theorem \ref{thm:brown_measure_p+iq}, the intersection of the hyperbola with the open rectangle is contained in the support of the Brown measure only if $a = b = 1/2$, so in general, the set in Proposition \ref{prop:l_0_compute} only contains the support of $\mu'$. 

The potential atoms of the Brown measure, $\{\alpha + i \beta, \alpha' + i \beta, \alpha + i \beta', \alpha' + i \beta'\}$ correspond to the two equations in (\ref{eqn:2_quadratics}) becoming linear equations. This can be viewed as a degenerate situation. Even if we allow solutions where $g, g^I \in \R$, these two linear equations are not satisfied, since it would require both $g, g^I \in \R$. This degeneracy of the 4 corners of the Brown measure also appears when we examine the heuristic for the support of the Brown measure (Proposition \ref{prop:quaternionic_atoms}). 

In order to determine the precise set of $z$ for which the condition in Proposition \ref{prop:l_0_compute} holds, we need to not only analyze individually which $x$ (resp. $y$) have the discriminants of the quadratic equations in $g$ (resp. $g^I$) attain a negative value at a real $m$, but we need to find an $m$ that \textit{simultaneously} works for both equations. We also need to analyze which $g, g^I$ that solve the equations actually come from a $Q$ (it is not sufficient that $\Abs{g}^2 = \Abs{g^I}$). Finally, we need to eliminate those $x$ (resp $y$) where $g \in I_p$ (resp $g^I \in I_q$), as the original equations are not defined for those $g, g^I$. In general, this turns out to be difficult, but if we intersect with the condition that $B = 0$, then we recover almost all of the boundary of the Brown measure:

\begin{proposition}
	\label{prop:l_0_B_0_compute}
	Let $X = p + i q$, where $p, q \in (M, \tau)$ are Hermitian, freely independent, and have 2 atoms, i.e.
	\begin{equation}
		\begin{aligned}
			\mu_p & = a \delta_\alpha + (1 - a) \delta_{\alpha'} \\
			\mu_q &= b \delta_\beta + (1 - b) \delta_{\beta'} \, ,
		\end{aligned} 
	\end{equation}
	where $a, b \in (0, 1)$, $\alpha_n, \alpha_n', \beta_n, \beta_n' \in \R$, $\alpha \neq \alpha'$, and $\beta \neq \beta'$.
	
	The set of $z \in \C$ such that
	\begin{equation}
		\mathcal{B}_{\mathbf{X}}(g) = z
	\end{equation}
	for some $g \in \C \setminus (\R \cup i \R)$ where $l(g)  = 0$ is equal to the support of the Brown measure of $X$ with at most finitely many points removed. The closure of this set is the support of $\mu'$ in Theorem \ref{thm:brown_measure_p+iq}.
\end{proposition}
\begin{proof}
	Let
	\begin{equation}
		\begin{aligned}
			\mathscr{A} &= \alpha' - \alpha \\
			\mathscr{B} &= \beta' - \beta \\
			x' & = x - \frac{\alpha + \alpha'}{2} \\
			y' & = y - \frac{\beta + \beta'}{2} \\
			\tilde{m} &= m - 1/2 \\
			\tilde{a} &= a - 1/2 \\
			\tilde{b} &= b - 1/2 \,.
		\end{aligned}
	\end{equation}
	Since the set in the Proposition is a subset of the set in Proposition \ref{prop:l_0_compute}, from Lemma \ref{lem:hyperbola_rectangle}, 
	\begin{equation}
		\label{eqn:H}
		\mathscr{H} = (x')^2 - \frac{\mathscr{A}^2}{4}   = (y')^2 - \frac{\mathscr{B}^2}{4} < 0 \,.
	\end{equation}
	Apply $G_p$ and $G_q$ to the equations in (\ref{eqn:l=B=0_equiv}) to obtain the equivalent equations:
	\begin{equation}
		\begin{aligned}
			g & = G_p \left( x + \frac{m}{g} \right)  \\
			i g &= G_q \left( y + \frac{1 - m}{i g} \right) \,.
		\end{aligned}
	\end{equation}
	Recall that the set in the Proposition is equal to the set of $z = x + i y$ such that there exists $g \in \C \setminus (\R \cup i \R)$, $m \in \R$ that satisfy these equations. 
	
	From simplifying these rational expressions and using the formulas for $G_p$ and $G_q$ (see (\ref{eqn:complex_green})), these equations are equivalent to: 
	\begin{equation}
		\label{eqn:2_quadratics_1}
		\begin{aligned}
			(x - \alpha) (x - \alpha') g^2 +  [(x - \alpha)(m - (1 - a)) + (x - \alpha')(m - a)] g + m^2 - m & = 0  \\
			- (y - \beta) (y - \beta') g^2 -  [(y - \beta)(m - b) + (y - \beta')(m - (1 - b))  ] i g + m^2 - m & = 0 \,.
		\end{aligned}
	\end{equation}
	In the new variables, the previous two quadratic equations are: 
	\begin{equation}
		\label{eqn:2_quadratics_2}
		\begin{aligned}
			\mathscr{H} g^2  + (2 x' \tilde{m} + \mathscr{A} \tilde{a})g + \tilde{m}^2 - 1/4 & = 0 \\
			- \mathscr{H} g^2 - (2 y' \tilde{m} - \mathscr{B} \tilde{b}) i g + \tilde{m}^2 - 1/4 &= 0  \,.
		\end{aligned}
	\end{equation}
	As $g \neq \bar{g}$ are solutions to the first equation, then 
	\begin{equation}
		\label{prop:l_0_B_0_compute:m}
		\tilde{m}^2 - 1 / 4 = \mathscr{H} \Abs{g}^2 < 0 \,.
	\end{equation}
	Taking the sum and difference of the equations in (\ref{eqn:2_quadratics_2}) and letting $z' = x' + i y'$ produces the equivalent system of equations:
	\begin{equation}
		\label{eqn:2_quadratics_3}
		\begin{aligned}
			(2 \bar{z'} \tilde{m} + \mathscr{A} \tilde{a} + i \mathscr{B} \tilde{b})g + 2 (\tilde{m}^2 - 1/4) & = 0 \\
			2 \mathscr{H} g^2 + (2 z' \tilde{m} + \mathscr{A} \tilde{a} - i \mathscr{B} \tilde{b} ) g &= 0 \,.
		\end{aligned}
	\end{equation}
	The second equation can be factored: 
	\begin{equation}
		g (2 \mathscr{H} g + (2 z' \tilde{m} + \mathscr{A} \tilde{a} - i \mathscr{B} \tilde{b} ))  = 0 \,.
	\end{equation}
	Then, (\ref{eqn:2_quadratics_3}) has a solution for some $g \in \C \setminus (\R \cup i \R)$, $\tilde{m} \in \R$ if and only if the following system of linear equations has a solution for some $g \in \C \setminus (\R \cup i \R)$, $\tilde{m} \in \R$: 
	\begin{equation}
		\begin{aligned}
			\label{eqn:2_quadratics_linear}
			(2 \bar{z'} \tilde{m} + \mathscr{A} \tilde{a} + i \mathscr{B} \tilde{b})g + 2 (\tilde{m}^2 - 1/4) & = 0 \\
			2 \mathscr{H} g + (2 z' \tilde{m} + \mathscr{A} \tilde{a} - i \mathscr{B} \tilde{b} ) &= 0 \,.
		\end{aligned}
	\end{equation}	
	By taking the determinant of the associated $2 \times 2$ matrix, this system has a solution for some $g \in \C$, $\tilde{m} \in \R$ if and only if 
	\begin{equation}
		\label{eqn:lem:l_0_B_0_det}
		\Abs{ z' \tilde{m} + \frac{\mathscr{A} \tilde{a} - i \mathscr{B} \tilde{b} }{2}}^2 = \mathscr{H}(\tilde{m}^2 - 1/4 ) 
	\end{equation}
	for some $\tilde{m} \in \R$. 
	
	We proceed to show that the $z \in C$ where there exists a $\tilde{m} \in \R$ that solves this equation is a set whose closure is the support of $\mu'$. Afterwards, we will consider the possibilities where $g \in \R \cup i \R$, $g \in I_p$, or $i g \in I_q$.
	
	Expanding out the absolute value in the previous equation,
	\begin{equation}
		\label{eqn:lem:l_0_B_0_poly}
		\left( x' \tilde{m} + \frac{\mathscr{A} \tilde{a} }{2} \right)^2 + \left( y' \tilde{m} - \frac{\mathscr{B} \tilde{b}}{2} \right)^2 = \mathscr{H}(\tilde{m}^2 - 1/4 ) \,.
	\end{equation}
	Rewriting as a polynomial in $\tilde{m}$,
	\begin{equation}
		( (x')^2 + (y')^2 - \mathscr{H}  ) \tilde{m}^2 + (\mathscr{A} \tilde{a} x' - \mathscr{B} \tilde{b} y') \tilde{m} + \frac{\mathscr{A}^2 \tilde{a}^2 + \mathscr{B}^2 \tilde{b}^2 + \mathscr{H}}{4} = 0 \,.
	\end{equation}
	There is an $\tilde{m} \in \R$ that solves this if and only if the discriminant is non-negative. Simplifying the discriminant using (\ref{eqn:H}),
	\begin{equation}
		\begin{aligned}
			0 & \leq (\mathscr{A} \tilde{a} x' - \mathscr{B} \tilde{b} y')^2 - ( (x')^2 + (y')^2 - \mathscr{H}  ) (\mathscr{A}^2 \tilde{a}^2 + \mathscr{B}^2 \tilde{b}^2 + \mathscr{H})
			\\ &= - (x')^2 (y')^2 - 2 a b \mathscr{A} \mathscr{B} x' y' - \frac{\mathscr{A}^2 \mathscr{B}^2 (4 a^2 + 4 b^2 - 1)}{16}  \,.
		\end{aligned}
	\end{equation}
	Define the new variables: 
	\begin{equation}
		\begin{aligned}
			\tilde{x} &= \frac{x'}{\mathscr{A}} \\
			\tilde{y} &= \frac{y'}{\mathscr{B}} \,.
		\end{aligned}
	\end{equation}
	The previous inequality is equivalent to: 
	\begin{equation}
		\label{eqn:lem:l_0_B_0_xy}
		(\tilde{x} \tilde{y})^2 + 2 \tilde{a} \tilde{b} (\tilde{x} \tilde{y}) + \frac{4 \tilde{a}^2 + 4 \tilde{b}^2 - 1}{16} \leq 0 \,.
	\end{equation}
	We now show that this condition along with $z = x + i y$ being on the intersection of the hyperbola with the open rectangle is equivalent to the following condition: 
	
	From Theorem \ref{thm:brown_measure_p+iq}, the support $\mu'$ is the closure of the set of $z \in \C$ on the hyperbola such that:
	\begin{equation}
		\label{eqn:lem:l_0_B_0_brown_measure}
		\Im \left( \left( z - \frac{\alpha + \alpha'}{2} - i \frac{\beta + \beta'}{2}  \right)^2\right) = \frac{\mathscr{A} \mathscr{B} \cos(2 \theta)}{2} \,,
	\end{equation}
	where $\theta \in (0, \pi / 2)$ satisfies
	\begin{equation}
		f(\sec^2 (\theta)) \leq 0 \,,
	\end{equation}
	where
	\begin{equation}
		f(t) = (a - b)^2 t^2 + (4 ab - 2(a + b)) t + 1 \,.
	\end{equation}
	From Lemma \ref{lem:hyperbola_rectangle}, for $z = x + i y$ on the hyperbola, (\ref{eqn:lem:l_0_B_0_brown_measure})  is equivalent to being on the open rectangle. The condition can be rewritten in the new variables as: 
	\begin{equation}
		\label{eqn:lem:l_0_B_0_theta}
		4 \tilde{x} \tilde{y} = \cos(2 \theta) \,.
	\end{equation}
	Next, by using the new variables, 
	\begin{equation}
		f(t) = (\tilde{a} - \tilde{b})^2 t^2 +  (4 \tilde{a} \tilde{b} - 1) t + 1 \,.
	\end{equation}
	From the double-angle formula, 
	\begin{equation}
		\sec^2(\theta) = \frac{1}{\cos^2(\theta)} = \frac{2}{1 + \cos(2 \theta)} = \frac{2}{1 + 4 \tilde{x} \tilde{y}} \,.
	\end{equation}
	Combining these final two expressions, a straightforward computation shows that $f(\sec^2(\theta)) \leq 0$ is equivalent to (\ref{eqn:lem:l_0_B_0_xy}). Hence, the closure of the $z = x + i y$ on the intersection of the hyperbola and open rectangle and satisfying (\ref{eqn:lem:l_0_B_0_xy}) is the support of $\mu'$.
	
	Finally, we need to remove the $z \in \C$ where there exists a $g \in \R \cup i \R$, $g \in I_p$, or $i g \in I_q$ that solves (\ref{eqn:2_quadratics_linear}). It suffices to remove solutions where $\Im(g) = c$ or $\Re(g) = c$ for some $c \in \R$. We claim that for any $c \in \R$ this removes only finitely many points. Since the support of $\mu'$ has no isolated points, then the closure of the set is still the support of $\mu'$. We will prove the case where $\Im(g) = c$, the case where $\Re(g) = c$ is similar. 
	
	If there exists a $g$ solving (\ref{eqn:2_quadratics_linear}), then since $\mathscr{H} < 0$ and $\tilde{m}^2 - 1/4 < 0$ (from (\ref{prop:l_0_B_0_compute:m})) so there are 2 equations for $g$:
	\begin{equation}
		\label{eqn:lem:l_0_B_0_g}
		- \frac{ z' \tilde{m} + \frac{\mathscr{A} \tilde{a} - i \mathscr{B} \tilde{b} }{2}  }{\mathscr{H}} = g = - \frac{\tilde{m}^2 - 1/4}{\bar{z'} \tilde{m} + \frac{\mathscr{A} \tilde{a} + i \mathscr{B} \tilde{b}}{2}} \,.
	\end{equation}	
	Using the first equation,
	\begin{equation}
		\label{eqn:lem:l_0_B_0_g_c}
		\Im(g) = c \iff \mathscr{H} c = - y' \tilde{m} + \frac{\mathscr{B} \tilde{b}}{2} \,.
	\end{equation}
	We may assume without loss of generality that $y' \neq 0$, as we can remove the finitely many points where this occurs on the hyperbola. Then, rewriting the final expression as $\tilde{m}$ in terms of $y'$,
	\begin{equation}
		\tilde{m} = \frac{1}{y'} \left( \frac{\mathscr{B} \tilde{b}}{2} - \mathscr{H} c \right) =   - y' c + \frac{1}{y'} \left( \frac{\mathscr{B} \tilde{b}}{2} + \frac{\mathscr{B}^2 c}{4}\right) \,.
	\end{equation}
	Using that 
	\begin{equation}
		\begin{aligned}
			x' & = \pm \sqrt{ \frac{\mathscr{A}^2 - \mathscr{B}^2}{4} + (y')^2} \\
			H & = (y')^2 - \frac{\mathscr{B}^2}{4} \,,
		\end{aligned}
	\end{equation}
	we can rewrite (\ref{eqn:lem:l_0_B_0_poly}) in terms of $y'$ and then manipulate the expression to see that $y'$ is a root of a rational equation. We claim that this rational equation is non-zero, so then there are only finitely many $y'$ that solve this equation. Hence, when we remove $\Im(g) = c$ for some $c$, we only remove finitely many points. 
	
	If (\ref{eqn:lem:l_0_B_0_poly}) and (\ref{eqn:lem:l_0_B_0_g_c}) hold, then the following equation also holds: 
	\begin{equation}
		\left(  \mathscr{H} (\tilde{m}^2 - 1/4) - (\mathscr{H} c)^2 - (x' m)^2 - \frac{\mathscr{A}^2 \tilde{a}^2}{4} \right)^2 - (x' \tilde{m} \mathscr{A} \tilde{a})^2 = 0 \,.
	\end{equation}	
	For $c \neq 0$, by rewriting $x', H, \tilde{m}$ in terms of $y'$ using the previous equations, then the left-hand side is a rational expression in $y'$. The highest degree term is $(y')^8$ with coefficient $c^4$, so the rational expression is non-zero.
	
	If $c = 0$, the highest degree term on the left-hand side is $(y')^4$ with coefficient $1/ 16$, so the rational expression is non-zero. 
	
	We conclude that when we remove the possible solutions of (\ref{eqn:lem:l_0_B_0_g}) where $g \in \R \cup i \R$, $g \in I_p$, or $i g \in I_q$ only removes finitely many points. Hence, it does not affect the closure of the set. 
\end{proof}

When $a = b = 1/2$ (equivalently, $\tilde{a} = \tilde{b} = 0$), we conclude from Proposition \ref{prop:l_0_B_0_compute} and Theorem \ref{thm:brown_measure_p+iq} that the closure of the set in Proposition \ref{prop:l_0_compute} is also the support of $\mu'$ (which is also the support of the Brown measure in this case). It is also easy to check directly (\ref{eqn:2_quadratics}) can be solved with $m = 1/2$ for all $z$ on the support of the Brown measure. Note that in these solutions $g, g^I \in i \R$, so these are not the same solutions for Proposition \ref{prop:l_0_B_0_compute}.

\subsection{When $p$ and $q$ have finitely many atoms}
In this subsection, we will show that the heuristic from (\ref{eqn:l=B=0_equiv}) implies that the boundary of the Brown measure of $X = p + i q$ is an algebraic curve, and produce a $2$-variable polynomial whose zero set contains this curve.

Recall that our heuristic from (\ref{eqn:l=B=0_equiv}) is that the boundary of the Brown measure of $X = p + i q$ is the closure of the set of $z = x + i y$ for which there exists $g \in \C \setminus (\R \cup i \R)$ and $m \in \R$  such that the following system of equations is satisfied:
\begin{equation}
	\begin{aligned}
		B_p(g) & = x + \frac{m}{g} \\
		B_q(i g) &= y + \frac{1 - m}{i g} \,.
	\end{aligned}
\end{equation}
In Proposition \ref{prop:l_0_B_0_compute}, we verified this claim was almost true when $p$ and $q$ have two atoms (the set was missing the atoms of the Brown measure).

Next, consider $X = p + i q$ where $p$ and $q$ have arbitrarily many atoms. Let 
\begin{equation}
	\begin{aligned}
		\label{eqn:p_q_many_atoms}
		\mu_p & = a_1 \delta_{\alpha_1} + \cdots + a_n \delta_{\alpha_n} \\
		\mu_q & = b_1 \delta_{\beta_1} + \cdots + b_k \delta_{\beta_k}  \, ,
	\end{aligned}
\end{equation}
where $\alpha_i, \beta_j \in \R$, $a_i, b_j \geq 0$ and $a_1 + \cdots + a_n = b_1 + \cdots + b_k = 1$. It is significant to what follows that we do not assume the $\alpha_i$ (resp. $\beta_j$) are distinct. 

Let 
\begin{equation}
	\begin{aligned}
		\bm{\alpha} & = (\alpha_1, \ldots, \alpha_n) \\
		\bm{\beta} & = (\beta_1, \ldots, \beta_k) \\
		\bm{a} & = (a_1, \ldots, a_n) \\
		\bm{b} & = (b_1, \ldots, b_k)
	\end{aligned}
\end{equation}
be shorthand for the positions and weights of the atoms in the measures.

The (complex) Green's functions for $p$ and $q$ are: 
\begin{equation}
	\begin{aligned}
		G_p(z) & = \frac{a_1}{z - \alpha_1} + \cdots + \frac{a_n}{z - \alpha_n} \\
		G_q(z) & = \frac{b_1}{z - \beta_1} + \cdots + \frac{b_k}{z - \beta_k} \, .
	\end{aligned}
\end{equation}
Applying $G_p$ and $G_q$, an equivalent system of equations is: 
\begin{equation}
	\label{eqn:l=B=0_G_equiv}
	\begin{aligned}
		G_p \left( x + \frac{m}{g}\right) & = g   \\
		G_q \left( y + \frac{1 - m}{i g} \right)  &= i g  \,.
	\end{aligned}
\end{equation}
Let 
\begin{equation}
	\label{eqn:Omega_p_q}
	\Omega_{p, q} = \left\lbrace z = x + i y \in \C : \text{ there exists } g \in \C \setminus (\R \cup i \R), m \in \R \text{ satisfying } (\ref{eqn:l=B=0_G_equiv})  \right\rbrace \,.
\end{equation}
From (\ref{eqn:l=B=0_equiv}), $\bar{\Omega_{p, q}}$ is heuristically understood to be the boundary of the Brown measure of $X = p + i q$.  

The main result of the subsection is the following: 

\begin{theorem}
	\label{thm:boundary_curve}
	Fix $\bm{a}, \bm{b}$. Then, for Lebesgue almost every $(\bm{\alpha}, \bm{\beta}) \in \R^n \times \R^k = \R^{n + k}$, $\Omega_{p, q}$ lies on a real algebraic curve, i.e. $\Omega_{p, q}$ lies in the zero set of some non-zero two-variable polynomial with real coefficients. In particular, we provide an explicit algorithm to produce such a polynomial.
\end{theorem}

The proof is split up into the following parts: 
\begin{enumerate}
	\item State the algorithm to find the two-variable polynomial whose zero set contains $\Omega_{p, q}$.
	\item Check the algorithm does indeed produce a two-variable polynomial whose zero set contains $\Omega_{p, q}$.
	\item Prove that in the generic situation, this algorithm produces a non-zero polynomial. 
\end{enumerate}

First, we provide some figures generated using Mathematica comparing the empirical spectral distribution of a deterministic $X_n = P_n + i Q_n$ with the sets $\bar{\Omega_{p, q}}$ and the algebraic curve the algorithm produces. Note that there is some noise in the data from the numerical solutions to the system of equations defining $\bar{\Omega_{p, q}}$.

\begin{figure}[H]
	\centering
	\begin{subfigure}{0.45 \textwidth}
		\centering
		\includegraphics[width = .9 \textwidth]{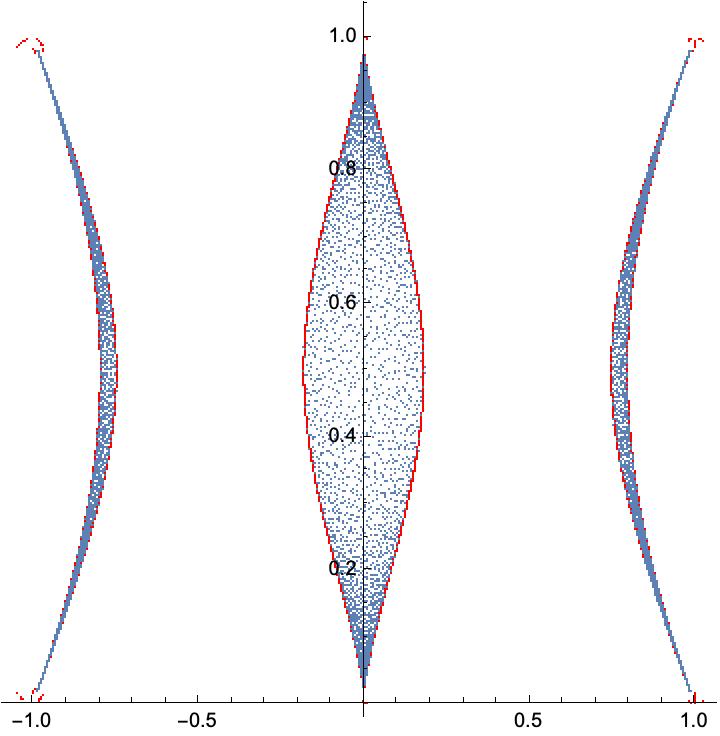}
		\caption{ESD of $X_n$ vs. $\bar{\Omega_{p, q}}$ from (\ref{eqn:Omega_p_q})}
	\end{subfigure}
	\hfill
	\begin{subfigure}{0.45 \textwidth}
		\centering
		\includegraphics[width = .9 \textwidth]{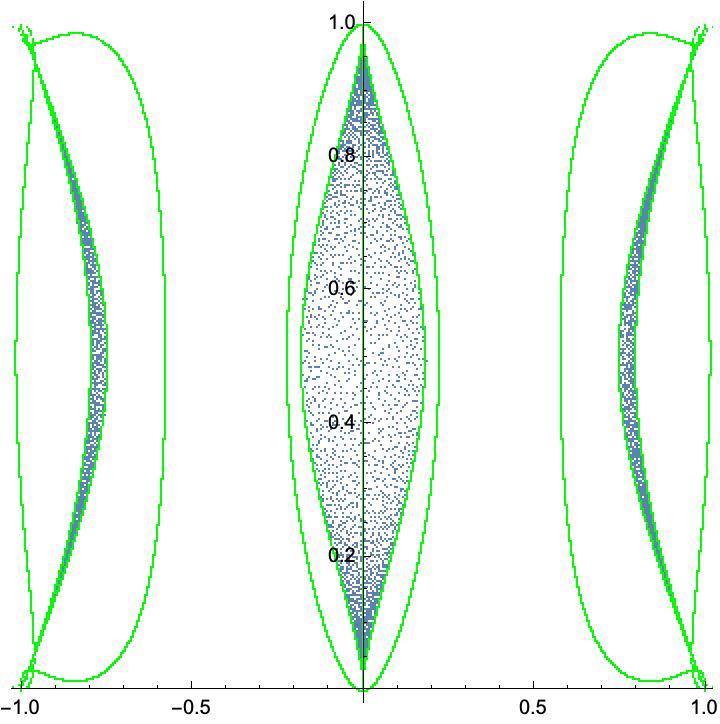}
		\caption{ESD of $X_n$ vs. algebraic curve}
	\end{subfigure}
	\caption{$X_n = P_n + i Q_n$ \\ $\mu_{P_n} \approx (1/3) \delta_{-1} + (1/3) \delta_{0} + (1/3) \delta_{1}$ \\ $\mu_{Q_n} = (1 / 2) \delta_0 + (1/2) \delta_1$ \\ $n = 10000$}
	\label{fig:thm:boundary_curve}
\end{figure} 

\begin{figure}[H]
	\centering
	\begin{subfigure}{0.45 \textwidth}
		\centering
		\includegraphics[width = .9 \textwidth]{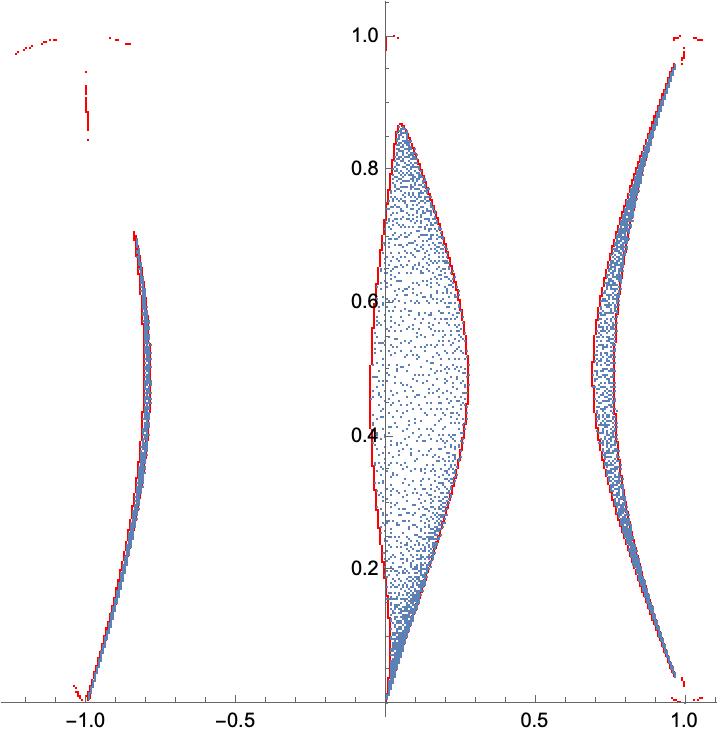}
		\caption{ESD of $X_n$ vs. $\bar{\Omega_{p, q}}$ from (\ref{eqn:Omega_p_q})}
	\end{subfigure}
	\hfill
	\begin{subfigure}{0.45 \textwidth}
		\centering
		\includegraphics[width = .9 \textwidth]{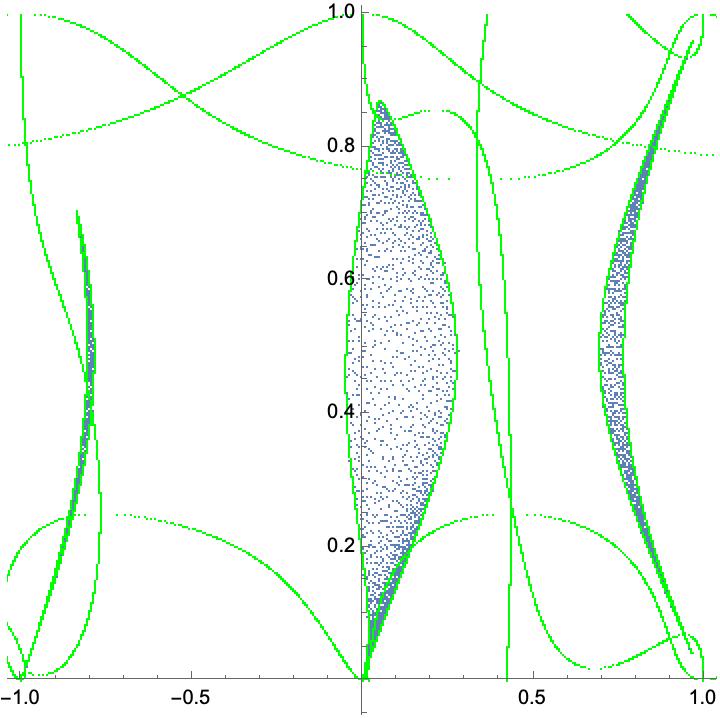}
		\caption{ESD of $X_n$ vs. algebraic curve}
	\end{subfigure}
	\caption{$X_n = P_n + i Q_n$ \\ $\mu_{P_n} \approx  (1/6) \delta_{-1} + (1/3) \delta_{0} + (1/2) \delta_{1}$ \\ $\mu_{Q_n} = (3 / 4) \delta_0 + (1/4) \delta_1$ \\ $n = 10000$}
	\label{fig:thm:boundary_curve_1}
\end{figure}

\subsubsection{Algorithm}
Now, we will state the algorithm that produces a two-variable polynomial whose zero set contains $\Omega_{p, q}$. 

First, the system of equations defining $\Omega_{p, q}$, (\ref{eqn:l=B=0_G_equiv}), can be written as a system of polynomial equations: 
\begin{equation}
	\label{eqn:l=B=0_poly_equiv}
	\begin{aligned}
		\prod_{i = 1}^{n} (g(x - \alpha_i) + m)  - \sum_{i = 1}^{n} \left(  a_i \prod_{s \neq i}^{} (g(x - \alpha_s) + m)\right)  & = 0 \\ 
		\prod_{j = 1}^{l}( i g(y - \beta_j) + 1 - m) - \sum_{j = 1}^{n} \left( b_j \prod_{s \neq j}^{} (i g(y - \beta_s) + 1 - m) \right)  & = 0  \,.
	\end{aligned}
\end{equation}
Next, we recall the resultant of two polynomials and some basic properties: 

\begin{definition}
	Let $A(x) = a_n  x^n + a_{n - 1} x^{n - 1} + \cdots + a_0$ and $B(x) = b_k x ^k + b_{k - 1} x^{k - 1}  + \cdots + b_0$ be one-variable polynomials with coefficients in a commutative ring $R$. The \textbf{resultant} of $A$ and $B$, $\Res(A, B)$, is the determinant of the $(n + k) \times (n + k)$ \textbf{Sylvester matrix}: 
	
	\begin{equation}
		\Res(A, B) = 
		\begin{vmatrix}
			a_n & & & & b_k & & & \\
			a_{n - 1} & a_n & & & b_{k - 1} & b_k & & \\
			a_{n - 2} & a_{n - 1} & \ddots & & b_{k - 2} & b_{k - 1} & \ddots & \\
			\vdots & \vdots & & a_n & \vdots & \vdots & & b_k \\
			a_0 & a_1 & & \vdots & b_0 & b_1 & & \vdots \\
			& a_0 & \ddots & \vdots & & b_0 & \ddots & \vdots \\
			& & \ddots & a_1 & & & \ddots & b_1 \\
			& & & a_0 & & & & b_0 
		\end{vmatrix}  
	\end{equation}
\end{definition}

Suppose that $R$ is an integral domain. Then, it makes sense to talk about the roots of $A(x)$ and $B(x)$ in some algebraically closed field containing $R$. 

We state the following well-known result:

\begin{proposition}
	\label{prop:roots_resultant_formula}
	Let $a_i, b_j \in R$, where $R$ is an integral domain. Let $\lambda_i, \mu_j$ be the roots of $A(x) = a_n  x^n + a_{n - 1} x^{n - 1} + \cdots + a_0$ and $B(x) = b_k x ^k + b_{k - 1} x^{k - 1}  + \cdots + b_0$ in some algebraically closed field containing $R$, respectively. 
	\begin{equation}
		\begin{aligned}
			A(x) & = a_n (x - \lambda_1) \cdots (x - \lambda_n) \\
			B(x) &= b_k (x - \mu_1) \cdots (x - \mu_k) \,.
		\end{aligned}
	\end{equation}
	Then, 
	\begin{equation}
		\Res(A, B) = a_n^k b_k^n \prod_{\substack{1 \leq i \leq n \\ 1 \leq j \leq k}}^{} (\lambda_i - \mu_j) \,.
	\end{equation}
\end{proposition}

A corollary of this result is: 

\begin{corollary}
	Let $a_i, b_j \in R$, where $R$ is an integral domain. Let  $A(x) = a_n  x^n + a_{n - 1} x^{n - 1} + \cdots + a_0$ and $B(x) = b_k x ^k + b_{k - 1} x^{k - 1}  + \cdots + b_0$. Then, $A$ and $B$ have a common root in some algebraically closed field containing $R$ if and only if $\Res(A, B) = 0$.
\end{corollary}

We can now state the algorithm to produce the two-variable polynomial:

\begin{algorithm}
	\label{alg:algebraic_curve}
	The following algorithm takes $p, q$ as in (\ref{eqn:p_q_many_atoms}) and produces a two-variable real polynomial $f(x, y)$:
	\begin{enumerate}
		\item Take the resultants of the polynomials in (\ref{eqn:l=B=0_poly_equiv}) with respect to $g$. Let $f_1(m, x, y)$ be this resultant.
		\item Divide $f_1(m, x, y)$ by $m^{n - 1} (m - 1)^{k - 1}$. Let $f_2(m, x, y)$ be the result of this, so that $f_2(m, x, y) m^{n - 1} (m - 1)^{k - 1}  = f_1(m, x, y)$. 
		\item Take the real and imaginary parts of $f_2(m, x, y)$ assuming that $m, x, y \in \R$. This produces real polynomials $\Re f_2(m, x, y)$ and $\Im f_2(m, x, y)$.
		\item Take the resultant of $\Re f_2(m, x, y)$ and $\Im f_2(m, x, y)$ with respect to $m$. Return this polynomial as $f(x, y)$, a real two-variable polynomial.  
	\end{enumerate}
\end{algorithm}

This idea of reducing the number of variables in a system of polynomial equations by taking resultants is not new (see \cite{Resultant} for a description of the technique), but the main issue is that by computing resultants one may introduce too many new solutions in the system. In particular, we wish to avoid the situation that the resultant is the zero polynomial, which gives no information about the solutions of the original system. 

The second step of dividing by $m^{n - 1} (m - 1)^{k - 1}$ avoids the resultant being the zero polynomial (in general). There is also the detail that $x, y, m$ are real variables, but $g$ is complex and we start with two complex equations in (\ref{eqn:l=B=0_poly_equiv}). This is handled by taking $\Re f_2$ and $\Im f_2$ in Step 3.

\subsubsection{Proof of correctness for algorithm}
Now, we will prove that Algorithm \ref{alg:algebraic_curve} produces a polynomial whose zero set contains $\Omega_{p, q}$.

Consider $z_0 = x_0 + i y_0 \in \Omega_{p, q}$ where $(x_0, y_0, g_0, m_0)$ solves (\ref{eqn:l=B=0_poly_equiv}). Recall that $x_0, y_0, m_0 \in \R$ and $g_0 \in \C \setminus (\R \cup i \R)$.

In Step 1, substitute $x = x_0$, $y = y_0$ and $m = m_0$ into the polynomials in (\ref{eqn:l=B=0_poly_equiv}) and treat them as polynomials in $g$. Then, these polynomials both have a root at $g = g_0$, so their resultant with respect to $g$ must be zero. Hence, $f_1(m_0, x_0, y_0) = 0$.

In Step 2, we must show that $m^{n - 1} (m - 1)^{k - 1}$ divides $f_1(m, x, y)$ and also that $f_2(m_0, x_0, y_0) = 0$. Supposing we have proven the first statement, the second statement follows immediately from the fact that $m_0 \neq 0, 1$:

If $m_0 = 0$, then the first equation in (\ref{eqn:l=B=0_G_equiv}) is $G_p(x_0) = g_0$ for $x_0 \in \R$, so then $g_0 \in \R$, a contradiction. Similarly, if $m_0 = 1$, the second equation is $G_q(y_0) = i g_0$ for $y_0 \in \R$, so $g_0 \in i \R$, a contradiction. 

The fact that $m^{n - 1} (m - 1)^{k - 1}$ divides $f_1(m, x, y)$ depends on the following Lemma:  

\begin{lemma}
	\label{prop:sylvester}
	
	As elements of $\Z[a_0, \ldots, a_n, b_0, \ldots, b_k, m]$, $m^{n - 1} (m - 1)^{k - 1}$ divides the following determinant:
	\begin{equation}
		\begin{vmatrix}
			a_n & & & b_k & & \\
			a_{n - 1} & \ddots & & b_{k - 1} & \ddots & \\
			a_{n - 2} m & & a_n & b_{k - 2} (m - 1) & & b_k \\
			\vdots & & \vdots & \vdots & & \vdots \\
			a_0 m^{n - 1} (m - 1) & & \vdots & b_0 (m - 1)^{k - 1} m & & \vdots \\
			& \ddots & a_1 m^{n - 2} & & \ddots & b_1 (m - 1)^{k - 2}  \\
			& & a_0 m^{n - 1} (m - 1) & & & b_0 (m - 1)^{k - 1} m
		\end{vmatrix}
	\end{equation}
	More precisely, this is the determinant of the matrix obtained from the Sylvester matrix for $A(x) = a_n  x^n + a_{n - 1} x^{n - 1} + \cdots + a_0$ and $B(x) = b_k x ^k + b_{k - 1} x^{k - 1}  + \cdots + b_0$ with the following changes:
	\begin{enumerate}
		\item Replace all $a_i$ with $a_{i} m^{n - 1 - i}$ for $i = 1, \ldots, n - 1$. 
		\item Replace all $b_i$ with $b_i (m - 1)^{k - 1 - i}$ for $i = 1, \ldots, k - 1$. 
		\item Replace all $a_0$ with $a_0 m^{n - 1}(m - 1)$.
		\item Replace all $b_0$ with $b_0 (m - 1)^{k - 1} m$.
	\end{enumerate}
\end{lemma}
\begin{proof}
	As the polynomial ring is a unique factorization domain and $m$, $m - 1$ are primes, it suffices to check that $m^{n - 1}$ and $(m - 1)^{k - 1}$ each individually divide the determinant. 
	
	We will just prove that $m^{n - 1}$ divides the determinant, as the other case can be obtained from switching the roles of $A$ and $B$ and using $m' = 1 - m$ instead of $m$
	
	Consider the terms that come from evaluating the determinant using the Leibniz formula. To show that $m^{n - 1}$ divides the determinant, it suffices to check that $m^{n - 1}$ divides any term that is a product of some non-zero entries of the matrix. 
	
	To simplify matters, consider only the power of $m$ in each coordinate, i.e. it suffices to check that $m^{n - 1}$ divides the non-zero terms of the Leibniz formula in the following determinant: 
	\begin{equation}
		\label{eqn:matrix_m}
		\begin{vmatrix}
			1 & & & 1 & & & \\
			1 & \ddots & & 1 & \ddots & & \\
			m & & 1 & 1 & & 1 & \\
			\vdots & & \vdots & \vdots & & \vdots & 1\\
			\vdots & & \vdots & 1 & & \vdots & \vdots \\
			m^{n - 2} & & \vdots& m & & 1 & \vdots \\
			m^{n - 1} &  & m^{n - 3} & & \ddots & 1  & 1\\
			& \ddots & m^{n - 2} & & &  m & 1 \\
			& & m^{n - 1} & & & & m
		\end{vmatrix}
	\end{equation}	
	The proof follows by induction on $n$ (with $k$ fixed). The base case $n = 1$ is trivial as $m^{n - 1} = 1$. For the inductive step, suppose that the claim has been verified for $n - 1$ and consider the statement for $n$. Any term in the Leibniz formula where one of the $m^{n - 1}$ terms is part of the product is clearly divisible by $m^{n - 1}$. 
	
	Hence, it suffices to consider only those terms in the Leibniz formula that are a product of matrix entries that are not $m^{n - 1}$. This corresponds to changing all instances of $m^{n - 1}$ to $0$ and looking for terms in the Leibniz formula that are products of non-zero entries of this new matrix:
	\begin{equation}
		\begin{vmatrix}
			1 & & & 1 & & & \\
			1 & \ddots & & 1 & \ddots & & \\
			m & & 1 & 1 & & 1 & \\
			\vdots & & \vdots & \vdots & & \vdots & 1\\
			\vdots & & \vdots & 1 & & \vdots & \vdots \\
			m^{ - 2} & & \vdots& m & & 1 & \vdots \\
			0 &  & m^{n - 3} & & \ddots & 1  & 1\\
			& \ddots & m^{n - 2} & & &  m & 1 \\
			& & 0 & & & & m
		\end{vmatrix}
	\end{equation}	
	In this new matrix, there is only one non-zero entry in the last row, the lower right $m$. Hence, all non-zero terms in the Leibniz formula for the new matrix are equal to $m$ multiplied with a non-zero term in the Leibniz formula for the $(n + k - 1) \times (n + k - 1)$ minor of the first $n + k - 1$ rows and columns: 
	\begin{equation}
		\begin{vmatrix}
			1 & & & 1 & & \\
			1 & \ddots & & 1 & \ddots & \\
			m & & 1 & 1 & & 1 \\
			\vdots & & \vdots & \vdots & & \vdots \\
			\vdots & & \vdots & 1 & & \vdots  \\
			m^{n - 2} & & \vdots& m & & 1  \\
			& \ddots  & m^{n - 3} & & \ddots & 1  \\
			&  & m^{n - 2} & & &  m  \\
		\end{vmatrix}
	\end{equation}		
	This $(n + k- 1) \times (n + k - 1)$ minor is just the original matrix in (\ref{eqn:matrix_m}) but with $n - 1$ instead of $n$. Hence, from induction, all non-zero terms in the Leibniz formula for this matrix are divisible by $m^{n - 2}$. Once this is multiplied by the lower right $m$, then all of these terms are divisible by $m^{n - 1}$, as desired.
\end{proof}

Now, we can prove that $m^{n - 1} (m - 1)^{k - 1}$ divides $f_1(m, x, y)$, which completes the verification of Step 2 of the algorithm:

\begin{proposition}
	\label{prop:divisible}
	Let $f_1(m, x, y)$ be the resultant of the polynomials in (\ref{eqn:l=B=0_poly_equiv}) with respect to $g$. Then, $m^{n - 1} (m - 1)^{k - 1}$ divides $f_1(m, x, y)$ in $\C[m, x, y]$. 
\end{proposition}
\begin{proof}
	Consider the two polynomials in (\ref{eqn:l=B=0_poly_equiv}):
	For the first polynomial, the first term is homogeneous in $g$ and $m$ with degree $n$ and the second term is homogeneous in $g$ and $m$ with degree $n - 1$. Similarly, for the second polynomial, the first term is homogeneous in $g$ and $1 - m$ with degree $k$ and the second term is homogeneous in $g$ and $1 - m$ with degree $k - 1$. Thus, by expanding the products in $g$ and factoring out $m$ and $m - 1$ respectively, these polynomials can be written respectively as: 
	\begin{equation}
		\begin{aligned}
			& p_n(x, m) g^n  + p_{n - 1}(x, m) g^{n - 1} + p_{n - 2}(x, m) m g^{n - 2} + \\ 
			& \qquad + p_{n - 3}(x, m) m^2 g^{n - 3} \cdots + p_{0}(x, m) m^{n - 1} \\
			& q_k(y, m) g^k  + q_{k - 1}(y, m) g^{k - 1}  \qquad + q_{k - 2}(y, m) (1 - m) g^{k - 2} + \\ 
			&\qquad + q_{k - 3}(x, m) (1 - m)^2 g^{k - 3} \cdots + q_{0}(y, m) (1 - m)^{k - 1}  .
		\end{aligned}
	\end{equation}
	for some complex polynomials $p_i(x, m), q_i(y, m)$.
	
	Using that $a_1 + \cdots + a_n = b_1 + \cdots + b_k = 1$, then $p_0(x, m) = m - 1$ and $q_0(y, m) = - m$. 
	
	The resultant of these polynomials with respect to $g$ is exactly the determinant in Lemma \ref{prop:sylvester} with $a_0 = 1$, $b_0 = (-1)^k$,  $a_i = p_i(x, m)$ for $i = 1, \ldots, n$ and $b_i = q_i(y, m)$ for $i = 1, \ldots, k$. Hence, $m^{n - 1} (m - 1)^{k - 1}$ divides $f_1(m, x, y)$ in $\C[m, x, y]$.
\end{proof}

Returning to the verification of the algorithm, consider Step 3. $f_2(m_0, x_0, y_0) = 0$ is equivalent to $\Re f_2(m_0, x_0, y_0) = \Im f_2(m_0, x_0, y_0) = 0$.

Treating $\Re f_2(m, x_0, y_0)$ and $\Im f_2(m, x_0, y_0)$ as polynomials in $m$, then these polynomials have a common root at $m = m_0$. Hence, the resultant of these two polynomials vanishes in $m$ vanishes at $(x_0, y_0)$. This completes the proof of the correctness of the algorithm. 

\subsubsection{A specific case}
In order to prove Theorem \ref{thm:boundary_curve} for the generic case, we first apply the algorithm in the specific case where $\bm{\alpha} = (0, \ldots, 0) \in \R^n$ and $\bm{\beta} = (0, \ldots, 0) \in \R^k$. We will see that in this case, the algorithm produces a non-zero polynomial:

\begin{proposition}
	\label{prop:algorithm_specific case} 
	
	The algorithm for any $\bm{a}, \bm{b}$ and $\bm{\alpha} = \bm{0}$, $\bm{\beta} = \bm{0}$ produces a non-zero polynomial.
\end{proposition}
\begin{proof}

	For any $\bm{a}$ and $\bm{b}$, (\ref{eqn:l=B=0_poly_equiv}) is: 
	\begin{equation}
		\label{eqn:l=B=0_poly_equiv_specific}
		\begin{aligned}
			0 & = (gx + m)^n - (gx + m)^{n - 1} = (g x + m)^{n - 1}(gx + m - 1)  \\
			0 &= (i g y + 1 - m)^k - (i g y + 1 - m)^{k - 1} = (i g y + 1 - m)^{k - 1} (i g y - m) \,.
		\end{aligned}
	\end{equation}
	Step 1 of the algorithm is taking the resultant of the polynomials in (\ref{eqn:l=B=0_poly_equiv_specific}) with respect to $g$. From continuity, it suffices to take the resultant when $x \neq 0$ and $y \neq 0$. In this situation, we can factor the polynomials: 
	\begin{equation}
		\begin{aligned}
			(g x + m)^{n - 1}(gx + m - 1) & = x^n \left(g + \frac{m}{x} \right)^{n - 1} \left( g + \frac{m - 1}{x} \right) \\
			(i gy + 1 - m)^{k - 1}(i g y - m) &= (i y)^k \left(g - \frac{i (1 - m)}{y} \right)^{k - 1} \left(g + \frac{i m}{y}\right) \,.
		\end{aligned}
	\end{equation}
	From Proposition \ref{prop:roots_resultant_formula},
	\begin{equation}
		\begin{aligned}
			f_1(m, x, y) & = x^{n k} (i y)^{n k} \left( \frac{m}{x} + \frac{i (1 - m)}{y} \right)^{(n - 1)(k - 1)} \left( \frac{m}{x} - \frac{i m}{y} \right)^{n - 1} \times \\
			& \qquad \qquad \left( \frac{m - 1}{x} + \frac{i (1 - m)}{y} \right)^{k - 1}  \left( \frac{m - 1}{x} -  \frac{i m}{y} \right)
			\\ &= ( i m y + (m - 1) x)^{(n - 1)(k - 1)} ( i m y + m x )^{n - 1} \times  \\
			& \qquad \qquad (i (m - 1) y + (m - 1) x )^{k - 1} (i(m - 1)  y + m x) 	\\ 
			&= m^{n - 1} (m - 1)^{k - 1} (i m y + (m - 1) x)^{(n - 1)(k - 1)} (x + i y )^{n - 1} \times \\
			& \qquad \qquad ( x + i  y)^{k - 1} (i(m - 1)  y + m x) 
			\\ &= m^{n - 1} (m - 1)^{k - 1} (x + i y )^{n + k - 2} \times \\
			& \qquad \qquad ( (m - 1) x + i m y )^{(n - 1)(k - 1)}  (m x + i(m - 1)  y) \,.
		\end{aligned}
	\end{equation}
	In Step 2 of the algorithm, we divide $f_1(m, x, y)$ by $m^{n - 1}(m - 1)^{k - 1}$. This is easy to do from the final expression for $f_1(m, x, y)$: 
	\begin{equation}
		f_2(m, x, y) = (x + i y)^{n + k - 2} ( (m - 1) x + i m y )^{(n - 1)(k - 1)}  (m x + i(m - 1)  y)  \,.
	\end{equation}
	In Step 3 of the algorithm, we compute $\Re f_2(m, x, y)$ and $\Im f_2(m, x, y)$ assuming $m, x, y \in \R$. In Step 4, we compute the resultant of these two real polynomials in $m$, resulting in a polynomial in $x$ and $y$. We will not compute these explicitly, but just argue that the result of the algorithm is a non-zero polynomial.
	
	First, consider a general $f_2(m, x, y) \in \C[m, x, y]$:
	\begin{equation}
		f_2(m, x, y) = \sum_{j, k, l}^{} c_{j, k, l} m^j x^k y^l  \, , \qquad c_{j, k l} \in \C \,.
	\end{equation}
	The result of applying Step 3 of the algorithm to $f_2(m, x, y)$ is: 
	\begin{equation}
		\begin{aligned}
			\Re f_2(m, x, y) & = \sum_{j, k, l}^{} \Re(c_{j, k, l}) m^j x^k y^l \\
			\Im f_2(m, x, y) &=  \sum_{j, k, l}^{} \Im(c_{j, k, l}) m^j x^k y^l \,.
		\end{aligned}
	\end{equation}
	Thus, the following equalities hold as polynomials in $\C[m, x, y]$: 
	\begin{equation}
		\begin{aligned}
			f_2(m, x, y) & = \Re f_2(m, x, y) + i \, \Im f_2(m, x, y) \\
			\bar{f_2}(m, x, y) &= \Re f_2(m, x, y) - i \, \Im f_2(m, x, y) \,,
		\end{aligned}
	\end{equation}
	where $\bar{f_2}(m, x, y) \in \C[m, x, y]$ is computed assuming that $m, x, y \in \R$.
	
	In Step 4, we compute the resultant of $\Re f_2(m, x, y)$ and $\Im f_2(m, x, y)$ with respect to $m$ and this is the result of the algorithm. This resultant vanishes at $(x_0, y_0)$ if and only if there is some $m_0 \in \C$ such that $\Re f_2(m_0, x_0, y_0) = \Im f_2(m_0, x_0, y_0) = 0$. This happens if and only if there is some $m_0 \in \C$ such that $f_2(m_0, x_0, y_0) = \bar{f_2}(m, x_0, y_0) = 0$. Hence, it suffices to show there exists $(x_0, y_0) \in \R$ where there does not exist $m \in \C$ such that: 
	\begin{equation}
		\begin{aligned}
			0 & = f_2(m, x_0, y_0) \\
			& = (x_0 + i y_0)^{n + k - 2} ( (m - 1) x_0 + i m y_0 )^{(n - 1)(k - 1)}  (x_0 m + i(m - 1)  y_0) \\
			0 &= \bar{f_2}(m, x_0, y_0) \\
			& = (x_0 - i y_0)^{n + k - 2}  ( (m - 1) x_0 - i m y_0 )^{(n - 1)(k - 1)}  (x_0 m - i(m - 1)  y_0) \, .
		\end{aligned}
	\end{equation}
	As $f_2$ and $\bar{f_2}$ are factored, it is easy to see that 
	\begin{equation}
		\begin{aligned}
			f_2(m, x_0, y_0) & = 0 \; \iff \; x_0 = y_0 = 0 \quad \text{or} \quad m = \frac{x_0}{x_0 + i y_0} \quad \text{or} \quad m = \frac{i y_0}{x_0 + i y_0} \\
			\bar{f_2}(m, x_0, y_0) &= 0 \; \iff \; x_0 = y_0 = 0 \quad \text{or} \quad m = \frac{x_0}{x_0 - i y_0} \quad \text{or} \quad m = \frac{-i y_0}{x_0 - i y_0} \, .
		\end{aligned}
	\end{equation}
	Consider $x_0 + i y_0 = e^{i \theta}$. Then, 
	\begin{equation}
		\begin{aligned}
			\frac{x_0}{x_0 + i y_0} & = \cos \theta e^{- i \theta} \, , \qquad \frac{i y_0}{x_0 + i y_0} = i \sin \theta e^{- i \theta} = \sin \theta e^{i (\pi / 2 - \theta)} \\
			\frac{x_0}{x_0 - i y_0} & = \cos \theta e^{i \theta} \, , \qquad \frac{-i y_0}{x_0 - i y_0} = -i \sin \theta e^{i \theta} = \sin \theta e^{i (\theta - \pi / 2)} \, .	
		\end{aligned}
	\end{equation}
	The roots of $f_2(m, x_0, y_0)$ and $\bar{f_2}(m, x_0, y_0)$ occur at angles $- \theta, \pi / 2 - \theta, \theta, \theta - \pi / 2$. When $\theta \in (0, \pi / 4)$, these four angles are all distinct because $\theta - \pi / 2 < - \theta < \theta < \pi / 2 - \theta$. Thus, $f_2(m, x_0, y_0)$ and $\bar{f_2}(m, x_0, y_0)$ cannot have a common root, and we conclude that the algorithm produces a non-zero polynomial in this instance. 
\end{proof}

\subsubsection{Extending to generic case}
We are now ready to prove Theorem \ref{thm:boundary_curve}: 

\begin{proof}[Proof of Theorem \ref{thm:boundary_curve}]
	We have already presented the algorithm and proved that it does produce a real two-variable that vanishes on $\Omega_{p, q}$. 
	
	All that is left to prove is that for any fixed $\bm{a}, \bm{b}$, for Lebesgue almost every $(\bm{\alpha}, \bm{\beta}) \in \R^n \times \R^k$, the polynomial from the algorithm is non-zero.
	
	First, change coordinates as follows: 
	
	Let $S^{n - 1}$ and $S^{k - 1}$ be the unit spheres in $\R^n$ and $\R^{k}$:
	\begin{equation}
		\begin{aligned}
			S^{n - 1} &= \{\bm{u} \in \R^n : \norm{\bm{u}} = 1  \} \\
			S^{k - 1} & = \{\bm{v} \in \R^k : \norm{\bm{v}} = 1\}  \,.
		\end{aligned}
	\end{equation}
	Consider the map $\phi: [0, \infty) \times S^{n - 1} \times [0, \infty) \times S^{k - 1} \to \R^n \times \R^k$ given by: 
	\begin{equation}
		\phi \left(  \left( r, \bm{u}, s, \bm{v} \right)  \right)  = (r \bm{u}, s \bm{v}) \,.
	\end{equation}
	Endow $[0, \infty) \times S^{n - 1} \times [0, \infty) \times S^{k - 1}$ with the product of the Lebesgue measures on the intervals and the normalized spherical measures and endow $\R^n \times \R^k$ with the usual Lebesgue measure. From the Change of Variables formula, $\phi$ maps sets of measure $0$ to sets of measure $0$. Hence, it suffices to prove the generic statement of the Theorem in the $\left( r, \bm{u}, s, \bm{v} \right)$ coordinates.
	
	It suffices to show that for any $\bm{a}, \bm{b}$ and $(\bm{u}, \bm{v}) \in S^{n - 1} \times S^{k - 1}$ and Lebesgue almost every $(r, s) \in [0, \infty) \times [0, \infty)$, the algorithm applied to $\bm{a}, \bm{b}, \bm{\alpha} = r \bm{u}, \bm{\beta} = s \bm{v}$ produces a non-zero polynomial. 
	
	Fix $\bm{a}, \bm{b}$ and $(\bm{u}, \bm{v}) \in S^{n - 1} \times S^{k - 1}$. It is straightforward to check that if we consider $\bm{\alpha}$ and $\bm{\beta}$ as functions of $r$ and $s$, respectively: $\bm{\alpha}(r) = r \bm{u}$, $\bm{\beta}(s) = s \bm{v}$, then the algorithm produces a real polynomial in $x, y, r, s$. Additionally, doing the algorithm and evaluation at a specific $r = r_0$, $s = s_0$ commute. 
	
	Hence, the algorithm produces a polynomial $p(x, y, r, s)$, where from Proposition \ref{prop:algorithm_specific case}, $p(x, y, \bm{0}, \bm{0})$ is a non-zero polynomial.
	
	Viewing $p(x, y, r, s)$ as a polynomial in $x, y$ with coefficients in $\R[r, s]$, we see that there is at least one coefficient that is not the zero polynomial. As the zero set of any non-zero polynomial in $\R[r, s]$ is Lebesgue measure $0$, then for Lebesgue almost every $(r, s) \in [0, \infty) \times [0 \infty)$, $p(x, y, r, s)$ has a non-zero coefficient. Hence, for almost every $r, s$, the result of the algorithm is a non-zero polynomial. 
\end{proof}

\section{Support of the Brown measure}
\label{sec:support}
In this section, we consider Heuristic \ref{heur:support} about the support of the Brown measure of $X = p + i q$, where $p, q \in (M, \tau)$ are Hermitian and freely independent. 

The main result of this section is that in the case when $p, q$ have $2$ atoms that have equal weights, then $X = p + i q$ satisfies this property when we restrict to points where we can use $\mathcal{B}_{\mathbf{X}}$: 

\begin{theorem}
	\label{thm:B_not_0}
	Suppose that $p, q \in (M, \tau)$ are Hermitian, freely independent operators such that their spectral measures are: 
	\begin{equation}
		\begin{aligned}
			\mu_p & = (1/2) \delta_\alpha + (1/2) \delta_{\alpha'} \\
			\mu_q & = (1/2) \delta_\beta + (1/2) \delta_{\beta'}
		\end{aligned}
	\end{equation}
	for some $\alpha \neq \alpha', \beta \neq \beta' \in \R$.
	
	Considering only points $z \in \C$ where $\mathcal{G}_{\mathbf{X}}(z_\epsilon)$ is in the domain of $\mathcal{B}_{\mathbf{X}}$ for sufficiently small $\epsilon > 0$, the support of the Brown measure of $X = p + i q$ is the closure of the set of $z$ such that 
	\begin{equation}
		\lim\limits_{\epsilon \to 0^+}
		\mathcal{G}_\mathbf{X}
		\left( z_\epsilon \right) 
		=
		\begin{pmatrix}
			A & i \bar{B} \\
			i B & \bar{A}
		\end{pmatrix}
	\end{equation}
	for some $B \neq 0$ or where the limit does not exist. 
\end{theorem}

Recall that from Theorem \ref{thm:brown_measure_p+iq} we know what the support of the Brown measure of $X$ is in this situation, it is the intersection of the hyperbola
\begin{equation}
	\label{eqn:hyperbola}
	\left\lbrace z = x + i y :  \left( x - \frac{\alpha + \alpha'}{2}  \right)^2 - \left(  y - \frac{\beta + \beta'}{2} \right)^2 = \frac{(\alpha' - \alpha)^2 - (\beta' - \beta)^2}{4}    \right\rbrace
\end{equation}
with the rectangle
\begin{equation}
	\label{eqn:rectangle}
	\left\lbrace z = x + i y: x \in [ \alpha \wedge \alpha', \alpha \vee \alpha'   ]  , y \in [\beta \wedge \beta', \beta \vee \beta'  ] \right\rbrace \,.
\end{equation}
\begin{remark}
	In the following sections comprising the proof of Theorem \ref{thm:B_not_0}, we will assume $z \in \C$ is as described in the Theorem: $\mathcal{G}_{\mathbf{X}}(z_\epsilon)$ is in the domain of $\mathcal{B}_{\mathbf{X}}$ for sufficiently small $\epsilon > 0$.
\end{remark}

We give a brief outline of the proof: 

Fix some $z$ and consider a sequence $\epsilon_k \to 0^+$. Recall that we use the notation $Q_\epsilon = \mathcal{G}_{\mathbf{X}}(z_\epsilon)$. There are some preliminary steps to reduce to the case where $Q_{\epsilon_k} \to Q \in \mathbb{H}$, $Q \neq 0$. This is discussed in the next subsection.

After passing to a subsequence, the following two cases follow from (\ref{eqn:lib}):

\begin{enumerate}
	\item There exists a sequence $\epsilon_k \to 0^+$ where $l_{\epsilon_k} \to 0$.
	\item There exists a sequence $\epsilon_k \to 0^+$ where $B_{\epsilon_k} \to 0$.
\end{enumerate} 

We will classify which $z$ is in each of these two cases: these cases impose conditions on the limit $Q$, which in turn impose conditions on $z$. Note that $\mathcal{B}_{\mathbf{X}}$ may not be well-defined and/or discontinuous at $Q$. The proof of Theorem \ref{thm:B_not_0} follows once all of these cases are understood. 

Before we continue, let us highlight that letting $a = b = 1/2$ in Definitions \ref{def:d_p_d_q} and \ref{def:I_q_I_q} makes $D_p$, $D_q$, $I_p$, and $I_q$ particularly simple: 
\begin{equation}
	\begin{aligned}
		D_p(w) & = ((\alpha' - \alpha)w )^2 + 1 \\
		D_q(w) &= ((\beta' - \beta)w)^2 + 1 \,.
	\end{aligned}
\end{equation}
\begin{equation}
	\begin{aligned}
		I_p 
		&= \left\lbrace   i y: \Abs{y} >  \frac{1 }{\Abs{\alpha' -  \alpha}}   \right\rbrace \\
		I_q
		&= \left\lbrace  i y: \Abs{y} >  \frac{1 }{\Abs{\beta' -  \beta}}   \right\rbrace \,.
	\end{aligned}
\end{equation} 
In particular, observe that $I_p , I_q \subset i \R$.

Additionally, the \textbf{\textit{hyperbola}} will always mean (\ref{eqn:hyperbola}), and the \textbf{\textit{rectangle}} will always mean (\ref{eqn:rectangle}).

\subsection{Preliminary reductions}
In this subsection, we will reduce to the case where $Q_{\epsilon_k}$ converges to some $Q \in \mathbb{H}$, $Q \neq 0$.

For a general sequence $\{Q_k\}$, there are three cases:

\begin{enumerate}
	\item The sequence $\{Q_k\}$ is not bounded.
	\item The sequence $\{Q_k\}$ converges to $0$.
	\item The sequence $\{Q_k\}$ is bounded but does not converge to $0$.
\end{enumerate}

By bounded/unbounded, we refer to the boundedness/unboundedness of the quaternionic norm. 

In the third case, we may just pass to a subsequence where $Q_k \to Q$, where $Q \neq 0$, which is what we desired. 

The second case is not possible, for $Q_k = Q_{\epsilon_k}$ as from Theorem \ref{thm:BX_cty},
\begin{equation}
	\Abs{z} = \lim\limits_{k \to \infty} \Abs{z_{\epsilon_k}} =  \lim\limits_{k \to \infty} \Abs{\mathcal{B}_{\mathbf{X}}(Q_{\epsilon_k})} = \infty \,.
\end{equation}
Thus, all that remains is the first case, where we may pass to a subsequence and assume that $\Abs{Q_k} \to \infty$. This is the subject of the following Proposition: 

\begin{proposition}
	\label{prop:unbounded_Q}
	Consider a sequence $\{Q_k\} \subset \mathbb{H}$ where $\Abs{Q_k} \to \infty$ and $\mathcal{B}_{\mathbf{X}}(Q_k)$ converges. Then, $\mathcal{B}_{\mathbf{X}}(Q_k)$ converges to one of: $\{\alpha + i \beta, \alpha + i \beta', \alpha' + i \beta, \alpha' + i \beta'\}$. 
\end{proposition}
\begin{proof}
	From (\ref{eqn:bX_free}) and $\Abs{Q_k} \to \infty$ implying $\Abs{Q_k^{-1} } \to 0$, it suffices to analyze
	\begin{equation}
		\lim\limits_{k \to \infty } \mathcal{B}_{\mathbf{p}}(Q_k)  \qquad \lim\limits_{k \to \infty} \mathcal{B}_{i \mathbf{q}}(Q_k) \,.
	\end{equation}	
	Using Proposition \ref{prop:B_cX} and noting that $\Abs{Q_k i} = \Abs{Q_k}$, then it suffices just to analyze the first limit and apply the result to the second limit.
	
	We proceed to show that there exists a subsequence $Q_{k_j}$ where $\mathcal{B}_\mathbf{p}(Q_{k_j})$ converges to one of $\alpha, \alpha'$. From Lemma \ref{lem:Q_real_convergence}, it suffices to show that the eigenvalues of $\mathcal{B}_\mathbf{p}(Q_{k_j})$, $B_p(g_{k_j}), B_p(\bar{g_{k_j}}) = \bar{B_p(g_{k_j})}$ converge to one of $\alpha, \alpha'$. 
	
	Since $\Abs{Q_k} \to \infty$, then $\Abs{g_k} \to \infty$ also. From the expression for $B_p$, (Proposition \ref{prop:complex_b_formula}),
	\begin{equation}
		\begin{aligned}
			\lim\limits_{k \to \infty} B_p(g_k) & = \lim\limits_{k \to \infty} \frac{\alpha + \alpha'}{2} + \frac{1 + \sqrt{D_p(g_k)}}{2 g_k} \\
			& = \frac{\alpha + \alpha'}{2} + \lim\limits_{k \to \infty} \frac{\sqrt{(\alpha' - \alpha)^2 g_k^2 + 1  }}{2 g_k} \,.
		\end{aligned}
	\end{equation} 
	The square of the quantity inside the final limit is: 
	\begin{equation}
		\frac{(\alpha' - \alpha)^2 g_k^2 + 1}{4 g_k^2} \,,
	\end{equation}
	which converges to $(\alpha' - \alpha)^2 / 4$ as $k \to \infty$. Hence, we may choose a subsequence $g_{k_j}$ such that 
	\begin{equation}
		\lim\limits_{j \to \infty} \frac{\sqrt{(\alpha' - \alpha)^2 g_{k_j}^2 + 1  }}{2 g_{k_j}} = \pm \frac{\alpha' - \alpha}{2} \,.
	\end{equation}
	Then, 
	\begin{equation}
		\lim\limits_{j \to \infty} B_p(g_{k_j}) = \frac{\alpha + \alpha'}{2} \pm \frac{\alpha' - \alpha}{2} \, ,
	\end{equation}
	i.e. $B_p(g_{k_j})$ converges to one of $\alpha, \alpha'$. Hence, $\mathcal{B}_\mathbf{p}(Q_{k_j})$ converges to one of $\alpha, \alpha'$. 
	
	By applying the same argument to $\mathcal{B}_{i \mathbf{q}}(Q_k)$, there exists a subsequence $k_{j_l}$ where $\mathcal{B}_{\mathbf{X}}(Q_{k_{j_l}})$ converges to one of $\{\alpha + i \beta, \alpha + i \beta', \alpha' + i \beta, \alpha' + i \beta'\}$. Hence, $\mathcal{B}_{\mathbf{X}}(Q_k)$ also converges to one of $\{\alpha + i \beta, \alpha + i \beta', \alpha' + i \beta, \alpha' + i \beta'\}$.
\end{proof}

Thus, for what follows, we may assume that we are considering a sequence $Q_k \to Q$ and $Q \neq 0$.

We return to the original dichotomy, which follows from (\ref{eqn:lib}):

\begin{enumerate}
	\item There exists a sequence $\epsilon_k \to 0^+$ where $l_{\epsilon_k} \to 0$.
	\item There exists a sequence $\epsilon_k \to 0^+$ where $B_{\epsilon_k} \to 0$.
\end{enumerate}

\subsection{$l_k \to 0$}
In this subsection, we determine the $z$ such that there exists a sequence $\epsilon_k \to 0^+$ where $Q_{\epsilon_k} \to Q \neq 0$ and $l_{\epsilon_k} \to 0$.

We summarize the results of the three cases when $l_k \to 0$:

\begin{enumerate}
	\item If $\mathcal{B}_{\mathbf{X}}$ is discontinuous at $Q$, then $\mathcal{B}_{\mathbf{X}}(Q_k)$ converges to $z$ on the intersection of the hyperbola with the open rectangle. (Proposition \ref{prop:l_0_discontinuous_limit})
	\item If $\mathcal{B}_{\mathbf{X}}$ is continuous at $Q$:
	\begin{enumerate}
		\item If $g \in \R$ or $g^I \in \R$, then $l(Q_k) \not \to 0$. (Proposition \ref{prop:l_not_0})
		\item If $g \not \in \R$ and $g^I \not \in \R$, then $\mathcal{B}_{\mathbf{X}}(Q_k) \to \mathcal{B}_{\mathbf{X}}(Q) = z$, which is on the intersection of the hyperbola with the open rectangle.  (Proposition \ref{prop:l_0_continuous_limit})
	\end{enumerate}
\end{enumerate}

For the first case, we have the following result: 

\begin{proposition}
	\label{prop:l_0_discontinuous_limit}
	If $Q_k$ converges to $Q \neq 0$, $l_k \to 0$, and $\mathcal{B}_{\mathbf{X}}$ is discontinuous at $Q$, then one of the following is true: 
	
	\begin{enumerate}
		\item $g_k \to I_p$ and $g_k^I \to \bar{I_q}$.
		\item $g_k \to \bar{I_p}$ and $g_k^I \to I_q$.
	\end{enumerate}
	
	In either case, $\mathcal{B}_{\mathbf{X}}(Q_k)$ converges to $z = x + i y$ on the intersection of the hyperbola 
	\begin{equation}
		H = \left\lbrace z = x + i y :  \left( x - \frac{\alpha + \alpha'}{2}  \right)^2 - \left(  y - \frac{\beta + \beta'}{2} \right)^2 = \frac{(\alpha' - \alpha)^2 - (\beta' - \beta)^2}{4}    \right\rbrace
	\end{equation}
	with the open rectangle
	\begin{equation}
		\interior{R} = \left\lbrace z = x + i y: x \in  (\alpha \wedge \alpha', \alpha \vee \alpha')     , y \in (\beta \wedge \beta', \beta \vee \beta' ) \right\rbrace \,.
	\end{equation} 
\end{proposition}
\begin{proof}
	From Theorem \ref{thm:BX_cty}, if $\mathcal{B}_{\mathbf{X}}$ is discontinuous at $Q$, then either $g_k$ converges to $I_p$ or $g_k^I$ converges to $I_q$. We will prove the second case, the first case is similar. 
	
	Recall that for $g_k, g_k^I \not \in \R$,
	\begin{equation}
		l_k = \frac{1}{2 \Abs{g_k}^2} \left( \frac{\bar{g_k} \sqrt{D_p(g_k)} - g_k \sqrt{D_p(\bar{g_k})} }{g_k - \bar{g_k}} + \frac{\bar{g_k^I} \sqrt{D_q(g_k^I)} - g_k^I \sqrt{D_q(\bar{g_k^I})} }{g_k^I - \bar{g_k^I}} \right) \,, 
	\end{equation}
	and there exists continuous extensions in the case where $g_k \in \R \setminus \{0\}$ or $g_k^I \in \R \setminus \{0\}$ (see Proposition \ref{prop:l}).
	
	If $g_k^I$ converges to $I_q$, then from Lemma \ref{lem:cty3}, the second term inside the parenthesis converges to $0$. As $l_k \to 0$, then the first term inside the parentheses (or its continuous extension to $\R$) also must converge to $0$: 
	\begin{equation}
		\tilde{l} (g_k) =  
		\begin{dcases}
			\frac{\bar{g_k} \sqrt{D_p(g_k)} - g_k \sqrt{D_p(\bar{g_k})} }{g_k - \bar{g_k}} & g_k \not \in \R \\
			\frac{- 1}{\sqrt{D_p(g_k)}} & g_k \in \R
		\end{dcases} \longrightarrow 0 \,.
	\end{equation}
	Let $g_k$ converge to $g \neq 0$. From Lemma \ref{lem:cty3}, $\tilde{l}$ is continuous on $\C \setminus \{0\}$ when $a = 1/2$. Hence, $\tilde{l}(g) = 0$. We proceed to show that $g \in \bar{I_p}$ by considering what happens if $g \in \C \setminus \bar{I_p}$: 
	
	If $g \in \R$, then 
	\begin{equation}
		\tilde{l}(g) = \frac{- 1}{D_p(g)} < 0 \,.
	\end{equation} 
	If $g \in \C \setminus (\R \cup \bar{I_p})$, then consider the following equivalences:
	\begin{equation}
		\begin{aligned}
			\tilde{l}(g) = 0 & \iff \Im \left( \bar{g} \sqrt{D_p(g)} \right)   = 0 \\
			& \iff \bar{g} \sqrt{D_p(g)} = t , \qquad t \in \R \\
			& \iff \Abs{g}^2 \sqrt{D_p(g)} = t g  \\
			& \iff \sqrt{D_p(g)} = \frac{t}{\Abs{g}^2} g    \\
			& \Longrightarrow (\alpha' - \alpha)^2 g^2 + 1 = \frac{t^2}{\Abs{g}^4} g^2    \\
			& \iff g^2 = \frac{1}{ \frac{t^2}{\Abs{g}^4} - (\alpha' - \alpha)^2   } \,.
		\end{aligned}
	\end{equation}
	The denominator in the last term is non-zero, or else the equality in the previous line is incorrect. Thus, for $\tilde{l}(g) = 0$, $g^2 \in \R$. As we assumed $g \not \in \R$, then $g \in i \R$. Since we assumed $g \not \in \bar{I_p}$, then $g = i y$ for some $y \in \R \setminus \{0\}$, $y < 1 / \Abs{\alpha' - \alpha}$. But, for such a $g$, $D_p(g) > 0$, so
	\begin{equation}
		\Im \left( \bar{g} \sqrt{D_p(g)} \right)  =  - \sqrt{D_p(i y)} y  \neq  0 \,.
	\end{equation}
	Hence, $\tilde{l}(g) \neq 0$. 
	
	Similar analysis switching $p, g$ for $q, g^I$ shows that there are two possibilities: 
	
	\begin{enumerate}
		\item $g_k \to \bar{I_p}$ and $g_k^I \to I_q$.
		\item $g_k \to I_p$ and $g_k^I \to \bar{I_q}$.
	\end{enumerate}
	
	In either situation, $g_k, g_k^I$ converge to $i \R$, so that in the real coefficients of $Q_k$, $(x_0)_k \to 0$ and $(x_3)_k \to 0$. In particular, from (\ref{eqn:quaternion_eigenvalues}) and (\ref{eqn:gI}) this implies that $g_k$ and $g_k^I$ both converge to some $i t$, $t > 0$.   Since $l_k \to 0$, then using (\ref{eqn:blue_formula_expanded}),
	\begin{equation}
		\lim\limits_{k \to \infty} \mathcal{B}_{\mathbf{X}}(Q_k) = \lim\limits_{k \to \infty} \beta_p(g_k) + i \beta_q(g_k^I) \,.
	\end{equation}
	From Lemma \ref{prop:beta_formula}, depending on if $g_k$ and $g_k^I$ approach $i t$ from the left or right, all $4$ limits are possible: 
	\begin{equation}
		\begin{aligned}
			\lim\limits_{k \to \infty} \mathcal{B}_{\mathbf{X}}(Q_k) 
			& = \lim\limits_{k \to \infty} \beta_p(g_k) + i \beta_q(g_k^I) \\
			& =  \left( \frac{\alpha + \alpha'}{2} \pm \frac{\sqrt{- D_p(i t)}}{2t} \right)  + i \left( \frac{\beta + \beta'}{2}  \pm \frac{\sqrt{- D_p(i t)}}{2t}  \right) \,.
		\end{aligned}
	\end{equation}
	Hence, if the limit is $z = x + i y$, then 
	\begin{equation}
		\begin{aligned}
			x & = \frac{\alpha + \alpha'}{2} \pm \frac{\sqrt{- D_p(i t)}}{2t} \\
			y & = \frac{\beta + \beta'}{2}  \pm \frac{\sqrt{- D_q(i t)}}{2t} \,.
		\end{aligned}
	\end{equation}
	Finally, we check that this $z$ is on the intersection of the hyperbola and the open rectangle. For the hyperbola equation: 
	\begin{equation}
		\begin{aligned}
			& \left( x - \frac{\alpha + \alpha'}{2}  \right)^2 - \left(  y - \frac{\beta + \beta'}{2} \right)^2	\\
			& = \left( \frac{\sqrt{- D_p(i t)}}{2t} \right) ^2 - \left(  \frac{\sqrt{- D_q(i t)}}{2t}\right)^2 \\
			& =  \left( \frac{\sqrt{ (\alpha' - \alpha)^2 t^2 - 1  }}{2t} \right) ^2 - \left(  \frac{\sqrt{ (\beta' - \beta)^2 t^2 - 1 } } {2t}\right)^2 \\
			&= \frac{(\alpha' - \alpha)^2 t^2 - 1}{4t^2} - \frac{(\beta' - \beta)^2 t^2 - 1}{4t^2} \\
			&= \frac{(\alpha' - \alpha)^2 - (\beta' - \beta)^2}{4} \,.
		\end{aligned}
	\end{equation}
	For the rectangle condition, from Lemma \ref{lem:hyperbola_rectangle} it suffices to check $x \in (\alpha \wedge \alpha', \alpha \vee \alpha')$:
	\begin{equation}
		\begin{aligned}
			\Abs{x - \frac{\alpha + \alpha'}{2}} 
			& = \frac{\sqrt{- D_p(i t)}}{2t} \\ 
			&= \frac{\sqrt{ (\alpha' - \alpha)^2 t^2 - 1  }}{2t} \\
			&= \sqrt{ \frac{(\alpha' - \alpha)^2 t^2 - 1}{4 t^2}  } \\ 
			&= \sqrt{ \frac{(\alpha' - \alpha)^2}{4 } - \frac{1}{4 t^2}  } \\
			& < \sqrt{\frac{(\alpha' - \alpha)^2}{4 }} \\
			& = \frac{\Abs{\alpha' - \alpha'}}{2} \,.
		\end{aligned}
	\end{equation}
\end{proof}

Now, consider the case where $\mathcal{B}_{\mathbf{X}}$ is continuous at $Q$ and $g \in \R$ or $g^I \in \R$. We can rule out this case from happening, so we can use (\ref{eqn:blue_formula_expanded}) to analyze the general case: 

\begin{proposition}
	\label{prop:l_not_0}
	If $Q_k$ converges to $Q \neq 0$, $\mathcal{B}_{\mathbf{X}}$ is continuous at $Q$, and $g \in \R$ or $g^I \in \R$, then $l_k \not \to 0$.
\end{proposition}
\begin{proof}
	The cases $g \in \R$ and $g^I \in \R$ are similar, so we will just prove the case when $g \in \R$. From Proposition \ref{prop:l} and Lemma \ref{lem:cty3}, $l$ can be extend continuously to $(g, g^I)$ and $l_k$ converges to $l(g, g^I)$:
	\begin{equation}
		l(g, g^I) = \frac{1}{2 \Abs{g}^2} \left( \frac{-1}{\sqrt{D_p(g)}} + \frac{\Im (\bar{g^I} \sqrt{D_q(g^I)})}{\Im (g^I)} \right) \,.
	\end{equation}
	From Lemma \ref{lem:conjugation}, we can just assume that $g^I = i g$ for the purpose of evaluating $l$: 
	\begin{equation}
		l(g, g^I) = l(g, i g) = \frac{1}{2 \Abs{g}^2} \left( \frac{-1}{\sqrt{D_p(g)}} + \frac{\Im (\bar{i g} \sqrt{D_q(i g)})}{\Im (i g)} \right) \,.
	\end{equation}	
	Since $\mathcal{B}_{\mathbf{X}}$ is continuous at $Q$, $g^I \not \in I_q$, so $i g \not \in I_q$. Combining this with $g \in \R$, then $D_q(i g) \geq 0$, and
	\begin{equation}
		\frac{\Im (\bar{i g} \sqrt{D_q(i g)})}{\Im (i g)}  = - \sqrt{D_q(ig)} \leq 0 \,.
	\end{equation}
	Therefore, 
	\begin{equation}
		l(g, i g) \leq \frac{1}{2 \Abs{g}^2 }  \left( \frac{- 1}{\sqrt{D_p(g)}} \right)  < 0 \,.
	\end{equation}
\end{proof}

For the final case, consider the following Proposition: 

\begin{proposition}
	\label{prop:l_0_continuous_limit}
	If $Q_k$ converges to $Q \neq 0$, $l_k \to 0$, and $\mathcal{B}_{\mathbf{X}}$ is continuous at $Q$, then $\mathcal{B}_{\mathbf{X}}(Q_k)$ converges to $\mathcal{B}_{\mathbf{X}}(Q) = z = x + i y$ on the intersection of the hyperbola 
	\begin{equation}
		H = \left\lbrace z = x + i y :  \left( x - \frac{\alpha + \alpha'}{2}  \right)^2 - \left(  y - \frac{\beta + \beta'}{2} \right)^2 = \frac{(\alpha' - \alpha)^2 - (\beta' - \beta)^2}{4}    \right\rbrace
	\end{equation}
	with the open rectangle
	\begin{equation}
		\interior{R} = \left\lbrace z = x + i y: x \in  (\alpha \wedge \alpha', \alpha \vee \alpha')     , y \in (\beta \wedge \beta', \beta \vee \beta' ) \right\rbrace \,.
	\end{equation} 
\end{proposition}
\begin{proof}
	Since $\mathcal{B}_{\mathbf{X}}$ is continuous at $Q$, then
	\begin{equation}
		\lim\limits_{k \to \infty} \mathcal{B}_{\mathbf{X}}(Q_k) = \mathcal{B}_{\mathbf{X}}(Q) \,.
	\end{equation}
	Since $Q \neq 0$, then from Proposition \ref{prop:l}, $l$ is continuous at $Q$, so
	\begin{equation}
		l(Q) = \lim\limits_{k \to \infty} l_k = 0 \,.
	\end{equation} 
	From Proposition \ref{prop:l_not_0}, since $l_k \to 0$, then $g, g^I \not \in \R$, i.e. $Q \not \in \R \cup i \R$.
	
	Hence, (\ref{eqn:blue_formula_expanded}) applies and shows that $\mathcal{B}_{\mathbf{X}} = z$ for some $z \in \C$.
	
	Applying Proposition \ref{prop:l_0_compute} completes the proof.
\end{proof}

\subsection{$B_k \to 0$}
In this subsection, we determine the $z$ such that there exists a sequence $\epsilon_k \to 0^+$ where $Q_{\epsilon_k} \to Q \neq 0$ and $B_{\epsilon_k} \to 0$, i.e. $Q \in \C$.

We summarize the results of the two cases when $B_k \to 0$:

\begin{enumerate}
	\item If $\mathcal{B}_{\mathbf{X}}$ is continuous at $Q$, then $\mathcal{B}_{\mathbf{X}}(Q_k) \to \mathcal{B}_{\mathbf{X}}(Q) = z$, and whenever $z$ is on the hyperbola, $z$ is not in the rectangle. (Proposition \ref{prop:B_0_continuous_limit})
	\item If $\mathcal{B}_{\mathbf{X}}$ is discontinuous at $Q$, then $\mathcal{B}_{\mathbf{X}}(Q_k)$ does not converge to a $z$ on the hyperbola. (Proposition \ref{prop:B_0_discontinuous_limit})
\end{enumerate}

Recall that $g \in \C \subset \mathbb{H}$ is a continuity point of $\mathcal{B}_{\mathbf{X}}$ if and only if $g \not \in I_p \cup \{0\}$ and $g^I \not \in I_q \cup \{0\}$.

For the first case, we have the following result: 

\begin{proposition}
	\label{prop:B_0_continuous_limit}
	If $Q_k$ converges to $Q \in \C$ where $Q \neq 0$ and  $\mathcal{B}_{\mathbf{X}}$ is continuous at $Q$, then $\mathcal{B}_{\mathbf{X}}(Q_k)$ converges to $ \mathcal{B}_{\mathbf{X}}(Q) = z = x + i y$ such that:
	\begin{equation}
		\left( x - \frac{\alpha + \alpha'}{2}  \right)^2 - \left(  y - \frac{\beta + \beta'}{2} \right)^2 = \frac{(\alpha' - \alpha)^2 - (\beta' - \beta)^2}{4} \,,
	\end{equation}
	then 
	\begin{equation}
		\Abs{ \left( x - \frac{\alpha + \alpha'}{2}  \right) \left(  y - \frac{\beta + \beta'}{2} \right) } > \frac{(\alpha' - \alpha)(\beta' - \beta) }{4} \,.
	\end{equation}
	In particular, this implies that $z$ is not on the intersection of the hyperbola and rectangle.
\end{proposition}
\begin{proof}
	Since $\mathcal{B}_{\mathbf{X}}$ is continuous at $Q$, then $\mathcal{B}_{\mathbf{X}}(Q_k) \to \mathcal{B}_{\mathbf{X}}(Q)$. Since $Q \in \C$, then $\mathcal{B}_{\mathbf{X}}(Q) = z \in \C$ also. Let $Q = g \in \C$. 
	
	Applying the addition law (\ref{eqn:bX_free}), 
	\begin{equation}
		z = \mathcal{B}_{\mathbf{X}}(g) = \mathcal{B}_{\mathbf{p}}(g) + i \mathcal{B}_{\mathbf{q}}(i g) - \frac{1}{g} \,.
	\end{equation}
	Hence,
	\begin{equation}
		\begin{aligned}
			z &= B_X(g) \\
			& = B_p(g) + i B_q(i g) - \frac{1}{g}  \\
			&= \left( \frac{\alpha + \alpha'}{2} + \frac{1 + \sqrt{ (\alpha' - \alpha)^2 g^2 + 1  }}{2 g}\right) + \\ 
			& \qquad \qquad  i \left( \frac{\beta + \beta'}{2} + \frac{1 + \sqrt{ (\beta' - \beta)^2 (i g)^2 + 1  } }{2 i g   } \right) - \frac{1}{g} \\
			&= \frac{\alpha + \alpha'}{2} + i \frac{\beta + \beta'}{2}  + \frac{ \sqrt{(\alpha' - \alpha )^2 g^2 + 1 }  }{2 g} + \frac{  \sqrt{1 - (\beta' - \beta)^2 g^2}}{2 g} \,.
		\end{aligned}
	\end{equation}
	
	We proceed to show that if $z$ lies on the hyperbola in (\ref{eqn:hyperbola}), then $z$ lies outside of the rectangle in (\ref{eqn:rectangle}). 
	
	Let
	\begin{equation}
		\begin{aligned}
			\mathscr{A} & = \alpha' - \alpha \\
			\mathscr{B} &= \beta' - \beta \,.
		\end{aligned}
	\end{equation}
	From Lemma \ref{lem:hyperbola_rectangle}, the equation of the hyperbola for $z = x + i y$ can be written as
	\begin{equation}
		\Re \left( \left( z - \frac{\alpha + \alpha'}{2} - i \frac{\beta + \beta'}{2}  \right)^2  - \frac{\mathscr{A}^2 - \mathscr{B}^2}{4} \right) =   0 \,.
	\end{equation}
	For $z = x + i y$ on the hyperbola, $z$ is on the rectangle if and only if 
	\begin{equation}
		\Abs{\Im \left( \left( z - \frac{\alpha + \alpha'}{2} - i \frac{\beta + \beta'}{2}  \right)^2  - \frac{\mathscr{A}^2 - \mathscr{B}^2}{4} \right)}  \leq \frac{\mathscr{A} \mathscr{B}}{2} \,.
	\end{equation}
	The relevant quantity can be simplified to:  
	\begin{equation}
		\label{eqn:prop_B=0_0}
		\left( z - \frac{\alpha + \alpha'}{2} - i \frac{\beta + \beta'}{2}  \right)^2  - \frac{\mathscr{A}^2 - \mathscr{B}^2}{4} = \frac{1 + \sqrt{ \mathscr{A}^2 g^2 + 1 } \sqrt{1 - \mathscr{B}^2 g^2}}{2 g^2}
	\end{equation}
	Hence, the hyperbola equation is: 
	\begin{equation}
		\label{eqn:prop_B=0_hyperbola}
		\Re \left( \frac{1 + \sqrt{ \mathscr{A}^2 g^2 + 1 } \sqrt{1 - \mathscr{B}^2 g^2}}{2 g^2} \right) = 0  \,,
	\end{equation}
	i.e. 
	\begin{equation}
		\label{eqn:prop_B=0_eta}
		\frac{1 + \sqrt{ \mathscr{A}^2 g^2 + 1 } \sqrt{1 - \mathscr{B}^2 g^2}}{2 g^2} = \eta \,,
	\end{equation}
	for some $\eta \in i \R$. 
	
	For $z$ on the hyperbola, $z$ is on the rectangle when
	\begin{equation}
		\label{eqn:prop_B=0_rectangle}
		\Abs{\frac{1 + \sqrt{ \mathscr{A}^2 g^2 + 1 } \sqrt{1 - \mathscr{B}^2 g^2}}{2 g^2}} \leq \frac{\mathscr{A}\mathscr{B}}{2} \,.
	\end{equation}
	Simplifying (\ref{eqn:prop_B=0_eta}) further, $z$ is on the hyperbola if and only if
	\begin{equation}
		\sqrt{ \mathscr{A}^2 g^2 + 1 } \sqrt{1 - \mathscr{B}^2 g^2} = 2 g^2 \eta - 1 \, ,
	\end{equation}
	Squaring both sides and rearranging, $g$ must satisfy: 
	\begin{equation}
		0 = g^2 ( (4 \eta^2 -  \mathscr{A}^2 \mathscr{B}^2   ) g^2 - (4 \eta + (\mathscr{A}^2 - \mathscr{B}^2))   ) \,.
	\end{equation}
	As $g \neq 0$, then: 
	\begin{equation}
		\label{eqn:prop_B=0_g^2}
		g^2 = \frac{4 \eta + (\mathscr{A}^2 - \mathscr{B}^2)}{4 \eta^2 + \mathscr{A}^2 \mathscr{B}^2} \,.
	\end{equation}
	All steps were reversible except for the step where both sides were squared. We proceed to determine which values of $\eta$ make (\ref{eqn:prop_B=0_hyperbola}) true, and show for these $\eta$, (\ref{eqn:prop_B=0_rectangle}) is not satisfied.
	
	Computation shows that the relevant quantities in (\ref{eqn:prop_B=0_eta}) are: 
	\begin{equation}
		\label{eqn:prop_B=0_eqn1}
		\begin{aligned}
			\mathscr{A}^2 g^2 + 1 & = \frac{(\mathscr{A}^2 + 2 \eta)^2}{4 \eta^2 + \mathscr{A}^2 \mathscr{B}^2} \\
			1 - \mathscr{B}^2 g^2 &= \frac{ (\mathscr{B}^2 - 2 \eta)^2 }{4 \eta^2 + \mathscr{A}^2 \mathscr{B}^2} \,.
		\end{aligned}
	\end{equation}
	To take the square root of these quantities, consider the following cases for $\eta \in i \R$:
	
	\begin{enumerate}
		\item $\Abs{\eta} < \mathscr{A}\mathscr{B} / 2$, i.e. $\Abs{\Im(\eta) } \leq \mathscr{A} \mathscr{B} / 2$.
		\item $\Abs{\eta} > \mathscr{A}\mathscr{B} / 2$, i.e. $\Im (\eta) > \mathscr{A}\mathscr{B} / 2$ or $\Im (\eta) < - \mathscr{A}\mathscr{B} / 2$.
	\end{enumerate}
	
	For the first case, note that the denominator in (\ref{eqn:prop_B=0_eqn1}) is positive and $\mathscr{A}^2 + 2 \eta, \mathscr{B}^2 - 2 \eta$ are in the right half-plane, so
	\begin{equation}
		\begin{aligned}
			\sqrt{\mathscr{A}^2 g^2 + 1} & = \frac{\mathscr{A}^2 + 2 \eta}{\sqrt{4 \eta^2 + \mathscr{A}^2 \mathscr{B}^2}} \\
			\sqrt{1 - \mathscr{B}^2 g^2} &= \frac{ \mathscr{B}^2 - 2 \eta }{\sqrt{4 \eta^2 + \mathscr{A}^2 \mathscr{B}^2}} \,.
		\end{aligned}
	\end{equation}
	Hence, 
	\begin{equation}
		\label{eqn:prop_B=0_eqn2}
		\frac{1 + \sqrt{ \mathscr{A}^2 g^2 + 1 } \sqrt{1 - \mathscr{B}^2 g^2}}{2 g^2} = \frac{\mathscr{A}^2 \mathscr{B}^2 + (\mathscr{B}^2 - \mathscr{A}^2) \eta }{4 \eta + (\mathscr{A}^2 - \mathscr{B}^2)} \,.
	\end{equation}
	If $\mathscr{A}^2 - \mathscr{B}^2 = 0$, then
	\begin{equation}
		\frac{\mathscr{A}^2 \mathscr{B}^2 + (\mathscr{B}^2 - \mathscr{A}^2) \eta }{4 \eta + (\mathscr{A}^2 - \mathscr{B}^2)} = \frac{\mathscr{A}^2 \mathscr{B}^2}{4 \eta} \,.
	\end{equation}
	For any $\eta \in i \R$, this quantity is purely imaginary, and so $z$ lies on the hyperbola. But, for $\Abs{\eta} < \mathscr{A}\mathscr{B} / 2$, $\Abs{\mathscr{A}^2 \mathscr{B}^2 / (4 \eta)}  > \mathscr{A}\mathscr{B} / 2$, so $z$ is not on the rectangle.
	
	If $\mathscr{A}^2 - \mathscr{B}^2 \neq 0$, the right-hand side in (\ref{eqn:prop_B=0_eqn2}) is purely imaginary when: 
	\begin{equation}
		\frac{- 4 \, \Im (\eta)}{\mathscr{A}^2 \mathscr{B}^2} = \frac{\mathscr{A}^2 - \mathscr{B}^2}{(\mathscr{B}^2 - \mathscr{A}^2)  \, \Im (\eta)} \, ,
	\end{equation}
	i.e. when $\Im (\eta) = \pm  \mathscr{A}\mathscr{B} / 2$. This is impossible for $\Abs{\eta} < \mathscr{A}\mathscr{B} / 2$, so $z$ is not on the hyperbola.
	
	Thus, when $\Abs{\eta} < \mathscr{A}\mathscr{B} / 2$, it is impossible for $z$ to be on both the hyperbola and rectangle. 
	
	For  $\Abs{\eta} > \mathscr{A}\mathscr{B} / 2$, first consider when $\Im (\eta) > \mathscr{A}\mathscr{B} / 2$. Then, the denominator in (\ref{eqn:prop_B=0_eqn1}) is negative, $\mathscr{A}^2 + 2 \eta$ is in the first quadrant, and $\mathscr{B}^2 - 2 \eta$ is in the fourth quadrant, so 
	\begin{equation}
		\begin{aligned}
			\sqrt{\mathscr{A}^2 g^2 + 1} & = \frac{  - i (\mathscr{A}^2 + 2 \eta)}{\sqrt{- 4 \eta^2 - \mathscr{A}^2 \mathscr{B}^2}} \\
			\sqrt{1 - \mathscr{B}^2 g^2} &= \frac{ i (\mathscr{B}^2 - 2 \eta) }{\sqrt{- 4 \eta^2 - \mathscr{A}^2 \mathscr{B}^2}} \,.
		\end{aligned}
	\end{equation}
	When $\Im(\eta) < - \mathscr{A}\mathscr{B} / 2$, the denominator in (\ref{eqn:prop_B=0_eqn1}) is negative, $\mathscr{A}^2 + 2 \eta$ is in the fourth quadrant, and $\mathscr{B}^2 - 2 \eta$ is in the first quadrant, so
	\begin{equation}
		\begin{aligned}
			\sqrt{\mathscr{A}^2 g^2 + 1} & = \frac{  i (\mathscr{A}^2 + 2 \eta)}{\sqrt{- 4 \eta^2 - \mathscr{A}^2 \mathscr{B}^2}} \\
			\sqrt{1 - \mathscr{B}^2 g^2} &= \frac{ - i (\mathscr{B}^2 - 2 \eta) }{\sqrt{- 4 \eta^2 - \mathscr{A}^2 \mathscr{B}^2}} \,.
		\end{aligned}
	\end{equation}
	In either case, when $\Abs{\eta} > \mathscr{A}\mathscr{B} / 2$, 
	\begin{equation}
		\frac{1 + \sqrt{ \mathscr{A}^2 g^2 + 1 } \sqrt{1 - \mathscr{B}^2 g^2}}{2 g^2} = \frac{ 4 \eta^2 + \eta (\mathscr{A}^2 - \mathscr{B}^2) }{ 4 \eta + (\mathscr{A}^2 - \mathscr{B}^2)  } = \eta \,.
	\end{equation}
	From (\ref{eqn:prop_B=0_g^2}), the denominator $4 \eta + (\mathscr{A}^2 - \mathscr{B}^2)$ is zero exactly when $g = 0$, which is impossible, so this expression makes sense. Thus, for all $\eta$ where $\Abs{\eta} > \mathscr{A}\mathscr{B} / 2$, $z$ is on the hyperbola. But, since $\Abs{\eta} > \mathscr{A}\mathscr{B}/ 2$, then (\ref{eqn:prop_B=0_rectangle}) does not hold and $z$ is not on the rectangle. 
\end{proof}

Finally, consider when $Q_k$ converges to $Q \in \C$, $Q \neq 0$, such that $Q$ is not a continuity point of $\mathcal{B}_{\mathbf{X}}$. Assuming that $Q = g$ and $g^I = i g$, this happens when either $g \in I_p$ or $i g \in I_q$, i.e. $g \in i \R \cup \R$. Then, in the addition law (\ref{eqn:bX_free}), one of $\mathcal{B}_{\mathbf{p}}$ or $\mathcal{B}_{i \mathbf{q}}$ is continuous. We can handle the other term with the following general Lemma: 

\begin{lemma}
	\label{lem:B_0_discontinuous_limit}
	Let $p \in (M, \tau)$ be Hermitian and consider a sequence $\{Q_k\} \subset \mathbb{H}$ such that $Q_k \to Q \in \C$, where $Q \neq 0$ either satisfies $Q \in \C \setminus \R$ or $B_p$ is continuous at $Q$.
	Then, 
	\begin{equation}
		\lim\limits_{k \to \infty} \mathcal{B}_{\mathbf{p}}(Q_k) = \begin{pmatrix}
			z & 0 \\
			0 & \bar{z} 
		\end{pmatrix} \, ,
	\end{equation}
	where $\lim_{k \to \infty} B_p(\zeta_k) = z$ for some sequence $\{\zeta_k\} \subset \C$ where $\zeta_k \to Q$.
	
	It follows that if $G_p$ is defined and continuous at $z$, then $Q = G_p(z)$.
\end{lemma}
\begin{proof}
	We will show the sequence
	\begin{equation}
		\zeta_k = 
		\begin{dcases}
			g_k & \Im (Q) \geq 0 \\
			\bar{g_k} & \Im (Q) < 0
		\end{dcases}
	\end{equation}
	satisfies the conclusion of the Lemma.
	
	When $Q \in \R$, then from hypothesis, $B_p$ is continuous at $Q$. Then, $g_k, \bar{g_k}$ converge to $Q$, so the eigenvalues of $\mathcal{B}_{\mathbf{p}}(Q_k)$, $B_p(g_k), B_p(\bar{g_k})$ converge to $B_p(Q)$. Since $Q \in \R$, then $B_p(Q) \in \R$. From Lemma \ref{lem:Q_real_convergence}, $\mathcal{B}_{\mathbf{p}}(Q_k)$ converges to $B_p(Q)$. Thus, our choice of $\zeta_k$ satisfies the conclusion of the Lemma.
	
	For $Q \not \in \R$, Consider a diagonalization of $Q_k$:
	\begin{equation}
		Q_k 
		= S_k^{-1} 
		g_k
		S_k \,.
	\end{equation}
	Then, 
	\begin{equation}
		\begin{aligned}
			\mathcal{B}_{\mathbf{p}}(Q_k) 
			& = \mathcal{B}_{\mathbf{p}} (S_k^{-1} 
			g_k
			S_k) \\
			& = 
			S_k^{-1}
			\mathcal{B}_{\mathbf{p}}
			(g_k)
			S_k \\
			& = S_k^{-1}
			B_p(g_k)
			S_k \,.
		\end{aligned}
	\end{equation}
	Recall that $g_k$ is the eigenvalue of $Q_k$ with $\Im(g_k) \geq 0$. From the definition of $\zeta_k$, it suffices to show that we can choose suitable $S_k$ for each of the following cases for $x_3 = \Im (Q)$:
	
	\begin{enumerate}
		\item If $x_3 < 0$, $S_k$ converges to an invertible matrix $S$ that switches the diagonal entries of $B_p(g_k)$, i.e. $S B_p(g_k) S^{-1} = B_p(\bar{g_k})$.
		\item If $x_3 > 0$, $S_k$ converges to an invertible matrix that fixes the diagonal entries of $B_p(g_k)$, i.e. $S B_p(g_k) S^{-1} = B_p(g_k)$.
	\end{enumerate}
	
	For the first case, where $x_3 < 0$, computation shows that we may choose
	\begin{equation}
		S_k = 
		\begin{pmatrix}
			i B_k & g_k - A_k \\
			\bar{g} - \bar{A_k} & i B_k
		\end{pmatrix}
	\end{equation}
	to diagonalize the matrix. Since $Q_k \to Q \in \C$, the diagonal terms of $S_k$ converge to $0$. When $(x_3)_k < 0$, the off-diagonal terms are: 
	\begin{equation}
		\begin{aligned}
			g_k - A_k 
			& = i \left( \sqrt{ (x_1)_k^2 + (x_2)_k^2 + (x_3)_k^2 } - (x_3)_k \right) \\
			& = i \left( \sqrt{ (x_1)_k^2 + (x_2)_k^2 + (x_3)_k^2 } + \Abs{(x_3)_k}  \right) 
		\end{aligned}
	\end{equation}
	As $B_k \to 0$, then $(x_1)_k, (x_2 )_k \to 0$ and $(x_3)_k \to x_3$, so this term converges to $2 i \Abs{x_3}$. Hence,
	\begin{equation}
		S_k \to \begin{pmatrix}
			0 & 2 i \Abs{x_3} \\
			- 2 i \Abs{x_3} & 0
		\end{pmatrix} \,.
	\end{equation}
	This matrix is invertible and switches the diagonal entries of the matrix. 
	
	If $x_3 > 0$, we alter the previous $S_k$ by dividing by $B_k$, choosing 
	\begin{equation}
		S_k = 
		\begin{dcases}
			\begin{pmatrix}
				i  & \frac{g_k - A_k}{B_k} \\
				\frac{\bar{g} - \bar{A_k}}{B_k} & i
			\end{pmatrix} & B_k \neq 0 \\		
			\begin{pmatrix}
				i  & 0 \\
				0 & i
			\end{pmatrix} & B_k = 0 \,.
		\end{dcases}
	\end{equation}
	It suffices to check that $S_k$ tends to its value at $B_k = 0$. For this, as the off-diagonal terms are conjugates, we examine just one of them: 
	\begin{equation}
		\frac{g_k - A_k}{B_k} = \frac{i \left( \sqrt{ (x_1)_k^2 + (x_2)_k^2 + (x_3)_k^2 } - (x_3)_k \right) }{(x_1)_k + i (x_2)_k} \,.
	\end{equation}
	Taking the absolute value of the right-hand side, 
	\begin{equation}
		\Abs{\frac{i \left( \sqrt{ (x_1)_k^2 + (x_2)_k^2 + (x_3)_k^2 } - (x_3)_k \right) }{(x_1)_k + i (x_2)_k}}
		= \frac{\sqrt{ (x_1)_k^2 + (x_2)_k^2 + (x_3)_k^2 } - (x_3)_k}{\sqrt{(x_1)_k^2 + (x_2)_k^2}} \,.
	\end{equation}
	When $(x_3)_k > 0$, applying the Mean Value Theorem to the function $f(t) = \sqrt{ (x_3)_k^2 + t^2 }$ for $t \in [0, \sqrt{(x_1)_k^2 + (x_2)_k^2}]$ yields
	\begin{equation}
		\frac{\sqrt{ (x_1)_k^2 + (x_2)_k^2 + (x_3)_k^2 } - (x_3)_k}{(x_1)_k^2 + (x_2)_k^2} 
		= \frac{ t'  }{\sqrt{ (x_3)_k^2 + (t')^2 }} \, ,
	\end{equation}
	for some $t' \in (0, \sqrt{(x_1)_k^2 + (x_2)_k^2})$. Since $(x_3)_k \to x_3 > 0$ and $B_k = (x_1)_k + (x_2)_k \to 0$, then the following inequalities show that the off-diagonal terms converge to $0$: 
	\begin{equation}
		\frac{ t'  }{\sqrt{ (x_3)_k^2 + (t')^2 }} \leq \frac{\sqrt{(x_1)_k^2 + (x_2)_k^2}}{\Abs{(x_3)_k}} \to 0 \,.
	\end{equation}
	For the final point, if $G_p$ is continuous at $z$, then 
	\begin{equation}
		Q = \lim\limits_{k \to \infty} \zeta_k = \lim\limits_{k \to \infty} G_p(B_p(\zeta_k)) = G_p(z) \,.
	\end{equation}
\end{proof}

Consider the previous result, but with $p$ replaced with $X = p + i q$. From the addition law (\ref{eqn:bX_free}) and the previous result, we could prove a result with two sequences $\zeta_k$ and $\zeta_k'$ converging to $Q$ and $Q i$ for $\mathcal{B}_{\mathbf{p}}$ and $\mathcal{B}_{i \mathbf{q}}$. For general $p, q$ we may not be able to use only one sequence by replacing $\zeta_k'$ with $i \zeta_k$, as the limits of $\mathcal{B}_{i \mathbf{q}}$ along these sequences may be different. But, in our situation, we can: 

\begin{corollary}
	\label{cor:bX}
	Consider a sequence $\{Q_k\} \subset \mathbb{H}$ such that $Q_k \to Q \in \C$, where $Q  \neq 0$. Then, 
	\begin{equation}
		\lim\limits_{k \to \infty} \mathcal{B}_{\mathbf{X}}(Q_k) = 
		\begin{pmatrix}
			z & 0 \\
			0 & \bar{z}
		\end{pmatrix} \, ,
	\end{equation}
	where $\lim_{k \to \infty} B_X(\zeta_k) = z$ for some sequence $\{\zeta_k\} \subset \C$ where $\zeta_k \to Q$.
	
	It follows that if $G_X$ is defined and continuous at $z$, then $Q = G_X(z)$.
\end{corollary}
\begin{proof}
	When $\mathcal{B}_{\mathbf{X}}$ is continuous at $Q$, we can just take $\zeta_k$ to be any sequence converging to $Q$. 
	
	When $\mathcal{B}_{\mathbf{X}}$ is discontinuous at $Q$, either $g \in I_p$ or $g^I \in I_q$. In either case, in the addition law (\ref{eqn:bX_free}), only one of $\mathcal{B}_{\mathbf{p}}$ or $\mathcal{B}_{i \mathbf{q}}$ is discontinuous at $Q$. 
	
	Lemma \ref{lem:B_0_discontinuous_limit} produces the appropriate $\zeta_k$ that works for the one of $\mathcal{B}_{\mathbf{p}}$ or $\mathcal{B}_{i \mathbf{q}}$ that is discontinuous at $Q$, and that $\zeta_k$ will also work for the other function that is continuous at $Q$. 
	
	For the final point, if $G_X$ is continuous at $z$, then 
	\begin{equation}
		Q = \lim\limits_{k \to \infty} \zeta_k = \lim\limits_{k \to \infty} G_X(B_X(\zeta_k)) = G_X(z) \,.
	\end{equation}
\end{proof}

Finally, we are ready to prove the following Proposition: 

\begin{proposition}
	\label{prop:B_0_discontinuous_limit}
	If $Q_k$ converges to $Q \in \C$ where $Q \neq 0$ and $\mathcal{B}_{\mathbf{X}}$ is discontinuous at $Q$, then $\mathcal{B}_{\mathbf{X}}(Q_k) \to z \in \C$ where $z = x + i y$ has
	\begin{equation}
		\left( x - \frac{\alpha + \alpha'}{2}  \right)^2 - \left(  y - \frac{\beta + \beta'}{2} \right)^2 \neq \frac{(\alpha' - \alpha)^2 - (\beta' - \beta)^2}{4}  \,.
	\end{equation}
\end{proposition}
\begin{proof}
	From Theorem \ref{thm:BX_cty}, $g \in I_p$ or $g^I \in I_q$. We will prove the first case, the second case is similar.
	
	From Corollary \ref{cor:bX}, 
	\begin{equation}
		\mathcal{B}_{\mathbf{X}}(Q_k) \to 
		\begin{pmatrix}
			z & 0 \\
			0 & \bar{z}
		\end{pmatrix} \, ,
	\end{equation}
	where $\lim_{k \to \infty} B_X(\zeta_k) =  z$ for some $\zeta_k \to g$. From the addition law (\ref{eqn:bX_free}),
	\begin{equation}
		B_X(\zeta_k) = B_p(\zeta_k) + i B_q(\zeta_k i) - \frac{1}{\zeta_k} \,.
	\end{equation}
	If $g \in I_p$, then $B_p$ is discontinuous at $g$. From Proposition \ref{prop:complex_b_formula}, 
	\begin{equation}
		\lim\limits_{k \to \infty} B_p(\zeta_k)  =
		\frac{\alpha + \alpha'}{2} + \frac{1 \pm i \sqrt{- D_p(g)}}{2 g} \,.
	\end{equation}
	Since $\zeta_k i \to i g \in \R$, then $B_q$ is continuous at $i g$. From Proposition \ref{prop:complex_b_formula},
	\begin{equation}
		\lim\limits_{k \to \infty} i B_q(\zeta_k i) 
		= i B_q(g i) = i \frac{\beta + \beta'}{2} + \frac{1 + \sqrt{D_q(i g) }}{2 g} \,.
	\end{equation}
	Hence, 
	\begin{equation}
		z = \frac{\alpha + \alpha'}{2} + i \frac{\beta + \beta'}{2} \pm i \frac{ \sqrt{ - D_p(g) }  }{2 g}  + \frac{\sqrt{D_q(i g)}}{2 g} \,.
	\end{equation}
	From Lemma \ref{lem:hyperbola_rectangle}, $z$ is on the hyperbola if and only if: 
	\begin{equation}
		\Re \left( \left( z - \frac{\alpha + \alpha'}{2} - i \frac{\beta + \beta'}{2}  \right)^2  - \frac{(\alpha' - \alpha)^2 - (\beta' - \beta)^2}{4} \right) =   0 \,.
	\end{equation}
	Substituting the expression for $z$ into this and simplifying, $z$ is on the hyperbola if and only if: 
	\begin{equation}
		\Re \left( \frac{1 \pm i \sqrt{ - D_p(g)  } \sqrt{ D_q(i g) }  }{2 g^2} \right) = 0 \,.
	\end{equation}
	Since $g \in I_p$, then $\sqrt{ - D_p(g)  }, \sqrt{ D_q(i g) }, g^2 \in \R$, so the real part of the previous equation is $1 / (2 g^2)$, which is non-zero. Hence, $\mathcal{B}_{\mathbf{X}}(Q_k)$ does not converge to a point on the hyperbola.
\end{proof}

\subsection{Proof when $p$ and $q$ have $2$ atoms}
In this subsection, we prove Theorem \ref{thm:B_not_0} and also make some observations about the 4 corners of the rectangle $R$.

First, a summary of the results in the previous subsections: 

\begin{enumerate}
	\item If $Q_k \to 0$, then $\Abs{\mathcal{B}_{\mathbf{X}}(Q_k)} \to \infty$. (Theorem \ref{thm:BX_cty})
	\item If $\Abs{Q_k} \to \infty$ and $\mathcal{B}_{\mathbf{X}}(Q_k)$ converges, then $\mathcal{B}_{\mathbf{X}}(Q_k)$ converges to one of $\{\alpha + i \beta, \alpha + i \beta' , \alpha' + i \beta , \alpha' + i \beta'\}$. (Proposition \ref{prop:unbounded_Q})
	\item If $Q_k \to Q \neq 0$:
	\begin{enumerate}
		\item If $l_k \to 0$:
		\begin{enumerate}
			\item If $\mathcal{B}_{\mathbf{X}}$ is discontinuous at $Q$, then $\mathcal{B}_{\mathbf{X}}(Q_k)$ converges to $z$ on the intersection of the hyperbola with the open rectangle. (Proposition \ref{prop:l_0_discontinuous_limit}) 
			\item If $\mathcal{B}_{\mathbf{X}}$ is continuous at $Q$:
			\begin{enumerate}
				\item If $g \in \R$ or $g^I \in \R$, then $l(Q_k) \not \to 0$. (Proposition \ref{prop:l_not_0})
				\item If $g \not \in \R$ and $g^I \not \in \R$, then $\mathcal{B}_{\mathbf{X}}(Q_k) \to \mathcal{B}_{\mathbf{X}}(Q) = z$, which is on the intersection of the hyperbola with the open rectangle. (Proposition \ref{prop:l_0_continuous_limit} )
			\end{enumerate}
		\end{enumerate}
		\item If $B_k \to 0$:
		\begin{enumerate}
			\item If $\mathcal{B}_{\mathbf{X}}$ is continuous at $Q$, then $\mathcal{B}_{\mathbf{X}}(Q_k) \to \mathcal{B}_{\mathbf{X}}(Q) = z$, and whenever $z$ is on the hyperbola, $z$ is not in the rectangle. (Proposition \ref{prop:B_0_continuous_limit}) 
			\item If $\mathcal{B}_{\mathbf{X}}$ is discontinuous at $Q$, then $\mathcal{B}_{\mathbf{X}}(Q_k)$ does not converge to a $z$ on the hyperbola. (Proposition \ref{prop:B_0_discontinuous_limit})
		\end{enumerate} 	
	\end{enumerate}
\end{enumerate}

As a corollary to these facts, we make an observation about the $4$ corners of the intersection of the hyperbola and rectangle: 

\begin{corollary}
	\label{cor:quaternionic_atoms}
	Let $z \in \{ \alpha + i \beta, \alpha' + i \beta, \alpha + i \beta', \alpha' + i \beta' \}$ and suppose that $\mathcal{G}_{\mathbf{X}}(z_\epsilon)$ is in the domain of $\mathcal{B}_{\mathbf{X}}$ for sufficiently small $\epsilon > 0$. Then, 
	\begin{equation}
		\lim\limits_{\epsilon \to 0^+} \Abs{\mathcal{G}_{\mathbf{X}}(z_\epsilon)} = \infty \,.
	\end{equation}
\end{corollary}
\begin{proof}
	The proof is based on the results in the previous subsections and noticing that only in the situation where $z \in \{ \alpha + i \beta, \alpha' + i \beta, \alpha + i \beta', \alpha' + i \beta' \}$ is possible is when $\Abs{Q_k} \to \infty$ and $\mathcal{B}_{\mathbf{X}}(Q_k)$ converges.
	
	Recall that $Q_\epsilon = \mathcal{G}_{\mathbf{X}}(z_\epsilon)$. It suffices to show that for every sequence $\epsilon_k \to 0^+$, $\Abs{Q_{\epsilon_k}} \to \infty$. 
	
	Suppose for the sake of contradiction that $Q_{\epsilon_k}$ is bounded for some $\epsilon_k \to 0^+$. 
	
	If $Q_{\epsilon_k} \to 0$, then
	\begin{equation}
		z = \lim\limits_{k \to \infty} z_{\epsilon_k} = \lim\limits_{k \to \infty} \mathcal{B}_{\mathbf{X}}(Q_{\epsilon_k}) \,,
	\end{equation}
	but from Theorem \ref{thm:BX_cty}, the final limit diverges. 
	
	If $Q_{\epsilon_k} \not \to 0$, then we may pass to a subsequence and assume that $Q_{\epsilon_k} \to Q \neq 0$. Then, from the results in the previous subsections, $z \not \in \{ \alpha + i \beta, \alpha' + i \beta, \alpha + i \beta', \alpha' + i \beta' \}$, a contradiction.
	
	Thus, it must be the case that 
	\begin{equation}
		\lim\limits_{\epsilon \to 0^+} \Abs{\mathcal{G}_{\mathbf{X}}(z_\epsilon)} = \infty \,.
	\end{equation}
\end{proof}

We can verify the domain condition in Corollary \ref{cor:quaternionic_atoms} to get the following concrete result: 

\begin{proposition}
	\label{prop:quaternionic_atoms}
	Let $z \in \{ \alpha + i \beta, \alpha' + i \beta, \alpha + i \beta', \alpha' + i \beta' \}$. Then, 
	\begin{equation}
		\lim\limits_{\epsilon \to 0^+} \Abs{\mathcal{G}_{\mathbf{X}}(z_\epsilon)} = \infty \,.
	\end{equation}
\end{proposition}
\begin{proof}
	Let $X_z = z - X$. Computation using Definition \ref{def:quaternionic_green} shows that: 
	\begin{equation}
		\mathcal{G}_\mathbf{X}(z_\epsilon) 
		= 
		\begin{pmatrix}
			\tau[ ((X_z)^* X_z + \epsilon^2)^{-1} (X_z)^*] & - i \epsilon \tau[ ( (X_z)^* X_z + \epsilon^2  )^{-1} ] \\
			i \epsilon \tau[ ( (X_z)^* X_z + \epsilon^2  )^{-1} ] & \tau[X_z ((X_z)^* X_z + \epsilon^2)^{-1} ]
		\end{pmatrix} .
	\end{equation}
	In light of Corollary \ref{cor:quaternionic_atoms} and the domain of $\mathcal{B}_{\mathbf{X}}$ described in Theorem \ref{thm:BX_cty}, it suffices to show that for $z \in \{ \alpha + i \beta, \alpha' + i \beta, \alpha + i \beta', \alpha' + i \beta' \}$ and $\epsilon > 0$, $\tau[ ((X_z)^* X_z + \epsilon^2)^{-1} (X_z)^*] \not \in \R \cup i \R$. 
	
	We will just show that $\tau[ ((X_z)^* X_z + \epsilon^2)^{-1} (X_z)^*] \not \in i \R$, the other case is similar. For this, let $X_z = p + i q$, where $p, q$ are Hermitian, freely independent, and have $2$ atoms. For $z \in \{ \alpha + i \beta, \alpha' + i \beta, \alpha + i \beta', \alpha' + i \beta' \}$, $p$ and $q$ both have $0$ as an atom. Hence, $p$ and $q$ are either positive or negative operators. At this point, we also drop the subscript $z$ so that $X = p + i q$. 
	
	Assume for the sake of contradiction that $\tau[ (X^* X + \epsilon^2)^{-1} X^*] \in i \R$. This is equivalent to
	\begin{equation}
		0 = \Re \, \tau[ (X^* X + \epsilon^2)^{-1} X^*]  = \tau[ (X^* X + \epsilon^2)^{-1} p ] \,.
	\end{equation}
	Since $p \geq 0$ or $p \leq 0$, without loss of generality assume that $p \geq 0$. Then, 
	\begin{equation}
		\begin{aligned}
			\tau[ ((X^* X + \epsilon^2)^{-1/2} p^{1/2})^* (X^* X + \epsilon^2)^{-1/2} p^{1/2} ] = \tau[ (X^* X + \epsilon^2)^{-1} p ] = 0 \,.
		\end{aligned}
	\end{equation}
	Since $\tau$ is faithful, then $(X^* X + \epsilon^2)^{-1/2} p^{1/2} = 0$. Hence, 
	\begin{equation}
		p = (X^* X + \epsilon^2)^{1/2} [(X^* X + \epsilon^2)^{-1/2} p^{1/2}] p^{1/2} = 0 \,.
	\end{equation}
	But, this is impossible, as $p$ has $2$ atoms. Hence, $\tau[ (X^* X + \epsilon^2)^{-1} X^*] \not \in i \R$.
\end{proof}

We prove one final Proposition before the proof of Theorem \ref{thm:B_not_0}:

\begin{proposition}
	\label{prop:limit_B=0}
	Let $z \in \C$ and suppose there exists a sequence $\epsilon_k \to 0^+$, such that $Q_{\epsilon_k} \to Q \neq 0$, where $Q \in \C$. Then, 
	\begin{equation}
		\lim\limits_{\epsilon \to 0^+} \mathcal{G}_{\mathbf{X}}(z_\epsilon) = \lim\limits_{\epsilon \to 0^+} Q_\epsilon = Q   \,.
	\end{equation} 	
	In particular, when $G_X$ is continuous at $z$, $Q = G_X(z)$.
\end{proposition}
\begin{proof}
	If there is a sequence $Q_{\epsilon_k} \to Q \neq 0$ where $Q \in \C$, then $B_{\epsilon_k} \to 0$. We proceed to upgrade the convergence $B_{\epsilon_k} \to 0$ along a specific sequence to $B_\epsilon \to 0$ as $\epsilon \to 0^+$: If $B_\epsilon \not \to 0$, then from (\ref{eqn:lib}), on some sequence $\epsilon_{k}' \to 0^+$, $l_{\epsilon_k'} \to 0$. From $B_{\epsilon_k} \to 0$ and Propositions \ref{prop:B_0_continuous_limit} and \ref{prop:B_0_discontinuous_limit}, $z$ does not lie on the intersection of the hyperbola and rectangle. But, from $l_{\epsilon_k'} \to 0$ and Propositions \ref{prop:l_0_discontinuous_limit} and \ref{prop:l_0_continuous_limit}, $z$ does lie on the intersection of the hyperbola and rectangle. This is a contradiction, so we conclude that $B_\epsilon \to 0$. 
	
	Next, consider an arbitrary sequence $\epsilon_k' \to 0$ such that $Q_{\epsilon_k'} \to Q' \neq 0$, where $Q' \in \C$. To prove the limit in the statement of the Proposition, it suffices to show that $Q' = Q$. 
	
	Corollary \ref{cor:bX} gives two sequences $\zeta_{\epsilon_k}, \zeta_{\epsilon_k'} \subset \C$ such that $\zeta_{\epsilon_k} \to Q$, $\zeta_{\epsilon_k'} \to Q'$, and:
	\begin{equation}
		\lim\limits_{k \to \infty} B_X(\zeta_{\epsilon_k}) = \lim\limits_{k \to \infty} B_X(Q_{\epsilon_k}) = z = \lim\limits_{k \to \infty} B_X(Q_{\epsilon_k'})  =  \lim\limits_{k \to \infty} B_X(\zeta_{\epsilon_k'}) \,.
	\end{equation}
	We consider the cases when $B_X$ is continuous or discontinuous at $Q$ and $Q'$: 
	
	If $B_X$ is continuous at both $Q$ and $Q'$, then since $B_X$ is invertible on its domain, it is injective. Hence, 
	\begin{equation}
		B_X(Q) = \lim\limits_{k \to \infty} B_X(\zeta_{\epsilon_k}) = z = \lim\limits_{k \to \infty} B_X(\zeta_{\epsilon_k'}) = B_X(Q') 
	\end{equation}
	implies that $Q = Q'$. 
	
	If $B_X$ is continuous at exactly one of $Q$ and $Q'$, assume $B_X$ is continuous at $Q$ but not $Q'$. From Proposition \ref{prop:complex_B_X}, $B_X$ is analytic in a neighborhood of $Q$, so from the Open Mapping Theorem, for $U \subset \C$ open where $Q \in U$, $Q' \not \in U$, $B_X(U)$ is an open set containing $B_X(Q) = z$. Since $B_X(\zeta_{\epsilon_k'})$ converges to $z$ also, then for sufficiently large $k$, $B_X(\zeta_{\epsilon_k'}) \in B_X(U)$, but $\zeta_{\epsilon_k'} \not \in U$. This contradicts the injectivity of $B_X$. 
	
	Finally, consider if $B_X$ is discontinuous at both $Q$ and $Q'$, i.e. $Q \in I_p$ or $Q i \in I_q$, and $Q' \in I_p$ or $Q' i \in I_q$.
	
	From the proof of Proposition \ref{prop:B_0_discontinuous_limit}, the possible limits for $B_X(\zeta_k)$ are: 
	\begin{equation}
		\begin{aligned}
			z & = \lim\limits_{k \to \infty} \mathcal{B}_{\mathbf{X}}(Q_{\epsilon_k}) \\
			& = 
			\begin{dcases}
				\frac{\alpha + \alpha'}{2} + i \frac{\beta + \beta'}{2} \pm i \frac{ \sqrt{ - D_p(g) }  }{2 g}  + \frac{\sqrt{D_q(i g)}}{2 g} & g \in I_p \\
				\frac{\alpha + \alpha'}{2} + i \frac{\beta + \beta'}{2} + \frac{ \sqrt{ D_p(g) }  }{2 g}  \pm i \frac{\sqrt{ - D_q(i g)}}{2 g} & i g \in I_q \,.
			\end{dcases}
		\end{aligned}
	\end{equation}
	Computation shows that:
	\begin{equation}
		\begin{aligned}
			\tilde{z}
			& = \left(  z - \frac{\alpha + \alpha'}{2} - i \frac{\beta + \beta'}{2}  \right)^2  - \frac{(\alpha' - \alpha)^2 - (\beta' - \beta)^2}{4} \\
			& = 
			\begin{dcases}
				\frac{1 \pm i \sqrt{ - D_p(g)  } \sqrt{ D_q(i g) }  }{2 g^2} & g \in I_p \\
				\frac{1 \pm i \sqrt{ D_p(g)  } \sqrt{- D_q(i g) }  }{2 g^2} & i g \in I_q \,.
			\end{dcases}
		\end{aligned}
	\end{equation}
	There are analogous equations for $Q'$, where $g$ is replaced with $g'$.  
	
	In these equations, we can determine if $g \in I_p$ or $i g \in I_q$ by observing that $\Re(\tilde{z}) < 0$ for $g \in I_p$ and $\Re(\tilde{z}) > 0$ for $i g \in I_q$. Then, by observing $\Re (\tilde{z}) = 1 / (2 g^2)$, we can recover $g$ up to a sign. Finally, by examining $\Re(z)$, we can determine what $g$ is. Hence, $Q = Q'$, as desired. 
	
	The last point follows from Corollary \ref{cor:bX}.
\end{proof}

Finally, we can prove Theorem \ref{thm:B_not_0}:

\begin{proof}[Proof of Theorem \ref{thm:B_not_0}]
	First, consider $z$ on the support of the Brown measure of $X$. From Theorem \ref{thm:brown_measure_p+iq}, $z = x + i y$ lies on the intersection of the hyperbola 
	\begin{equation}
		H = \left\lbrace z = x + i y :  \left( x - \frac{\alpha + \alpha'}{2}  \right)^2 - \left(  y - \frac{\beta + \beta'}{2} \right)^2 = \frac{(\alpha' - \alpha)^2 - (\beta' - \beta)^2}{4}    \right\rbrace
	\end{equation}
	with the rectangle
	\begin{equation}
		R = \left\lbrace z = x + i y: x \in [ \alpha \wedge \alpha', \alpha \vee \alpha'   ]  , y \in [\beta \wedge \beta', \beta \vee \beta'  ] \right\rbrace \,.
	\end{equation}
	Recall that
	\begin{equation}
		Q_\epsilon = 
		\mathcal{G}_{\mathbf{X}}
		\left( z_\epsilon	\right) 
		= \mathcal{G}_{\mathbf{X}}
		\left( 
		\begin{pmatrix}
			z & i \epsilon \\
			i \epsilon & \bar{z}
		\end{pmatrix}
		\right) 
		= \begin{pmatrix}
			A_\epsilon & i \bar{B_\epsilon} \\
			i B_\epsilon & \bar{A_\epsilon}
		\end{pmatrix} \,.
	\end{equation}
	If $Q_\epsilon$ converges to some $Q$, then from Theorem \ref{thm:BX_cty}, $Q \neq 0$. Further, $B \neq 0$: if $B = 0$, then from Propositions \ref{prop:B_0_continuous_limit} and \ref{prop:B_0_discontinuous_limit}, $z_\epsilon = \mathcal{B}_{\mathbf{X}}(Q_\epsilon)$ converges to $z$ not on the intersection of the hyperbola and rectangle, a contradiction. 
	
	Conversely, assume $z$ has that $Q_\epsilon \to Q$ where $B \neq 0$. Then, from (\ref{eqn:blue_formula_expanded}), $l_\epsilon \to 0$. From Propositions \ref{prop:l_0_discontinuous_limit} and \ref{prop:l_0_continuous_limit}, $z_\epsilon = \mathcal{B}_{\mathbf{X}}(Q_\epsilon)$ converges to $z$ on the intersection of the hyperbola and rectangle, and hence the support of the Brown measure of $X$.
	
	All that remains to show is that when $Q_\epsilon$ does not have a limit as $\epsilon \to 0^+$, then $z$ is on the intersection of the hyperbola and rectangle. 
	
	First, we consider when $Q_\epsilon$ does not stay bounded as $\epsilon \to 0^+$. Choose a sequence $\epsilon_k \to 0^+$ such that $\Abs{Q_{\epsilon_k}} \to \infty$. From Proposition \ref{prop:unbounded_Q}, $z$ is one of $\{\alpha + i \beta, \alpha + i \beta', \alpha' + i \beta, \alpha' + i \beta'\}$, which are the boundary points of the intersection of the hyperbola with the rectangle.
	
	Now, suppose that $Q_\epsilon$ remains bounded as $\epsilon \to 0^+$ but has no limit. From Propositions \ref{prop:l_0_discontinuous_limit} and \ref{prop:l_0_continuous_limit}, it suffices to show that $l_\epsilon \to 0$. For this, it suffices to show that for any sequence $\epsilon_k \to 0^+$, $B_{\epsilon_k} \not \to 0$. 
	
	For the sake of contradiction, assume that there is some $\epsilon_k \to 0^+$ where $B_{\epsilon_k} \to 0$. Passing to a subsequence, we may assume that $Q_{\epsilon_k}$ converges to $Q \neq 0$. From Proposition (\ref{prop:limit_B=0}), $Q_\epsilon$ converges to $Q$, a contradiction to the assumption that $Q_\epsilon$ had no limit.	
\end{proof}

\newpage

\printbibliography

\end{document}